\documentclass[12pt,reqno]{amsart}

\usepackage{subfiles}
\usepackage{euscript}
\usepackage{tabu}
\usepackage[margin=0.75in]{geometry}
\usepackage[dvipsnames]{xcolor} 		             		
\usepackage{graphicx}			
\usepackage{amssymb}
\usepackage{mathrsfs}
\usepackage{amsthm}
\usepackage{amsmath}
\usepackage{stmaryrd}

\usepackage{tikz}
\usepackage{tikz-cd}
\usetikzlibrary{arrows,arrows.meta}
\usetikzlibrary{arrows.meta,        
                decorations.pathmorphing,
                matrix}%
\tikzset{
  commutative diagrams/.cd, 
  arrow style=tikz, 
  diagrams={>=cm to}
}

\usepackage{accents}
\usepackage{upgreek}
\usepackage{enumerate}
\usepackage{bm}
\usepackage{mathtools}
\usepackage[all]{xy}
\usepackage{caption}

\usepackage{url}
\usepackage{float}
\usepackage{todonotes}
\usepackage{bbm}
\usepackage{colonequals}
\usepackage{longtable}

\usepackage[full]{textcomp}
\usepackage[cal=cm]{mathalfa}
\usepackage{xparse}
\usepackage{comment}
\usepackage{cite}

\usepackage[normalem]{ulem}

\theoremstyle{definition}
\newtheorem{innercustomthm}{Theorem}
\newenvironment{customthm}[1]
  {\renewcommand\theinnercustomthm{#1}\innercustomthm}
  {\endinnercustomthm}
  
\theoremstyle{definition}
\newtheorem{innercustomcor}{Corollary}
\newenvironment{customcor}[1]
  {\renewcommand\theinnercustomcor{#1}\innercustomcor}
  {\endinnercustomcor}
  
\theoremstyle{definition}

\makeatletter
\def\@tocline#1#2#3#4#5#6#7{\relax
  \ifnum #1>\c@tocdepth 
  \else
    \par \addpenalty\@secpenalty\addvspace{#2}%
    \begingroup \hyphenpenalty\@M
    \@ifempty{#4}{%
      \@tempdima\csname r@tocindent\number#1\endcsname\relax
    }{%
      \@tempdima#4\relax
    }%
    \parindent\z@ \leftskip#3\relax \advance\leftskip\@tempdima\relax
    \rightskip\@pnumwidth plus4em \parfillskip-\@pnumwidth
    #5\leavevmode\hskip-\@tempdima
      \ifcase #1
       \or\or \hskip 1em \or \hskip 2em \else \hskip 3em \fi%
      #6\nobreak\relax
    \dotfill\hbox to\@pnumwidth{\@tocpagenum{#7}}\par
    \nobreak
    \endgroup
  \fi}
\makeatother

\makeatletter
\DeclareRobustCommand{\cev}[1]{%
  \mathpalette\do@cev{#1}%
}
\newcommand{\do@cev}[2]{%
  \fix@cev{#1}{+}%
  \reflectbox{$\m@th#1\vec{\reflectbox{$\fix@cev{#1}{-}\m@th#1#2\fix@cev{#1}{+}$}}$}%
  \fix@cev{#1}{-}%
}
\newcommand{\fix@cev}[2]{%
  \ifx#1\displaystyle
    \mkern#23mu
  \else
    \ifx#1\textstyle
      \mkern#23mu
    \else
      \ifx#1\scriptstyle
        \mkern#22mu
      \else
        \mkern#22mu
      \fi
    \fi
  \fi
}

\makeatother

\usetikzlibrary{calc}
\usetikzlibrary{fadings}
\usetikzlibrary{decorations.pathmorphing}
\usetikzlibrary{decorations.pathreplacing}

\newcounter{marginnote}
\DeclareMathAlphabet{\mathpzc}{OT1}{pzc}{m}{it}

\usepackage[backref=page]{hyperref}
\hypersetup{
  colorlinks   = true,          
  urlcolor     = blue,          
  linkcolor    = violet,          
  citecolor   = blue             
}
\theoremstyle{definition}
\newtheorem{theorem}{Theorem}[section]

\newtheorem{corollary}[theorem]{Corollary}
\newtheorem{lemma}[theorem]{Lemma}
\newtheorem{proposition}[theorem]{Proposition}
\newtheorem{remark}[theorem]{Remark}

\newtheorem*{runningexample*}{Running example}

\newtheorem*{aside*}{Aside}

\newtheorem{definition}[theorem]{Definition}
\newtheorem{example}[theorem]{Example}

\newtheorem{proposition-definition}[theorem]{Proposition-Definition}

\DeclareMathOperator{\Pic}{Pic}

\newcommand{\ol}[1]{\overline{#1}}

\newcommand{\bcd}{\begin{center}\begin{tikzcd}}
\newcommand{\ecd}{\end{tikzcd}\end{center}}

\newcommand{\PP}{\mathbb{P}}

\newcommand{\Mcal}{\mathcal{M}}

\newcommand{\Mbar}{\ol{\Mcal}}

\newcommand{\M}{\mathcal{M}}

\newcommand{\Mtil}{\widetilde{\mathcal{M}}}
\newcommand{\gr}{\mathrm{gr}}

\NewDocumentCommand{\compatibilitydatum}{m m m m m m O{} O{} O{}}{
\begin{equation*} \begin{tikzcd}[ampersand replacement=\&]
  \: \arrow{r} \& {#1} \arrow{r} \arrow{d}{#7} \& {#2} \arrow{r} \arrow{d}{#8} \& {#3} \arrow{r}{[1]} \arrow{d}{#9} \& \: \\
  \: \arrow{r} \& {#4} \arrow{r} \& {#5} \arrow{r} \& {#6} \arrow{r} \& \:
\end{tikzcd} \end{equation*}}

\NewDocumentCommand{\commutingsquare}{m m m m o O{} O{} O{} O{}}{
\begin{equation}\begin{tikzcd}[ampersand replacement=\&] \label{#5}
  #1 \arrow{r}{#6} \arrow{d}{#7} \& #2 \arrow{d}{#8} \\
  #3 \arrow{r}{#9} \& #4
\end{tikzcd}\IfValueTF{#5}{\label{#5}}{} \end{equation}}

\NewDocumentCommand{\Cartesiansquare}{m m m m O{} O{} O{} O{}}{
\begin{equation*}\begin{tikzcd}[ampersand replacement=\&]
  #1 \arrow{r}{#5} \arrow{d}{#6} \arrow[dr, phantom, "\square"] \& #2 \arrow{d}{#7} \\
  #3 \arrow{r}{#8} \& #4
\end{tikzcd} \end{equation*}}

\NewDocumentCommand{\Cartesiansquarelabel}{m m m m m O{} O{} O{} O{}}{
\begin{tikzcd}[ampersand replacement=\&]
  #1 \arrow{r}{#6} \arrow{d}{#7} \arrow[dr, phantom, "\square"] \& #2 \arrow{d}{#8} \\
  #3 \arrow{r}{#9} \& #4
\end{tikzcd}\IfValueTF{#5}{\label{#5}}{}
}

\NewDocumentCommand{\triangleofspaces}{m m m O{} O{} O{}}{
\begin{tikzcd} [ampersand replacement=\&]
#1 \arrow{r}{#4} \arrow[bend right]{rr}{#5} \& #2 \arrow{r}{#6} \& #3
\end{tikzcd}}

\title{Topology of the Vakil--Zinger moduli space}

\author[T. Song]{Terry Dekun Song}\address{Department of Pure Mathematics and Mathematical Statistics, University of Cambridge, Cambridge, CB3 0WA}\email{\url{ds2016@cam.ac.uk}}

\begin{document}
\begin{abstract}
  We derive a set of generators for the rational homology of the desingularised genus one mapping space $\Mtil_{1,n}(\mathbb{P}^r,d)$ constructed by Vakil--Zinger and qualitatively describe the relations among the generators. The results build on a detailed study of the stratifications of the moduli spaces coming from tropical geometry and the constraints coming from the weight filtration on the Borel--Moore homology groups of strata, extending the techniques from the previous study on $\Mbar_{g,n}.$ Our results imply that the even homology of $\Mtil_{1,n}(\mathbb{P}^r,d)$ is tautological and controlled by genus-zero and reduced genus-one Gromov--Witten theory. We verify the Hodge and Tate conjectures for $\Mtil_{1,n}(\mathbb{P}^r,d),$ completely describe its rational Picard group, and recover known results on the vanishing of odd cohomology. Our techniques also apply to the pure weight homology groups of genus one stable maps $\Mbar_{1,n}(\PP^r,d).$
\end{abstract}
\maketitle
\setcounter{tocdepth}{1}
\tableofcontents
\section{Introduction}
Let $\mathcal{M}_{1,n}(\mathbb{P}^r,d)$ be the moduli space of degree $d$ maps from $n$-marked smooth genus one curves to projective space. The stable maps compactification $$\mathcal{M}_{1,n}(\mathbb{P}^r,d)\subset \Mbar_{1,n}(\mathbb{P}^r,d)$$ is stratified by degree-decorated dual graphs with smooth locally closed strata. However, the strata where the minimal genus one subcurve (\emph{the core}) is contracted have excess dimensions, so $\Mbar_{1,n}(\mathbb{P}^r,d)$ is reducible. The \emph{main component} $\Mbar^{\mathrm{main}}_{1,n}(\mathbb{P}^r,d)\subset \Mbar_{1,n}(\PP^r,d)$ is defined as the closure of $\mathcal{M}_{1,n}(\mathbb{P}^r,d)$ in $ \Mbar_{1,n}(\mathbb{P}^r,d).$ It is singular, and smoothability of stable maps is subtle to characterise and work with.

In \cite{vakilzinger}, Vakil and Zinger constructed a beautiful desingularsation $$\Mtil_{1,n}(\PP^r,d)\to \Mbar_{1,n}^{\mathrm{main}}(\PP^r,d).$$ The stack $\Mtil_{1,n}(\PP^r,d)$ is given a modular interpretation by Ranganathan--Santos-Parker--Wise \cite{rspw} as a moduli space of maps from \emph{centrally aligned} genus one curves: they constructed a weighted strata blow-up $\Mbar_{1,n}^{\mathrm{cen}}\to \Mbar_{1,n}$ parametrising genus one stable curves together with contractions to elliptic singularities, such that the natural map $\Mtil_{1,n}(\PP^r,d)\to \Mbar_{1,n}$ factors through.

This work describes the rational homology of $\Mtil_{1,n}(\PP^r,d)$ by a set of generators and their relations (Theorems \ref{thm:maingen}, \ref{thm:rels}) by the stratification spectral sequence, topology of linear systems on low genus curves, and how their basepoints control the relations among the homology classes.

\subsection{Central alignment stratification} The generators in the presentation of $H_\star(\Mtil_{1,n}(\PP^r,d))$ are extracted from the stratification of $\Mtil_{1,n}(\PP^r,d)$ by centrally aligned $(1,n,d)$-graphs: this was set out in \cite{rspw} and reviewed in \cite{vzdualcomplex}. 

When a genus one subgraph has positive degree, a centrally aligned graph agrees with the usual stable map dual graph, which we term as $(1,n,d)$-graphs, and so are the corresponding strata in $\Mtil_{1,n}(\PP^r,d)$. When the core has total degree zero, the $(1,n,d)$-graph--typically denoted as $\mathbf{G}$--is decorated with the data of a central alignment $\rho$ (Definition \ref{defn:centrgrph}) that blocks a subset of vertices into levels. 

In this work, we coarsen the stratification to equivalence classes of centrally aligned $(1,n,d)$-graphs under contractions and sproutings of rational tails - genus zero subgraphs that do not receive a central alignment (Definition \ref{defn:coarcl}). Analogous to the ordinary stable maps \cite[§3]{bm}, locally closed strata $\Mtil_{[\mathbf{G},\rho]}$ in $\Mtil_{1,n}(\PP^r,d)$ are the fibre products of (Lemma \ref{lem:kunneth}):
\begin{enumerate}
  \item a dual graph stratum in $\Mbar_{1,n}^{\mathrm{cen}},$ which is a torus bundle over a dual graph stratum in $\Mbar_{1,n},$
  \item a collection of pointed genus zero maps together with prescribed tangent vectors on the attaching points, such that the maps collapse the sum of the tangent vectors,
  \item genus zero stable maps.
\end{enumerate}
 

\begin{example}
  A stratum $[\mathbf{G},\rho]$ in $\Mtil_{1,3}(\PP^r,4).$ The three fibre product factors are shaded in grey, yellow, and green respectively. The dotted circle indicates genus zero stable maps, and the numbers indicate degrees.
\end{example}

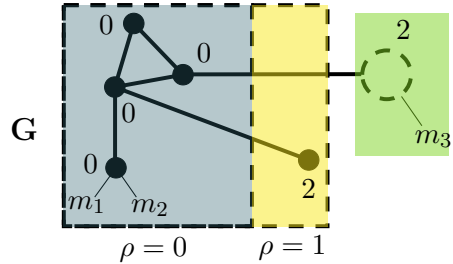
\[\begin{tikzpicture}[x=0.75pt,y=0.75pt,yscale=-1,xscale=1]

\draw [line width=1.5]    (59.11,49.27) -- (59.11,85.56) ;
\draw  [fill={rgb, 255:red, 0; green, 0; blue, 0 }  ,fill opacity=1 ] (54.41,86.7) .. controls (54.41,83.86) and (56.72,81.55) .. (59.56,81.55) .. controls (62.41,81.55) and (64.72,83.86) .. (64.72,86.7) .. controls (64.72,89.55) and (62.41,91.86) .. (59.56,91.86) .. controls (56.72,91.86) and (54.41,89.55) .. (54.41,86.7) -- cycle ;
\draw    (48.85,101.5) -- (59.56,86.7) ;
\draw    (59.56,86.7) -- (69.85,101.5) ;
\draw [line width=1.5]    (58.97,45.96) -- (156.77,82.77) ;
\draw  [fill={rgb, 255:red, 0; green, 0; blue, 0 }  ,fill opacity=1 ] (151.61,82.77) .. controls (151.61,79.92) and (153.92,77.62) .. (156.77,77.62) .. controls (159.62,77.62) and (161.93,79.92) .. (161.93,82.77) .. controls (161.93,85.62) and (159.62,87.93) .. (156.77,87.93) .. controls (153.92,87.93) and (151.61,85.62) .. (151.61,82.77) -- cycle ;
\draw [line width=1.5]    (93.38,39.56) -- (183.25,39.44) ;
\draw  [fill={rgb, 255:red, 0; green, 0; blue, 0 }  ,fill opacity=1 ] (63.38,13.86) .. controls (63.38,11.01) and (65.69,8.7) .. (68.54,8.7) .. controls (71.39,8.7) and (73.7,11.01) .. (73.7,13.86) .. controls (73.7,16.71) and (71.39,19.02) .. (68.54,19.02) .. controls (65.69,19.02) and (63.38,16.71) .. (63.38,13.86) -- cycle ;
\draw  [fill={rgb, 255:red, 0; green, 0; blue, 0 }  ,fill opacity=1 ] (53.81,45.96) .. controls (53.81,43.11) and (56.12,40.8) .. (58.97,40.8) .. controls (61.82,40.8) and (64.13,43.11) .. (64.13,45.96) .. controls (64.13,48.8) and (61.82,51.11) .. (58.97,51.11) .. controls (56.12,51.11) and (53.81,48.8) .. (53.81,45.96) -- cycle ;
\draw [line width=1.5]    (68.54,13.86) -- (58.97,45.96) ;
\draw [line width=1.5]    (68.54,13.86) -- (93.38,39.72) ;
\draw  [fill={rgb, 255:red, 0; green, 0; blue, 0 }  ,fill opacity=1 ] (88.22,39.72) .. controls (88.22,36.87) and (90.53,34.56) .. (93.38,34.56) .. controls (96.23,34.56) and (98.54,36.87) .. (98.54,39.72) .. controls (98.54,42.57) and (96.23,44.88) .. (93.38,44.88) .. controls (90.53,44.88) and (88.22,42.57) .. (88.22,39.72) -- cycle ;
\draw [line width=1.5]    (93.38,39.72) -- (58.97,45.96) ;
\draw  [dash pattern={on 5.63pt off 4.5pt}][line width=1.5]  (33.25,4.62) -- (128.15,4.62) -- (128.15,116.19) -- (33.25,116.19) -- cycle ;
\draw  [dash pattern={on 5.63pt off 4.5pt}][line width=1.5]  (33.25,4.62) -- (166.25,4.62) -- (166.25,116.19) -- (33.25,116.19) -- cycle ;
\draw  [dash pattern={on 5.63pt off 4.5pt}][line width=1.5]  (183.77,39.77) .. controls (183.77,32.98) and (189.28,27.47) .. (196.08,27.47) .. controls (202.88,27.47) and (208.39,32.98) .. (208.39,39.77) .. controls (208.39,46.57) and (202.88,52.08) .. (196.08,52.08) .. controls (189.28,52.08) and (183.77,46.57) .. (183.77,39.77) -- cycle ;
\draw    (203.56,50.7) -- (213.85,65.5) ;
\draw  [draw opacity=0][fill={rgb, 255:red, 46; green, 98; blue, 110 }  ,fill opacity=0.35 ] (33.25,4.62) -- (128.15,4.62) -- (128.15,116.19) -- (33.25,116.19) -- cycle ;
\draw  [draw opacity=0][fill={rgb, 255:red, 248; green, 231; blue, 28 }  ,fill opacity=0.5 ] (128.15,4.62) -- (166.25,4.62) -- (166.25,116.19) -- (128.15,116.19) -- cycle ;
\draw  [draw opacity=0][fill={rgb, 255:red, 126; green, 211; blue, 33 }  ,fill opacity=0.5 ] (180.44,8.52) -- (227.38,8.52) -- (227.38,80.52) -- (180.44,80.52) -- cycle ;

\draw (61.11,52.67) node [anchor=north west][inner sep=0.75pt]  [font=\small]  {$0$};
\draw (41.9,78.48) node [anchor=north west][inner sep=0.75pt]  [font=\small]  {$0$};
\draw (34,99.77) node [anchor=north west][inner sep=0.75pt]  [font=\small]  {$m_{1}$};
\draw (66,101.4) node [anchor=north west][inner sep=0.75pt]  [font=\small]  {$m_{2}$};
\draw (199.61,11.17) node [anchor=north west][inner sep=0.75pt]  [font=\small]  {$2$};
\draw (151.9,91.48) node [anchor=north west][inner sep=0.75pt]  [font=\small]  {$2$};
\draw (49.9,8.48) node [anchor=north west][inner sep=0.75pt]  [font=\small]  {$0$};
\draw (99.01,22.38) node [anchor=north west][inner sep=0.75pt]  [font=\small]  {$0$};
\draw (209,66.4) node [anchor=north west][inner sep=0.75pt]  [font=\small]  {$m_{3}$};
\draw (5,60.4) node [anchor=north west][inner sep=0.75pt]    {$\mathbf{G}$};
\draw (60,120.4) node [anchor=north west][inner sep=0.75pt]  [font=\small]  {$\rho =0$};
\draw (130.15,119.59) node [anchor=north west][inner sep=0.75pt]  [font=\small]  {$\rho =1$};
\end{tikzpicture}\]

\subsection{Weight filtrations on strata} In the following, we use \emph{Borel--Moore homology} as it receives a cycle class map from Chow homology and is functorial with respect to stratification. For smooth spaces, Borel--Moore homology groups are Poincaré dual to singular cohomology. The central alignment stratification of $\Mtil_{1,n}(\PP^r,d)$ induces a spectral sequence in Borel--Moore homology $$E^1_{p,q} = \bigoplus_{\substack{[\mathbf{G},\rho]\\ \dim \Mtil_{[\mathbf{G},\rho]} = p}} H^{\mathrm{BM}}_{p+q}(\mathcal{M}_{[\mathbf{G},\rho]})\Rightarrow H_{p+q}(\Mtil_{1,n}(\PP^r,d)).$$ The first page differentials are induced by boundary maps, generalising the familiar excision long exact sequence, and are compatible with mixed Hodge structures \cite[Lemma 3.8]{ara}. Because $\Mtil_{1,n}(\PP^r,d)$ is a smooth, proper Deligne--Mumford stack, $H_{p+q}(\Mtil_{1,n}(\PP^r,d))$ is pure of weight $-(p+q).$ Truncating the differentials of the spectral sequence $E^{k}_{p,q}\to E^{k}_{p-k,q+k-1}$ at weight $-(p+q),$ we observe that only the pure weight quotient $\mathrm{gr}^W_{-(p+q)}E^1_{p,q}$ survives to $E^{\infty}_{p,q},$ and that there are successive surjections $$\mathrm{gr}^{W}_{-(p+q)}E^1_{p,q}\twoheadrightarrow \cdots \twoheadrightarrow \mathrm{gr}^{W}_{-(p+q)} E^{k}_{p,q} \twoheadrightarrow \mathrm{gr}^{W}_{-(p+q)} E^{k+1}_{p,q} \twoheadrightarrow \cdots \twoheadrightarrow E^{\infty}_{p,q}$$ with $$\mathrm{gr}^W_{-(p+q)} E^{k+1}_{p,q} = \mathrm{gr}^W_{-(p+q)} E^{k}_{p,q}/\left(\mathrm{im}: \mathrm{gr}^W_{-(p+q)} E^{k}_{p+1,q}\to \mathrm{gr}^W_{-(p+q)} E^{k}_{p,q}\right).$$

\begin{remark}
  Dually, we may also truncate the stratification spectral sequence for the compactly supported cohomology
$$E_1^{p,q} = \bigoplus_{\substack{[\mathbf{G},\rho]\\ \dim \mathcal{M}_{[\mathbf{G},\rho]} = p}}H_c^{p+q}(\mathcal{M}_{[\mathbf{G},\rho]})\Rightarrow H^{p+q}(\Mtil_{1,n}(\PP^r,d))$$ 
and have successive injections $$E_{\infty}^{p,q}\subset \cdots\subset \mathrm{gr}^{W}_{p+q} E_{k+1}^{p,q}\subset \mathrm{gr}^{W}_{p+q} E_{k}^{p,q}\subset\cdots \subset\mathrm{gr}^{W}_{p+q}E_1^{p,q}$$ given by $$\mathrm{gr}^{W}_{p+q} E_{k+1}^{p,q} = \ker\left( \mathrm{gr}^{W}_{p+q} E_{k}^{p,q}\to \mathrm{gr}^{W}_{p+q} E_{k}^{p+k, q-k+1}\right)\subset \mathrm{gr}^{W}_{p+q} E_{k}^{p,q}.$$ The surjectivity (resp. injectivity) follows from the weight bounds on $H_k^{\mathrm{BM}}$ and $H_c^k.$
\end{remark}

Thus, the pure weight Borel--Moore homology groups $$\mathrm{gr}^W_{-\star }H_{\star}^{\mathrm{BM}}(\mathcal{M}_{[\mathbf{G},\rho]}) = \bigoplus_{k} \mathrm{gr}^W_{-k }H_{k}^{\mathrm{BM}}(\mathcal{M}_{[\mathbf{G},\rho]})$$ generate $H_\star(\Mtil_{1,n}(\PP^r,d)),$ and the relations among the generators are subquotients of the `off-by-one' Borel--Moore homology groups $\mathrm{gr}^W_{-\star+1} H_\star^{\mathrm{BM}}(\mathcal{M}_{[\mathbf{G},\rho]}).$ The fibre product descriptions of the strata $\Mtil_{[\mathbf{G},\rho]}$ lead to Künneth formulas for their Borel--More homology groups (Lemma \ref{lem:kunneth}), so the pure and off-by-one weight graded pieces are determined by the pure weight classes on each fibre product factor, which we now discuss.



\subsection{Generators and relations}\label{sec:introgenrel}

We approach the weight pieces of $\mathcal{M}_{1,n}(\PP^r,d)$ via its partial compactification by linear systems. Let $$\mathrm{Pic}^d_{1,n}\to \mathcal{M}_{1,n}$$ be the relative degree-$d$ (rigidified) Picard stack over $\mathcal{M}_{1,n}.$ It has relative polarisation $$\Theta\in A^1(\mathrm{Pic}^d_{1,n})$$ and universal Poincaré line bundle $\mathcal{P}_d$ over $\mathrm{Pic}^d_{1,n}\times_{\mathcal{M}_{1,n}} \mathcal{C}_{1,n}.$ Let $\pi: \mathrm{Pic}^d_{1,n}\times_{\mathcal{M}_{1,n}} \mathcal{C}_{1,n}\to \mathrm{Pic}^d_{1,n}$ be the projection map, we define $$\hat{\mathcal{Q}}_{1,n}(\PP^r,d):=\mathbb{P}_{\mathrm{Pic}^d_{1,n}}(R\pi_* \mathcal{P}_d^{\oplus r+1}).$$ It is a projective bundle over $\mathrm{Pic}^d_{1,n},$ such that the fibre over $(E,L)$ is $$\mathbb{P}(H^0(E, L)^{\oplus r+1}).$$ It has hyperplane class $$H_{\hat{\mathcal{Q}}}\in A^1(\hat{\mathcal{Q}}_{1,n}(\PP^r,d)).$$ There is an open immersion $$\mathcal{M}_{1,n}(\PP^r,d)\subset \hat{\mathcal{Q}}_{1,n}(\PP^r,d)$$ given by tuples of sections with no common basepoints.

The classes $\Theta, H_{\hat{\mathcal{Q}}}$ produce Borel--Moore homology classes under the cycle class map. Together with flat pullback from $W_{-\star} H_\star^{\mathrm{BM}}(\mathcal{M}_{1,n}),$ they generate $W_{-\star} H_\star^{\mathrm{BM}}(\mathcal{M}_{1,n}(\PP^r,d)),$ and an analogous description holds for its off-by-one weight cohomology (Lemmas \ref{lem:m1nrd}, \ref{lem:m1nrdoff}). The same discussion applies to maps from nodal elliptic curves (Lemma \ref{cor:Mcycle}).

We turn to second fibre product factor, which are genus zero maps together with tangent vectors. Let $\mathsf{M}_{0,1}$ and $\mathsf{M}_{0,2}$ be the Artin stacks of smooth rational curves with one and two marked points, respectively. They are identified with classifying stacks of automorphisms of $\mathbb{P}^1$ that preserves the set of marked points. In particular, there are isomorphisms $$H^\star(\mathsf{M}_{0,i})\cong \mathbb{Q}[\psi]$$ for $i=1,2,$ where $$\psi\in H^2(\mathsf{M}_{0,i})$$ is the $\psi$-class of any marked point and is well-defined up to a sign. These $\psi$-classes pull back to classes supported on the genus zero vertices with factorisation property that are univalent or bivalent, and these are their only contribution to pure weight Borel--Moore homology (Lemma \ref{lem:MF}). To see this, we first determine the cohomology of parametrised genus zero maps satisfying the factorisation property: \begin{enumerate}
  \item Given a tangent vector on the distinguished point $p\in \mathbb{P}^r,$ say $v\in T_p \mathbb{P}^r$, we define the space of pointed maps with derivative $v$ and determine its cohomology in Section \ref{sec:prestan}. Analogous to the results in \cite{farbwolf} concering the moduli space $\mathrm{Map}^*_d(\mathbb{P}^1, \mathbb{P}^r)$ of pointed, parametrized maps from $\PP^1$ to $\PP^r$, they are largely independent of the degree of the map, in which case their cohomology is isomorphic to that of $\mathbb{A}^r\setminus \mathrm{pt}.$
  \item We describe the space of admissible tangent vectors: they are tuples of tangent vectors in $T_p \mathbb{P}^r$ satisfying suitable linear dependencies.
\end{enumerate}

Combining information on the base and fibre directions, Corollary \ref{cor:genDF} states that the pure weight cohomology of the space of \emph{parametrised} maps with factorisation property is given only by $H^0\cong \mathbb{Q}.$ The only source of pure weight classes from the moduli space is hence the $\psi$-classes from taking quotients by the automorphism groups.

Finally, genus zero stable maps and their pointed analogues $$\Mbar_{0,n}^*(\PP^r,d):=\mathrm{ev}_{n+1}^{-1}(p)\subset \Mbar_{0,n+1}(\PP^r,d)$$ are smooth and proper. Their homology groups are known to be generated by tautological cycles \cite{opretorus}, a complete set of relations have been derived via wall-crossing \cite{mmchow,mminter,mmtaut}, and their Betti numbers have been determined \cite{getzpand}.

\begin{customthm}{A}[Corollary \ref{cor:maingen}]\label{thm:maingen}
  Let $[\mathbf{G},\rho]$ label a boundary stratum closure $\Mbar_{[\mathbf{G},\rho]}\xhookrightarrow{\iota_{[\mathbf{G},\rho]}} \Mtil_{1,n}(\mathbb{P}^r,d)$ with interior $\Mtil_{[\mathbf{G},\rho]},$ let $F\in H_{\star}^{\mathrm{BM}}(\Mtil_{[\mathbf{G},\rho]})$ be a polynomial of the following classes:
  \begin{enumerate}
    \item pullback from \(W_{-\star} H_{\star}^{\mathrm{BM}}(\mathcal{M}_{1,m})\) for \(m\leq n+d\),
    \item $\Theta\in A^1(\mathrm{Pic}^{d'}_{1,m})$ and $H_{\hat{\mathcal{Q}}}\in A^1(\mathcal{M}_{1,n'}(\mathbb{P}^r,d'))$ defined above,
    \item pullback from $H^\star(\mathsf{M}_{0,1})$ and $H^\star(\mathsf{M}_{0,2})$, each generated by their $\psi$-classes,
    \item pullback from rational tails, which are isomorphic to $\Mbar_{0,n'}(\mathbb{P}^r,d'),$
  \end{enumerate}
  and let $\tilde{F}\in H_{\star}^{\mathrm{BM}}(\Mbar_{[\mathbf{G},\rho]})$ be a lift\footnote{Given open embedding $U\subset X$ with $U, X$ both smooth, the restriction map on pure weight Borel--Moore homology $\mathrm{gr}^W_{-\star} H^{\mathrm{BM}}_\star(X)\to \mathrm{gr}^W_{-\star} H^{\mathrm{BM}}_\star(U)$ is surjective. We discuss choices of lifts in §\ref{subsec:strcl}.} of the class $F.$ The homology group $H_\star(\Mtil_{1,n}(\mathbb{P}^r,d))$ is additively generated by classes of the form $(\iota_{[\mathbf{G},\rho]})_* \tilde{F}.$ 
\end{customthm}

\begin{remark}
  The same technique shows that   can be applied to the pure weight homology of Kontsevich stable maps $\mathrm{gr}^W_{-\star} H_\star(\Mbar_{1,n}(\PP^r,d)),$ which is generated by items (1), (2), (3) listed above.
\end{remark}

The precise description of the generators is more refined than the statement above: see Corollary \ref{cor:maingen} for more details. It is helpful to present the generators as centrally aligned dual graphs with vertex or level decorations by Borel--Moore homology classes.

The relations among the generators are controlled by the off-by-one weight graded pieces of the strata. A qualitative summary of the relations is as follows. 

\begin{customthm}{B}\label{thm:rels}
  The relations among the generators are multiplicatively generated by:
  \begin{enumerate}
    \item relations pulled back from $\overline{\mathcal{M}}_{1,n}^{\mathrm{cen}},$ which are generated by WDVV and Getzler's relations as well as $\psi$-class relations coming from the central alignment\footnote{See §\ref{subsubsec:quotaut} for the definition of the relevant $\psi$-classes and the relation.},
    \item basepoint relations, which we define and compute in §\ref{subsec:basptrelcomp} and \ref{subsec:baseptrel}.
  \end{enumerate}
\end{customthm}

See Theorem \ref{thm:relations} and definitions in that section for a more precise statement. The term `generate' is intended as taking pulled back relations from each stratum closure and pushing forward to $\Mtil_{1,n}(\PP^r,d)$. The linear span of all such relations hence depends on the complexity of the dual graph stratification of $\Mtil_{1,n}(\PP^r,d).$

The basepoint relations concern pure weight classes on positive degree cores of the strata, which correspond to items (1) and (2) of the generators in Theorem \ref{thm:maingen}. The relations are derived by considering maps of pairs\footnote{Here, $[\mathbb{G},\rho]$ is a (centrally aligned) $(1,n,d)$-graph consisting of a genus one vertex and a rational tail or a level.} $$(\mathcal{M}_{1,n}(\PP^r,d)\cup \Mtil_{[\mathbb{G},\rho]}, \Mtil_{[\mathbb{G},\rho]})\to (\hat{\mathcal{Q}}_{1,n}(\PP^r,d), \hat{\mathcal{Q}}_{1,n}(\PP^r,d)\setminus \mathcal{M}_{1,n}(\PP^r,d)),$$ that `collapses' rational tails (genus zero stable maps) to basepoints. This allows us to relate the $E^1$-page differentials on the left hand side to the boundary maps on the right hand side. This is analogous to maps of pairs $(\mathrm{Bl}_Z X, E)\to (X,Z),$ where $E\subset \mathrm{Bl}_Z X$ is the exceptional divisor: the boundary map $\gr^W_{-k}H^{\mathrm{BM}}_{k-1}(X\setminus Z)\to \gr^W_{-k} H^{\mathrm{BM}}_{k}(E)$ is completely determined by the map $\gr^W_{-k}H^{\mathrm{BM}}_{k-1}(X\setminus Z)\to \gr^W_{-k} H^{\mathrm{BM}}_{k}(Z).$

Because the basepoint quasimaps have codimension $r,$ the basepoint relations are in homological codimension no less than $2r.$ This implies that in low codimension, all the relations are `tautological' in the following sense.

\begin{customcor}{C}
  The relations among the generators in $H_{k}^{\mathrm{BM}}(\Mtil_{1,n}(\PP^r,d))$ for $k<2r$ are multiplicatively generated by relations in $H_{\star}(\Mbar_{1,n}^{\mathrm{cen}}).$
\end{customcor}

\begin{remark}
  Let $\mathfrak{M}_{1,n}^{\mathrm{cen}}$ be the moduli stack of prestable centrally aligned curves, and let $\mathfrak{Pic}_{{\mathfrak{M}}_{1,n}^{\mathrm{cen}}}^d\to \mathfrak{M}_{1,n}^{\mathrm{cen}}$ be the relative degree-$d$ Picard stackover it. One could consider the natural map $$\Mtil_{1,n}(\PP^r,d)\to \mathfrak{Pic}_{{\mathfrak{M}}_{1,n}^{\mathrm{cen}}}^d,$$ stratify the image of the morphism $\mathfrak{Pic}_{{\mathfrak{M}}_{1,n}^{\mathrm{cen}}}^{d, \mathrm{VZ}}\subset \mathfrak{Pic}_{{\mathfrak{M}}_{1,n}^{\mathrm{cen}}}^d$ to determine the weight graded pieces of $H_\star^{\mathrm{BM}}(\mathfrak{Pic}_{{\mathfrak{M}}_{1,n}^{\mathrm{cen}}}^{d, \mathrm{VZ}})$ and then present $\Mtil_{1,n}(\PP^r,d)$ as the basepoint-free locus in the quasimaps moduli space, which is a projective bundle over $\mathfrak{Pic}_{{\mathfrak{M}}_{1,n}^{\mathrm{cen}}}^{d, \mathrm{VZ}}.$ This gives an alternative way to derive the above generators and relations of $H_\star(\Mtil_{1,n}(\PP^r,d)).$
\end{remark}

\subsection{Applications}
The even homology of $\Mbar_{1,n}$ is spanned by boundary classes \cite{pet}. Combining this with the description of generators given above, the even homology of $\Mtil_{1,n}(\mathbb{P}^r,d)$ is generated by tautological classes and boundary strata. This allows us to deduce

\begin{customcor}{D}[Corollary \ref{cor:HT}]
  The Hodge and Tate conjectures hold for $\Mtil_{1,n}(\PP^r,d).$
\end{customcor}

From another perspective, Theorem \ref{thm:maingen} implies that the even homology groups of $\Mtil_{1,n}(\PP^r,d)$ are entirely controlled by intersection numbers of tautological cycles, namely the reduced genus one Gromov--Witten theory of $\PP^r.$ 

Via Poincaré duality, Theorem \ref{thm:maingen} also describes the Hodge structures in $H^\star(\Mtil_{1,n}(\mathbb{P}^r,d)).$ Following \cite{clp}, we denote $$\mathsf{L}:=H^2(\mathbb{P}^1),\mathsf{S}_{k+1}:=W_k H^k(\mathcal{M}_{1,k}).$$ The Eichler--Shimura isomorphism states that $\mathsf{S}_{k+1}$ corresponds to $\mathrm{SL}_2(\mathbb{Z})$-cusp forms of weight $k+1.$ They are possibly non-zero for odd $k\geq 11.$ The Hodge structures present in $H^\star(\Mtil_{1,n}(\mathbb{P}^r,d))$ are products of $\mathsf{L}$ and $\mathsf{S}_{k+1}$ for $k\geq n+d.$ In particular, the odd weight Hodge structures in $H^\star(\Mtil_{1,n}(\mathbb{P}^r,d))$ have weight at least $11.$ Since $\Mtil_{1,n}(\mathbb{P}^r,d)$ has pure weight cohomology, we recover the following result of Fontanari \cite{Fontanari}.

\begin{customcor}{E}[Corollary \ref{cor:oddvan}]
  For odd $k< 11$ and all $n,$ $H^k(\Mtil_{1,n}(\mathbb{P}^r,d))$ vanishes.
\end{customcor}

The presence of Hodge structures $\mathsf{L}^j \prod_{k}\mathsf{S}_{k+1}^{\ell_k}$ is tied to the graphical stratification of $\Mtil_{1,n}(\mathbb{P}^r,d)$. This allows more precise descriptions of what each Hodge structure contributes.

\begin{example}[§\ref{sec:degodd}]
  For $r\geq 11,$ the lowest $d$ such that $H^\star(\Mtil_{1,0}(\mathbb{P}^r,d))$ has non-vanishing odd cohomology is $d=11.$ In this case, $H^{123}(\Mtil_{1,0}(\mathbb{P}^r,11))\neq 0$ is the first non-vanishing odd cohomological degree. The lowest odd cohomological degree $k$ with  $H^k(\Mtil_{1,0}(\mathbb{P}^r,d))$ non-vanishing is realised by $k = 13$ when $d \geq 66.$
\end{example}

Using the fact that $\mathrm{gr}^W_2 H^1(\mathcal{M}_{1,n}(\PP^r,d))$ vanishes, we determine the rational Picard group of $\Mtil_{1,n}(\PP^r,d)$ as follows.

\begin{customcor}{F}[Proposition \ref{prop:Pic}]\label{cor:F}
 The second cohomology group $H^2(\Mtil_{1,n}(\mathbb{P}^r,d))$ has a basis given by $\Theta, H_{\hat{\mathcal{Q}}},$ and boundary divisors. The cycle class map $$A^1(\Mtil_{1,n}(\mathbb{P}^r,d))_{\mathbb{Q}}\to H^2(\Mtil_{1,n}(\mathbb{P}^r,d))$$ is an isomorphism.
\end{customcor}

It is natural to expect the basepoint relations in even homological degrees to admit lifts in the Chow group of $\Mtil_{1,n}(\PP^r,d).$ In that case, the cycle class map $A_{i}(\Mtil_{1,n}(\mathbb{P}^r,d))_{\mathbb{Q}}\to H^{2i}(\Mtil_{1,n}(\mathbb{P}^r,d))$ is an ismorphism for all $i>0.$

\subsection{Related work}
The strategy of understanding weight filtrations within a stratified space is far from being new. The ideas can be traced back to the earlier work of Arbarello--Cornalba \cite{ac} on $\Mbar_{g,n}$ and have been utilised to great effect by Petersen \cite{pet} on $\Mbar_{1,n}$ with its dual graph stratification. They also underly the recent breakthroughs on cohomology of $\M_{g,n}$ and $\Mbar_{g,n}$ by Bergström--Faber--Payne \cite{bfp}, Canning--Larson \cite{clchow}, Canning--Larson--Payne \cite{clp, clp11}, and Canning--Larson--Payne--Willwacher \cite{clpw, clpw2, clpw25}.

Unlike the moduli of stable curves, the strata in the mapping spaces are in general non-trivial fibre products. This has been investigated by Oprea in the setting of genus zero stable maps to flag varieties \cite{opreaflag}.

The work combines the above perspectives and techniques with the recent study of combinatorics of genus one mapping spaces: this includes the plethystic structure involving stable maps \cite{genusonechar} and the explicit genus zero geometry of rational functions satisfying factorisation property \cite{vzdualcomplex} that is inspired by previous work of Farb--Wolfson \cite{farbwolf} on parametrised pointed genus zero maps.

Another important class of compactified mapping spaces is the moduli space of genus one stable quotients \cite{mop} $\overline{\mathcal{Q}}_{1,n}(\PP^r,d).$ In particular, as $\overline{\mathcal{Q}}_{1,0}(\PP^r,d)$ is smooth and proper, we expect that the technique and calculations in the present work are applicable. They will complement previous work of Y. Cooper \cite{Cooper2014} on its geometric and topological properties.

\subsection{Future directions}
Ongoing work of the author \cite{relations} returns to the generators and relations of $H^*(\Mbar_{0,n}(\mathbb{P}^r,d))$ using the dual graph stratification. The results obtained from this will complement the ones given by Mustata--Mustata \cite{mmchow,mminter,mmtaut} using wall-crossing techniques; comparing the two seems non-obvious and is worth pursuing.

Regarding the interior $\mathcal{M}_{1,n}(\mathbb{P}^r,d)$, the Betti numbers as well as the full weight filtration remain unknown: this work concerns the pure and off-by-one pieces, and the top weight cohomology is known to vanish from \cite{vzdualcomplex} using boundary complex techniques. On the other hand, results about the cohomology of $\Mtil_{1,n}(\PP^r,d)$ and their strata closures in this work could shed light on the graph complex that calculates the other weight graded pieces $\mathrm{gr}^W_k H_c^\star(\mathcal{M}_{1,n}(\PP^r,d)).$ The weight graded Euler characteristics of $\mathcal{M}_{1,n}(\PP^r,d )$ has been determined by the recent Serre characteristics calculations in \cite{virtualhodge}, which gives evidence towards degeree stabilisation in a range of cohomological degrees.

While the present results do not immediately yield closed formulas for the Poincaré polynomials of $\Mtil_{1,n}(\PP^r,d)$, ongoing work with S. Kannan aims to calculate their $S_n$-equivariant Grothendieck ring classes that specialises to Poincaré polynomials. 

Battistella--Carocci \cite{battistellacarocci} constructed a normal crossings compactification of $\mathcal{M}_{2,n}(\mathbb{P}^r,d)$ by combining the perspective of \cite{rspw} with the theory of genus two Gorenstein singularities and admissible covers. The techniques of this work, combined with Petersen's previous work on the cohomology of $\mathcal{M}_{2,n}$ \cite{PetA2, PetM2ct}, should lead to a description of the cohomology of their construction.

\subsection*{Acknowledgement} The author is grateful to his supervisor Dhruv Ranganathan for enlightening discussions and thanks Samir Canning and Siddarth Kannan for valuable comments on previous drafts. Thanks are also due to Alexis Aumonier, Sam Payne, Michele Pernice, Oscar Randal-Williams, and Ravi Vakil for helpful conversations and encouragement.

This work is supported by a Cambridge Trust international scholarship.

\subsection*{Convention} All (co)homology groups are taken with $\mathbb{Q}$-local systems and carry mixed Hodge structures unless stated otherwise. `Smooth' means smooth as a Deligne--Mumford stack, of which the coarse moduli space may acquire singularities. For $M$ a smooth DM stack of dimension $m,$ we use $[M]\in H_{2m}^{\mathrm{BM}}(M)$ to denote its fundamental class.

\section{Stratification of mapping spaces}\label{sec:graph}
\subsection{Dual graphs}
We recall the decorated dual graph stratification on \(\Mbar_{g,n}(\mathbb{P}^r,d)\).

\begin{definition}\cite[Definition 1.1]{vzdualcomplex}
  Let $g, n, d \geq 0$. 
  
  A $(g, n, d)$-graph is a tuple $\mathbf{G} = (G, w, \delta, m)$ where:
\begin{enumerate}
\item $G$ is a connected graph;
\item $w : V(\mathbf{G}) \to \mathbb{Z}_{\geq 0}$ is called the \textit{genus function};
\item $\delta: V(\mathbf{G}) \to \mathbb{Z}_{\geq 0}$ is called the \textit{degree function};
\item $m: \{1, \ldots, n\} \to V(\mathbf{G})$ is called the \textit{marking function}. 
\end{enumerate}
They are required to satisfy:
\begin{enumerate}[(1)]
\item $\dim_{\mathbb{Q}} H_1(G, \mathbb{Q}) + \sum_{v \in V(\mathbf{G})} w(v) = g$;
\item $\sum_{v \in V(\mathbf{G})} \delta(v) = d$.
\end{enumerate}
A $(g, n, d)$-graph is called \textit{stable} if for all vertices $v \in V(\mathbf{G})$ with $\delta(v) = 0$, we have
\[ 2 w(v) - 2 + \mathrm{val}(v) + |m^{-1}(v)| > 0, \]
where $\mathrm{val}(v)$ means the graph valence of $v$.

The automorphism group of a $(g,n,d)$-graph is a graph automorphism that commutes with the genus, degree, and marking decoration functions.

Given a $(g,n,d)$-graph $\mathbf{G}$ and an edge $e\in E(\mathbf{G}),$ we define $\mathbf{G}/e$ as a $(g,n,d)$-graph whose underlying graph is the quotient graph $\mathbf{G}/e,$ and whose decorations are defined as: \begin{enumerate}
    \item If $e$ is not a loop, the two vertices $v_1$ and $v_2$ incident to $e$ are combined into a new vertex $v'$. Then $$\delta(v') = \delta(v_1) + \delta(v_2), w(v') = w(v_1) + w(v_2),$$ and $m^{-1}(v') = m^{-1}(v_1) \cup m^{-1}(v_2)$.
    \item If $e$ is a loop, we increase $w$ on the vertex supporting $e$ by $1.$ The degree and marking functions are retained.
\end{enumerate}
The resulting $(g,n,d)$-graph $\mathbf{G}/e$ is denoted as the edge contraction of $\mathbf{G}$ by the edge $e.$

The set of $(g,n,d)$-graphs forms a category with morphisms generated by isomorphisms and edge contractions.
\end{definition}

\begin{definition}
  Let $\mathbf{G}$ be a $(g,n,d)$-graph. Given a (naive) subgraph $\mathbf{G}'$ of $\mathbf{G}$, its enhancement $\hat{\mathbf{G}}'$ is the dual graph with genus and marking decorations restricted from that of $\mathbf{G},$ together with additional legs in bijection to edges connecting vertices in $\mathbf{G}'$ to vertices outside $\mathbf{G}'.$ We denote the set of additional legs as $L'(\mathbf{G}').$
\end{definition}

In the following, `subgraphs' will always mean their enhancements.

\begin{definition}
  Let $\mathbf{G}$ be a $(1,n,d)$-graph. The \emph{core} of $\mathbf{G},$ denoted as  $\mathbf{G}^{c},$ is the minimal subgraph of $\mathbf{G}$ that has genus one.

  There is a poset structure on $V(\mathbf{G})\setminus V(\mathbf{G}^c)$ where $f_1\leq f_2$ if the unique path\footnote{Notice that contracting the core of $\mathbf{G}$ gives a rooted tree, so the path is indeed unique.} from $f_2$ to $\mathbf{G}^{c}$ passes through $f_1$. We extend the poset structure to $V(\mathbf{G})$ by declaring that all elements in $V(\mathbf{G}^c)$ are tied as the unique minimum.
\end{definition}

Centrally aligned dual graphs are $(1,n,d)$-graphs with a partially defined ordering on their vertices. The locally closed strata of the normal crossings boundary $\Mtil_{1,n}(\mathbb{P}^r,d)\setminus \mathcal{M}_{1,n}(\mathbb{P}^r,d)$ are labeled by centrally aligned $(1,n,d)$-graphs. The definitions below follow \cite{rspw} and are adapted from the exposition in \cite{vzdualcomplex}, which primarily focuses on the closely related and finer stratification of \emph{radially} aligned $(1,n,d)$-graphs.

\begin{definition}\label{defn:centrgrph}
  A centrally aligned $(1,n,d)$-graph is a $(1,n,d)$-graph that has a core with positive total degree, or a $(1,n,d)$-graph $\mathbf{G}$ where the core has total degree zero, and the data of a connected subgraph $\mathbf{G}_\rho$ of $\mathbf{G}$ with a level map $\rho: V(\mathbf{G}_\rho)\to \{0,\dots, |\rho|\}$. The pair $(\mathbf{G}_\rho, \rho)$ are required to satisfy the following conditions:
\begin{enumerate}
  \item the subgraph $\mathbf{G}_\rho$ contains the core,
  \item the pre-image $\rho^{-1}(0)$ is equal to the core,
  \item in $V(\mathbf{G}_\rho),$ only $\rho^{-1}(|\rho|)$ has positive degree vertices, and their degrees sum to greater than one,
  \item the map $\rho$ is a map of posets, where the poset structure on $V(\mathbf{G}_\rho)$ is restricted from that of $V(\mathbf{G}).$
\end{enumerate}
  
  An isomorphism of centrally aligned $(1,n,d)$-graphs is an isomorphism of the underlying $(1,n,d)$-graphs that commutes with the alignment functions.
\end{definition}

\begin{remark}
  The central alignment defined above is the same as \cite[§4.6]{rspw} with $\delta$ equal to the contraction radius of the $(1,n,d)$-dual graph.
\end{remark}

\begin{example}
  The figure below shows a centrally aligned $(1,n,d)$-graph with $n=5$ and $d=4.$ The core has total degree zero, and the only the contraction radius vertices are assigned the alignment function $\rho$.
\end{example}

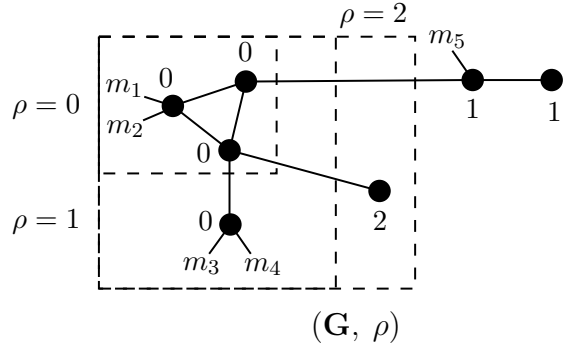
\begin{figure}[H]
  \centering
  \tikzset{every picture/.style={line width=0.75pt}} 

\begin{tikzpicture}[x=0.75pt,y=0.75pt,yscale=-1,xscale=1]

\draw    (125.11,82.27) -- (125.11,118.56) ;
\draw  [fill={rgb, 255:red, 0; green, 0; blue, 0 }  ,fill opacity=1 ] (120.41,119.7) .. controls (120.41,116.86) and (122.72,114.55) .. (125.56,114.55) .. controls (128.41,114.55) and (130.72,116.86) .. (130.72,119.7) .. controls (130.72,122.55) and (128.41,124.86) .. (125.56,124.86) .. controls (122.72,124.86) and (120.41,122.55) .. (120.41,119.7) -- cycle ;
\draw    (96.54,60.7) -- (81.52,67.16) ;
\draw    (96.54,59.86) -- (82.1,55.22) ;
\draw    (114.85,134.5) -- (125.56,119.7) ;
\draw    (125.56,119.7) -- (135.85,134.5) ;
\draw    (125.11,82.27) -- (200.77,102.77) ;
\draw  [fill={rgb, 255:red, 0; green, 0; blue, 0 }  ,fill opacity=1 ] (195.61,102.77) .. controls (195.61,99.92) and (197.92,97.62) .. (200.77,97.62) .. controls (203.62,97.62) and (205.93,99.92) .. (205.93,102.77) .. controls (205.93,105.62) and (203.62,107.93) .. (200.77,107.93) .. controls (197.92,107.93) and (195.61,105.62) .. (195.61,102.77) -- cycle ;
\draw    (133.38,47.56) -- (247.77,46.77) ;
\draw  [fill={rgb, 255:red, 0; green, 0; blue, 0 }  ,fill opacity=1 ] (242.61,46.77) .. controls (242.61,43.92) and (244.92,41.62) .. (247.77,41.62) .. controls (250.62,41.62) and (252.93,43.92) .. (252.93,46.77) .. controls (252.93,49.62) and (250.62,51.93) .. (247.77,51.93) .. controls (244.92,51.93) and (242.61,49.62) .. (242.61,46.77) -- cycle ;
\draw    (248.68,46.8) -- (287.59,46.8) ;
\draw  [fill={rgb, 255:red, 0; green, 0; blue, 0 }  ,fill opacity=1 ] (282.43,46.8) .. controls (282.43,43.95) and (284.74,41.64) .. (287.59,41.64) .. controls (290.44,41.64) and (292.75,43.95) .. (292.75,46.8) .. controls (292.75,49.65) and (290.44,51.96) .. (287.59,51.96) .. controls (284.74,51.96) and (282.43,49.65) .. (282.43,46.8) -- cycle ;
\draw    (237.49,31.98) -- (247.77,46.77) ;
\draw  [fill={rgb, 255:red, 0; green, 0; blue, 0 }  ,fill opacity=1 ] (91.38,59.86) .. controls (91.38,57.01) and (93.69,54.7) .. (96.54,54.7) .. controls (99.39,54.7) and (101.7,57.01) .. (101.7,59.86) .. controls (101.7,62.71) and (99.39,65.02) .. (96.54,65.02) .. controls (93.69,65.02) and (91.38,62.71) .. (91.38,59.86) -- cycle ;
\draw  [fill={rgb, 255:red, 0; green, 0; blue, 0 }  ,fill opacity=1 ] (119.95,82.27) .. controls (119.95,79.42) and (122.26,77.11) .. (125.11,77.11) .. controls (127.96,77.11) and (130.27,79.42) .. (130.27,82.27) .. controls (130.27,85.12) and (127.96,87.43) .. (125.11,87.43) .. controls (122.26,87.43) and (119.95,85.12) .. (119.95,82.27) -- cycle ;
\draw    (96.54,59.86) -- (125.11,82.27) ;
\draw    (96.54,59.86) -- (133.38,47.56) ;
\draw  [fill={rgb, 255:red, 0; green, 0; blue, 0 }  ,fill opacity=1 ] (128.22,47.56) .. controls (128.22,44.71) and (130.53,42.4) .. (133.38,42.4) .. controls (136.23,42.4) and (138.54,44.71) .. (138.54,47.56) .. controls (138.54,50.41) and (136.23,52.72) .. (133.38,52.72) .. controls (130.53,52.72) and (128.22,50.41) .. (128.22,47.56) -- cycle ;
\draw    (133.38,47.56) -- (125.11,82.27) ;
\draw  [dash pattern={on 4.5pt off 4.5pt}] (59.24,24.92) -- (148.87,24.92) -- (148.87,93.93) -- (59.24,93.93) -- cycle ;
\draw  [dash pattern={on 4.5pt off 4.5pt}] (59.24,24.92) -- (178.73,24.92) -- (178.73,152.07) -- (59.24,152.07) -- cycle ;
\draw  [dash pattern={on 4.5pt off 4.5pt}] (59.24,24.92) -- (218.73,24.92) -- (218.73,152.07) -- (59.24,152.07) -- cycle ;

\draw (61,45.26) node [anchor=north west][inner sep=0.75pt]  [font=\small]  {$m_{1}$};
\draw (61.56,65.5) node [anchor=north west][inner sep=0.75pt]  [font=\small]  {$m_{2}$};
\draw (107.01,77.38) node [anchor=north west][inner sep=0.75pt]  [font=\small]  {$0$};
\draw (107.9,111.48) node [anchor=north west][inner sep=0.75pt]  [font=\small]  {$0$};
\draw (100,134.77) node [anchor=north west][inner sep=0.75pt]  [font=\small]  {$m_{3}$};
\draw (131,135.4) node [anchor=north west][inner sep=0.75pt]  [font=\small]  {$m_{4}$};
\draw (242.61,55.17) node [anchor=north west][inner sep=0.75pt]  [font=\small]  {$1$};
\draw (195.9,111.48) node [anchor=north west][inner sep=0.75pt]  [font=\small]  {$2$};
\draw (283.9,56.48) node [anchor=north west][inner sep=0.75pt]  [font=\small]  {$1$};
\draw (224,20.4) node [anchor=north west][inner sep=0.75pt]  [font=\small]  {$m_{5}$};
\draw (164.68,162.21) node [anchor=north west][inner sep=0.75pt]    {$({\mathbf{G}} ,\ \rho )$};
\draw (87.9,39.48) node [anchor=north west][inner sep=0.75pt]  [font=\small]  {$0$};
\draw (128.01,26.38) node [anchor=north west][inner sep=0.75pt]  [font=\small]  {$0$};
\draw (14,52.4) node [anchor=north west][inner sep=0.75pt]  [font=\small]  {$\rho =0$};
\draw (14,110.4) node [anchor=north west][inner sep=0.75pt]  [font=\small]  {$\rho =1$};
\draw (179,6.4) node [anchor=north west][inner sep=0.75pt]  [font=\small]  {$\rho =2$};
\end{tikzpicture}
\end{figure}

We require edge contractions among centrally aligned $(1,n,d)$-graphs to be compatible with their levels.

\begin{definition}
 Let $(\mathbf{G},\rho)$ be a centrally aligned $(1,n,d)$-graph. Given an integer $i \in \{1, \ldots, |\rho|\}$, we define the \textit{radial merge} of $(\mathbf{G}, \rho)$ along $i,$ denoted as $(\mathbf{G}_{\setminus i}, \rho_{\setminus i})$ as follows:
\begin{enumerate}
  \item post-compose $\rho$ with the surjection $\{0, \ldots, k\} \to \{0, \ldots, k - 1\}$ which decreases all $j \geq i$ by $1$;
  \item whenever $v, w \in V(\mathbf{G})$ with $v \in f^{-1}(i - 1)$ and $w \in f^{-1}(i)$, such that there is an edge $e$ between $v$ and $w$, perform the edge contraction of $e$ as for $(1,n,d)$-graphs.
\end{enumerate}

The set of centrally aligned $(1,n,d)$-graphs forms a category in which the morphisms are generated by isomorphisms, radial merges, and contractions of edges in $(\mathbf{G},\rho)$ outside of $\mathbf{G}_{\rho}$ as well as edges connecting vertices in $\rho^{-1}(0).$
\end{definition}

\begin{definition}\label{defn:coarep}
  The rational tail subgraphs of a centrally aligned $(1,n,d)$-graph $(\mathbf{G}, \rho)$ are the maximal connected subgraphs of $\mathbf{G}$ that do not contain any vertex in $\mathbf{G}_\rho.$ The coarse representative of $(\mathbf{G}, \rho)$ is the centrally aligned $(1,n,d)$-graph obtained by contracting each rational tail subgraph to a single vertex.
\end{definition}

\begin{definition}\label{defn:coarcl}
  Two centrally aligned $(1,n,d)$-graphs are \emph{coarsely equivalent} if they have isomorphic coarse representatives and extend this to an equivalence relation on the set. We use $[\mathbf{G}, \rho]$ to denote the equivalence class of $(\mathbf{G}, \rho).$ The morphisms among centrally aligned $(1,n,d)$-graphs induce morphisms on their coarse equivalence classes. In particular, an automorphism of a coarse class agrees with that of its coarse representative as a centrally aligned graph.
\end{definition}

\subsection{Graph strata of $\Mtil_{1,n}(\PP^r,d)$}
We now describe the \emph{coarse} strata $\Mtil_{(\mathbf{G}, \rho)}$ associated to coarse equivalence classes of centrally aligned dual graphs. Each stratum is a fibre product of: \begin{enumerate}
  \item a stratum in the moduli space of centrally aligned genus one curves $\Mbar_{1,n}^{\mathrm{cen}}$ \cite[§3]{rspw},
  \item a collection of genus zero maps satisfying the factorisation property \cite[Definition 4.1]{rspw},
  \item pointed genus zero stable maps.
\end{enumerate}

The factorisation property is defined in loc. cit. in terms of elliptic singularities and presented in later work \cite[§2.4]{bnr}, \cite[Lemma 3.10]{vzdualcomplex} as a linear dependency condition on certain derivatives of the genus zero maps which we will explain.

\begin{definition}
  Let $$\Mbar_{0,n}^*(\PP^r,d)\subset \Mbar_{0,n+1}(\PP^r,d),\mathcal{M}_{0,n}(\PP^r,d)\subset \M_{0,n+1}(\PP^r,d)$$ be the fibres over $p\in \mathbb{P}^r$ of the evaluation maps $$\mathrm{ev}_{n+1}: \Mbar_{0,n+1}(\PP^r,d)\to \PP^r, \M_{0,n+1}(\PP^r,d)\to \PP^r$$
  which are Zariski locally trivial fibrations.
\end{definition}

\begin{definition}
  Let $\boldsymbol{\delta}\in \mathbb{Z}_{\geq 0}^\ell,\boldsymbol{m}\in \mathbb{Z}_{\geq 0}^\ell$ be a tuple of degrees and (numbers of) marked points. Define $\mathbb{P}^{\circ}\mathcal{V}_{(\boldsymbol{\delta}, \boldsymbol{m})}$ be the $\mathbb{G}_m^{\ell-1}$-torsor over $\prod_{i=1}^\ell \M_{0,m_i}^*(\PP^r,\delta_i)$ recording a collection of tangent vectors $[v_1,\dots, v_\ell]$ at the frozen marked point (namely $p_{n+1}$) up to rescaling.
  
  We define the factorisation mapping space as \[\mathcal{M}^{\mathbf{F}}_{(\boldsymbol{\delta}, \boldsymbol{m})}:=\{((f_1,\dots,f_\ell), [v_1,\dots, v_\ell])\in \mathbb{P}^{\circ}\mathcal{V}_{(\boldsymbol{\delta}, \boldsymbol{m})}\mid \sum_{i=1}^\ell \partial f_i(v_i) = 0\in T_p\PP^r\}.\] This is a closed subscheme of $\mathcal{M}^{\mathbf{F}}_{(\boldsymbol{\delta}, \boldsymbol{m})}.$
\end{definition}

\begin{definition}
  Let $(\mathbf{G}, \rho)$ be a centrally aligned $(1,n,d)$-graph where the core has total degree zero. Its \emph{contraction core}, denoted as $\mathbf{G}^{cc},$ is the maximal subgraph with vertices in $V_\rho(\mathbf{G})\setminus \rho^{-1}(|\rho|)$: they all have degree zero. The restriction of $\rho$ to $\mathbf{G}^{cc}$ gives the pair $(\mathbf{G}^{cc}, \rho|_{\mathbf{G}^{cc}})$ the structure of a centrally aligned genus one tropical curve \cite[§3.3]{rspw}. Let $$\Mtil_{(\mathbf{G}^{cc}, \rho|_{\mathbf{G}^{cc}})}\subset \Mbar_{1,|L'(\mathbf{G}^{cc})|}^{\mathrm{cen}}$$ be the locally closed stratum associated to the pair. For brevity, we may use notation $\Mtil_{\mathbf{G}^{cc}}$ when there is no risk of confusion.
\end{definition}

\begin{remark}
  The $\Mbar_{1,n}^{\mathrm{cen}}\to \Mbar_{1,n}$ is a sequence of weighted blow-ups along boundary strata closures of $\Mbar_{1,n}$ \cite[Prop 3.3.4]{rspw}. The stratum $\Mtil_{\mathbf{G}^{cc}}$ is the total space of a torus fibre bundle over the stable curves stratum $\M_{\mathbf{G}^{cc}}\subset \Mbar_{1,n}.$
\end{remark}

  Let $[\mathbf{G}, \rho]$ be a coarse equivalence class of centrally aligned $(1,n,d)$-graph such that its core has total degree zero, and denote its associated stratum as $\Mtil_{[\mathbf{G}, \rho]}.$
  
  \begin{definition}Let $$\Mtil^{\mathrm{ord}}_{[\mathbf{G}, \rho]}\to \Mtil_{[\mathbf{G}, \rho]}$$ be the finite cover that labels the vertices on $\rho^{-1}(|\rho|)$ and the rational tail subgraphs of $[\mathbf{G}, \rho].$ We have $\Mtil^{\mathrm{ord}}_{[\mathbf{G}, \rho]}/\mathrm{Aut}([\mathbf{G}, \rho])\cong \Mtil_{[\mathbf{G}, \rho]}.$ Let $(\boldsymbol{\delta}^{(r)}_{[\mathbf{G}, \rho]},\boldsymbol{m}^{(r)}_{[\mathbf{G}, \rho]})$ be the tuples of degrees and markings of the vertices on the contraction radius $\rho^{-1}(|\rho|).$ Let $(\boldsymbol{\delta}_{[\mathbf{G}, \rho]}, \boldsymbol{m}_{[\mathbf{G}, \rho]})$ be the tuples of degrees and markings of the rational tails subgraphs.
  \end{definition}

  \begin{lemma}\label{lem:kunneth}
    The Borel--Moore homology of $\Mtil^{\mathrm{ord}}_{[\mathbf{G}, \rho]}$ has the Künneth formula $$H_\star^{\mathrm{BM}}(\Mtil^{\mathrm{ord}}_{[\mathbf{G}, \rho]})\cong H_\star(\mathbb{P}^r)\otimes H_\star^{\mathrm{BM}}(\Mtil_{\mathbf{G}^{cc}})\otimes H_\star^{\mathrm{BM}}(\mathcal{M}^{\mathbf{F}}_{(\boldsymbol{\delta}^{(r)}, \boldsymbol{m}^{(r)})})\otimes H_\star\left(\prod_{j=1}^{N_{[\mathbf{G}, \rho]}}\Mbar_{0,m_i}^* (\mathbb{P}^r,\delta_i)\right).$$ Taking finite group quotient, $$H_\star^{\mathrm{BM}}(\Mtil^{\mathrm{ord}}_{[\mathbf{G}, \rho]})_{\mathrm{Aut}[\mathbf{G},\rho]} = H_\star^{\mathrm{BM}}(\Mtil_{[\mathbf{G}, \rho]}).$$
  \end{lemma}
  \begin{proof}
    Restricting maps parametrised by $\Mtil^{\mathrm{ord}}_{[\mathbf{G}, \rho]}$ to the contraction core, vertices on the contraction radius $|\rho|$ and the rational tail subgraphs exhibits $\Mtil^{\mathrm{ord}}_{[\mathbf{G}, \rho]}$ as a total space of a iterated Zariski locally trivial fibration with fibres as tensor product summands shown above. The fibrations are Zariski locally trivial because they are base changes of $\mathrm{ev}_{n+1}: \Mbar_{0,n+1}(\PP^r,d)\to \PP^r$ and $\M_{0,n+1}(\PP^r,d)\to \PP^r.$
  \end{proof}

\begin{remark}
  When the core has positive total degree, the stratum $\Mtil_{[\mathbf{G},\rho]}$ agrees with the associated stratum in $\Mbar_{1,n}(\PP^r,d).$ Let $\mathcal{M}_{\mathbf{G}^{c}}\subset \Mbar_{1,n'_{\mathbf{G}^{c}}}(\mathbb{P}^r,\delta_{\mathbf{G}^{c}})$ be the genus one stable map stratum specified by the subgraph $\mathbf{G}^c,$ and let $(\boldsymbol{\delta}_{[\mathbf{G}, \rho]}, \boldsymbol{m}_{[\mathbf{G}, \rho]})$ be the tuples of degrees and markings of the rational tails subgraphs.
  
  In this case, we have the similar formula $$H_\star^{\mathrm{BM}}(\Mtil_{[\mathbf{G},\rho]})\cong \left(H^\star(\mathcal{M}_{\mathbf{G}^{c}})\otimes H^\star(\prod_{j=1}^{N_{[\mathbf{G},\rho]}}\Mbar_{0,m_i}^* (\mathbb{P}^r,\delta_i))\right)_{\mathrm{Aut}[\mathbf{G},\rho]}.$$
\end{remark}

The Künneth formulas are compatible with weight filtrations. This reduces the pure and off-by-one weight pieces of $H_\star^{\mathrm{BM}}(\mathcal{M}_{[\mathbf{G},\rho]})$ to those of its tensor product summands.

\section{Building blocks}\label{sec:build}
In this section, we determine the pure and off-by-one weight graded pieces of the Borel--Moore homology groups of the mapping spaces from each vertex or level of centrally aligned $(1,n,d)$-graphs. Because all the mapping spaces we encounter are smooth, it is often convenient to determine their singular cohomology and apply Poincaré duality.

\subsection{Genus zero maps with factorisation}\label{sec:gen0fac}
The weight graded pieces of $H_\star^{\mathrm{BM}}(\mathcal{M}^{\mathbf{F}}_{(\boldsymbol{\delta}, \boldsymbol{m})})$ are Poincaré dual to those of $H^\star(\mathcal{M}^{\mathbf{F}}_{(\boldsymbol{\delta}, \boldsymbol{m})}),$ which we describe via the concrete geometry of \emph{parametrised} mapping spaces.

\begin{definition}
  Let $\mathrm{Map}_{d}^{*}(\mathbb{P}^1, \mathbb{P}^r)=\{f: \mathbb{P}^1\to \mathbb{P}^r\mid f(\infty) = [1:1:\cdots:1]=:p\in \mathbb{P}^r\}$ be the space of parametrised pointed degree-$d$ maps from $\mathbb{P}^1$ to $\mathbb{P}^r.$ It is identified with the space of basepoint free polynomials as an open subset of $\mathbb{A}^{(r+1)d}$: $$\{(f_0(t),\cdots, f_r(t))\in \mathbb{A}^{(r+1)d}\mid f_i(t)\in \mathbb{C}[t]\text{ monic of degree }d, \bigcap_{i=0}^r f_i^{-1}(0) = \varnothing\}.$$ 
\end{definition}

\begin{definition}\label{defn:MapF}
Let $\boldsymbol{\delta}\in \mathbb{Z}_{\geq 0}^k$ be a degree vector. Define $$\mathrm{Map}_{\boldsymbol{\delta}}^{*, \mathsf{F}}(\mathbb{P}^1, \mathbb{P}^r)\subset \prod_{i=1}^k \mathrm{Map}_{\boldsymbol{\delta}_i}^{*}(\mathbb{P}^1, \mathbb{P}^r)$$ as the subspace consisting of tuples of parametrised pointed maps such that the images of tangent lines \(d_\infty^{(i)}f_i: T_{\infty} \mathbb{P}^1\to T_p \mathbb{P}^r\) satisfies some non-vanishing linear dependency: for any choice of basis vectors \(v^{(i)}\in T_\infty \mathbb{P}^1\), there exists \((\alpha_1,\dots, \alpha_k)\in (\mathbb{C}^\star)^{k}\) such that \(\sum_{i=1}^k \alpha_i d_\infty^{(i)}f_i(v^{(i)})=0.\)

Let \[\widetilde{\mathrm{Map}}_{\boldsymbol{\delta}}^{*, \mathsf{F}}(\mathbb{P}^1, \mathbb{P}^r)\to \mathrm{Map}_{\boldsymbol{\delta}}^{*, \mathsf{F}}(\mathbb{P}^1, \mathbb{P}^r)\] parametrise in addition the data of the non-vanishing linear dependency itself up to common rescaling. Over a tuple of maps $(f_0,\dots, f_r)\in \mathrm{Map}_{\boldsymbol{\delta}}^{*, \mathsf{F}}(\mathbb{P}^1, \mathbb{P}^r),$ the fibre is $$\{[\alpha_1:\dots:\alpha_k]\in (\mathbb{C}^\star)^{k}/(\mathbb{C}^\star\cdot \mathrm{Id})\mid \sum_{i=1}^k \alpha_i d_\infty^{(i)}f_i(v^{(i)})=0\}.$$
\end{definition}

\begin{remark}\label{rem:parunpar}
  Up to permutation of the entries, the mapping space $\mathcal{M}^{\mathbf{F}}_{(\boldsymbol{\delta}^{(r)}, \boldsymbol{m}^{(r)})}$ is related to their corresponding parametrised mapping spaces in the following standard way:
  \begin{enumerate}
    \item for $\boldsymbol{m}_i^{(r)} = 0,1,$ $\mathcal{M}^{\mathbf{F}}_{(\boldsymbol{\delta}^{(r)}, \boldsymbol{m}^{(r)})}$ is the quotient of the parametrised mapping space by $\mathrm{Aut}(\mathbb{P}^1, *)$ (a Borel subgroup of $\mathrm{PGL}_2(\mathbb{C})$) or $\mathbb{C}^\star$ respectively, which acts on parametrised maps from the corresponding univalent or bivalent component;
    \item for $\boldsymbol{m}^{(r)}_i\geq 2$,  $\mathcal{M}^{\mathbf{F}}_{(\boldsymbol{\delta}^{(r)}, \boldsymbol{m}^{(r)})}$ is non-canonically isomorphic to the product of the parametrised mapping space with an ordered configuration space\footnote{The numerics are one off from the standard isomorphism $\mathrm{Map}_d (\mathbb{P}^1,\mathbb{P}^r)\cong \mathcal{M}_{0,3}(\mathbb{P}^r,d)$ because of the presence of a frozen marked point.} $\mathrm{Conf}^{\boldsymbol{m}^{(r)}-2}(\mathbb{P}^1\setminus \{0,1,\infty\}).$
  \end{enumerate}
\end{remark}

\subsubsection{Pointed maps with prescribed tangent lines}\label{sec:prestan}
We start by studying pointed mapping spaces with a prescribed tangent line along the marked point. 

Via the open embedding $$\mathrm{Map}_{\boldsymbol{\delta}_i}^{*}(\mathbb{P}^1, \mathbb{P}^r)\subset \mathbb{A}^{\boldsymbol{\delta}_i(r+1)},$$ prescribing a tangent line is equivalent to taking an affine linear subspace in $\mathbb{A}^{\boldsymbol{\delta}_i(r+1)},$ where the coordinates are the polynomial coefficients. The affine coordinates specify a basis vector \(v\in T_\infty \mathbb{P}^1\), and taking the derivative at \(v\) gives the map $$d_\infty: \prod_{i=1}^k  \mathrm{Map}_{\boldsymbol{\delta}_i}^{*}(\mathbb{P}^1, \mathbb{P}^r)\to \prod_{i=1}^k T_p \mathbb{P}^r,$$ which is the product of maps $$d_\infty^{(i)}: \mathrm{Map}_{\boldsymbol{\delta}_i}^{*}(\mathbb{P}^1, \mathbb{P}^r)\to T_p \mathbb{P}^r.$$

An elementary calculation gives the following:

\begin{lemma}\cite[Lemma 4.1]{vzdualcomplex}\label{lem-derivative}
  The map \(d_\infty^{(i)}\) factors through \[\mathrm{Map}_{\boldsymbol{\delta}_i}^{*}(\mathbb{P}^1, \mathbb{P}^r)\to \mathbb{A}^{\boldsymbol{\delta}_i(r+1)}\xrightarrow{p_{\boldsymbol{\delta}_i-1}} \mathbb{A}^{r+1}\to T_p \mathbb{P}^r,\] where the projection \(p_{\boldsymbol{\delta}_i-1}\) records the \(z^{\boldsymbol{\delta}_i-1}\)-coefficient of the \((r+1)\)-tuple of polynomials, and the second map is the linear map \(\mathbb{C}^{r+1}\to \mathbb{C}^{r+1}/\mathbb{C}\cdot (1,1,\dots, 1)\cong T_p \mathbb{P}^r.\)
\end{lemma}

\begin{definition}
  Let \(w\in T_p \mathbb{P}^r,\) define \[\mathrm{Map}_{{\delta}}^{*, w}(\mathbb{P}^1, \mathbb{P}^r):=\{f\in \mathrm{Map}_{{\delta}}^{*}(\mathbb{P}^1, \mathbb{P}^r)\mid d_\infty(v) = w \}.\] When there is no risk of confusion, we may use $\mathrm{Map}_{{\delta}}^{*, w}$ to denote $\mathrm{Map}_{{\delta}}^{*, w}(\mathbb{P}^1, \mathbb{P}^r)$ for sake of brevity.
\end{definition}

\begin{remark}By Lemma \ref{lem-derivative}, \(\mathrm{Map}_{{\delta}}^{*, w}(\mathbb{P}^1, \mathbb{P}^r)\) is the intersection of the open subset $\mathrm{Map}_{{\delta}}^{*}(\mathbb{P}^1, \mathbb{P}^r)\subset \mathbb{A}^{\delta(r+1)}$ and the affine subspace $$L_{\delta}^{w}: = \{(f_0,\cdots, f_r)\mid \left[\left(p_{\delta-1}(f_i)\right)_{i=0}^r\right] = [w]\in \mathbb{C}^{r+1}/\langle(1,1,\cdots,1)\rangle\cong T_p \mathbb{P}^r\},$$ so it is in particular an open in an affine space isomorphic to $\mathbb{A}^{\delta(r+1)-r}.$\end{remark}

We calculate the cohomology of \(\mathrm{Map}_{{\delta}}^{*, w}(\mathbb{P}^1, \mathbb{P}^r)\) following the techniques in \cite{farbwolf} which calculated the cohomology of $\mathrm{Map}_{\delta}^* (\mathbb{P}^1, \mathbb{P}^r)$.

\begin{lemma}\label{lem:mapwdelta}
  When $\delta\geq 2$, for all $w$, $$
    H^i(\mathrm{Map}_{{\delta}}^{*, w})=\begin{cases*}
    \mathbb{Q}(-r), \hspace{5pt} i = 2r-1\\
    \mathbb{Q}(0), \hspace{13pt} i=0 \\
    0, \hspace{30pt} \text{otherwise}
  \end{cases*}$$

  When $\delta = 1$ and $w\neq 0$, $\mathrm{Map}_\delta^{*, w} \cong \mathbb{A}^1$, whereas the space is empty when $\delta = 1$ and $w=0.$ Because $\mathrm{Map}_{{\delta}}^{*, w}$ is smooth, there is Poincaré duality $$H_i^{\mathrm{BM}}(\mathrm{Map}_{{\delta}}^{*, w})\cong H^{2((\delta-1)(r+1)+1)-i}(\mathrm{Map}_{{\delta}}^{*, w})^\vee$$ compatible with the weight filtration.
\end{lemma}

\begin{remark}
  For $\delta\geq 2,$ the cohomology of $\mathrm{Map}_{\delta}^{*,w}(\mathbb{P}^1, \mathbb{P}^r)$ together with its mixed Hodge structure agrees with that of $\mathbb{A}^r\setminus \mathrm{pt}$ and are \emph{independent} of $\delta$ and $w.$ This is closely related to the result of Farb--Wolfson \cite[Theorem 1.2]{farbwolf} that when $\delta>0$, the cohomology of $\mathrm{Map}_{\delta}^{*}(\mathbb{P}^1, \mathbb{P}^r)$ agrees with that of $\mathbb{A}^r\setminus \mathrm{pt}$ and is independent of $\delta.$
\end{remark}

\begin{proof}
  When $w\neq 0$ and $\delta = 1,$ the space $\mathrm{Map}_{1}^{*,w}(\mathbb{P}^1, \mathbb{P}^r)$ is identified with the space of monic polynomials $\{(t + \alpha_0,\dots, t+ \alpha_r)\in \mathbb{A}^{r+1}\mid [(\alpha_0,\dots, \alpha_r)]=w\in \mathbb{C}^{r+1}/\langle (1,\dots, 1)\rangle\}$ and is hence isomorphic to $\mathbb{A}^1.$

  When $\delta\geq 2$, we use induction on $\delta.$ The inductive hypothesis is that $$H_i^{\mathrm{BM}}(\mathrm{Map}_{{\delta}}^{*, w})=\begin{cases*}
    \mathbb{Q}(\delta(r+1)-2r+1), \hspace{3pt} i = 2(\delta(r+1)-2r)+1\\
    \mathbb{Q}(\delta(r+1)-r), \hspace{28pt} i=2(\delta(r+1)-r) \\
    0, \hspace{101pt} \text{otherwise}
  \end{cases*}$$ and that $H^{\mathrm{BM}}_{2\delta(r+1)+1}(\mathrm{Map}_{{\delta}}^{*, w})$ is identified with the fundamental class of the basepoint locus under the boundary map. The base case is split into cases $w = 0$ and $w\neq 0.$

  When $w = 0,\delta = 2,$ $\mathrm{Map}_{2}^{*,0}(\mathbb{P}^1, \mathbb{P}^r) = L_0\setminus \mathbb{A}^2,$ where $\mathbb{A}^2$ is identified with the space of monic quadratic polynomials and $\mathbb{A}^2\hookrightarrow L_{2}^{0}$ is the diagonal map. The inductive hypothesis for this case follows from the excision long exact sequence.

  When $w \neq 0,\delta =2,$ $\mathrm{Map}_{2}^{*,0} = L_{2}^{w}\setminus (\mathbb{A}^1\times \mathrm{Map}_{1}^{*,w})$ where $\mathbb{A}^1\times \mathrm{Map}_{1}^{*,w}\to  L_{2}^{w}$ is given by multiplication $$(z, (t+\alpha_0,\dots, t+\alpha_r))\mapsto ((t+z)(t+\alpha_0),\dots, (t+z)(t+\alpha_r))$$ and is an isomorphism onto the complement $ L_{w}\setminus \mathrm{Map}_{2}^{*,0}$. As $\mathrm{Map}_{1}^{*,w}\cong \mathbb{A}^1$ from earlier, we have that $\mathrm{Map}_{2}^{*,0}\cong L_{2}^{w}\setminus \mathbb{A}^2$ and the excision long exact sequence implies the inductive hypothesis.
  
  For the inductive step, we stratify $L_{\delta}^{w}$ by the number of base points: let $L_{\delta, -k}^{w}\subset L_{\delta}^{w}$ be the closed subspace of tuples of polynomials with at least $k$ basepoints and let $L_{\delta, -k}^{w, \circ}\subset L_{\delta, -k}^{w}$ be the open subspace of polynomials with precisely $k$ base points. From \cite[§3]{farbwolf}, the multiplication map $\mathbb{A}^k\times \mathrm{Map}_{\delta-k}^{*, w}\to L_{\delta, -k}^{w, \circ}$ given by $(f, (g_0, \dots, g_k))\mapsto (fg_0,\dots, fg_k)$ is proper and a \textit{homeomorphism}, so the pullback is an isomorphism on compactly supported cohomology, compatible with mixed Hodge structures. Now we apply the stratification spectral sequence associated to the filtration\footnote{When $w = 0,L_{\delta, -(d-1)}^{w} = L_{\delta, -d}^w,$ in which case we formally delete the item $L_{\delta, -d}^w$ from the  filtration.} $$\dots \subsetneq L_{\delta, -(k+1)}^w\subsetneq L_{\delta, -k}^w\subsetneq L_{\delta, 0}^w = L_{\delta}^{w},$$ which is given by $E^1_{-k, q} = H^{\mathrm{BM}}_{-k+q}(L_{\delta, -k}^{w,\circ})\Rightarrow H^{\mathrm{BM}}_{-k+q}(L_{\delta, -d}^w).$ 
  
  Recall that the inductive hypothesis states that for all $\delta'<\delta$, the Borel--Moore homology group of $\mathrm{Map}_{\delta'}^{*, w}$ is the fundamental class and the preimage of fundamental class of $\mathrm{Map}_{\delta'-1}^{*, w}$ under the boundary map; for $k>0,$ taking the homeomorphisms between $L_{\delta, -k}^{w, \circ}$ and $\mathbb{A}^k\times \mathrm{Map}_{\delta-k}^{*, w},$ a similar description holds for $L_{\delta, -k}^{w, \circ}.$ Because the $d^1$ differentials among the $E^1$-pages $H^{\mathrm{BM}}_{-k+q}(L_{\delta, -k}^{w,\circ})\to H^{\mathrm{BM}}_{-(k-1)+q}(L_{\delta, -(k-1)}^{w,\circ})$ are given by boundary maps, they are isomorphisms for all $k>1.$ Therefore, $E^2_{-k, q}$ vanishes for all $k>1.$
  
  It remains to determine $d^1$ on the terms $E^1_{-1,q}$ and $E^1_{0,q}.$ Because the spectral sequence is in the fourth quadrant, $E^2_{-1,q} = E^{\infty}_{-1,q},$ which is a subquotient of $H^{\mathrm{BM}}_{-1+q}(L_{\delta}^{w})\cong H^{\mathrm{BM}}_{-1+q}(\mathbb{A}^{\delta(r+1)-r}),$ and similarly $E^2_{0,q} = E^{\infty}_{0,q}.$ Therefore, $E^2_{-1,q}=0$ for all $q,$ and $E^2_{0,q} = 0$ for $q\neq 2(\delta(r+1)-r),$ and $E^2_{-1,2(\delta(r+1)-r)}=\mathbb{Q}(\delta(r+1)-r).$ In other words, $d_1^{-1,q}$ are isomorphisms for all $q,$ and $d_1^{0,q}$ are isomorphisms for all $q\neq 2(\delta(r+1)-r),$ and $E_1^{0, 2(\delta(r+1)-r)} = E_2^{0, 2(\delta(r+1)-r)} =\mathbb{Q}(\delta(r+1)-r).$ We deduce that the Borel--Moore homology of $\mathrm{Map}_{\delta}^{*,w}$ is given by $$\begin{cases*}
    \mathbb{Q}(\delta(r+1)-2r), \hspace{23pt} i = 2(\delta(r+1)-2r)+1\\
    \mathbb{Q}(\delta(r+1)-r), \hspace{28pt} i=2(\delta(r+1)-r) \\
    0, \hspace{101pt} \text{otherwise}
  \end{cases*}$$ and $H_{2(\delta(r+1)-2r)+1}^{\mathrm{BM}}(\mathrm{Map}_{\delta}^{*,w})$ is identified with $H_{2((\delta-1)(r+1)-r+1)}^{\mathrm{BM}}(\mathbb{A}^1\times \mathrm{Map}_{\delta-1}^{*,w})$ under the boundary map. This finishes the inductive step.
\end{proof}

\subsubsection{Spaces of linearly dependent vectors}
We describe the spaces of tangent vectors of $\mathbb{P}^r$ that may arise from maps with linearly dependent derivatives and determine the weight pieces of their Borel--Moore homology groups.

\begin{definition}
  Let \(V\) be an \(r\)-dimensional vector space over \(\mathbb{C}\): for our purposes, \(V= T_p \mathbb{P}^r\). Denote $\mathsf{D}_k^*\subset V^k$ as the locus of vectors that admit some non-vanishing linear dependency. Let \(\pi_{\sim}:\widetilde{\mathsf{D}}_k^* \to \mathsf{D}_k^*\) parametrise in addition the non-vanishing linear dependency itself up to common rescaling. Abusing notation, we may replace the subscript $k$ by an indexing set of size $k$, such as $[k] = \{1,\dots,k\}.$
  
  Let $\boldsymbol{\delta}\in \mathbb{Z}_{\geq 0}^k$ be a length-$k$ degree vector. Define \[\mathsf{D}_{\boldsymbol{\delta}}^* = \{\mathbf{v}\in {\mathsf{D}}^*_{k}\mid v_i\neq 0\text{ if }\boldsymbol{\delta}_i = 1; v_i= 0 \text{ if }\boldsymbol{\delta}_i = 0\}\subset \mathsf{D}_k,\] \[{\widetilde{\mathsf{D}}}^*_{\boldsymbol{\delta}}:=\pi_{\sim}^{-1}(\mathsf{D}_{\boldsymbol{\delta}}^*)\subset \widetilde{\mathsf{D}}_k^*.\]
\end{definition}

\begin{remark}
  From the description in Lemma \ref{lem:mapwdelta}, the spaces $\mathsf{D}^*_{\boldsymbol{\delta}}$ and ${\widetilde{\mathsf{D}}}^*_{\boldsymbol{\delta}}$ describe the realisable tangent vectors in the parametrised mapping space $\mathrm{Map}_{\boldsymbol{\delta}}^{*, \mathsf{F}}(\mathbb{P}^1, \mathbb{P}^r)$ and $\widetilde{\mathrm{Map}}_{\boldsymbol{\delta}}^{*, \mathsf{F}}(\mathbb{P}^1, \mathbb{P}^r)$ from Definition \ref{defn:MapF}.
\end{remark}

\begin{remark}
  The space $\widetilde{\mathsf{D}}_k^*$ deformation retracts to $(\mathbb{C}^\star)^k/\mathbb{C}^\star,$ and we present the cohomology of \(\widetilde{\mathsf{D}}_k^*\) as $H^\star((\mathbb{C}^\star)^{k}/\mathbb{C}^\star\cdot \mathrm{Id})\cong \bigwedge \left(\mathbb{Q}\{\alpha_1,\dots, \alpha_k\}/\langle\sum_{i=1}^k
  \alpha_i\rangle\right),$ where each $\alpha_i\in H^1((\mathbb{C}^\star)^{k}/\mathbb{C}^\star)$ is the generator of the cohomology the corresponding factor in $(\mathbb{C}^\star)^{k}.$
\end{remark}

\begin{definition}
  For a multi-degree vector $\boldsymbol{\delta},$ define $I_{\boldsymbol{\delta}}^{(1)}:=\{i\in [k]\mid \boldsymbol{\delta}_i = 1\}$ and $I_{\boldsymbol{\delta}}^{(0)}:=\{i\in [k]\mid \boldsymbol{\delta}_i = 0\}.$ Denote $[k]_{\boldsymbol{\delta}_{\setminus 0}}:=[k]\setminus I_{\boldsymbol{\delta}}^{(0)}.$
\end{definition}

\begin{remark}\label{rem:minuszerodeg}
  The projection map $\widetilde{\mathsf{D}}^*_{\boldsymbol{\delta}}\to \widetilde{\mathsf{D}}^*_{\boldsymbol{\delta}_{\setminus 0}}$ gives an isomorphism $\widetilde{\mathsf{D}}^*_{\boldsymbol{\delta}}\cong \widetilde{\mathsf{D}}^*_{\boldsymbol{\delta}_{\setminus 0}}\times (\mathbb{C}^\star)^{I_{\boldsymbol{\delta}}^{(0)}}.$
\end{remark}

Using the open embedding $\widetilde{\mathsf{D}}^*_{\boldsymbol{\delta}}\subset \widetilde{\mathsf{D}}_k^*,$ we describe the pure and off-by-one weight graded pieces as follows.

\begin{lemma}\label{lem:Dtilstar}
  Use $N_{k}$ to denote $\dim \widetilde{\mathsf{D}}^*_{\boldsymbol{\delta}} = (r+1)(k-1).$ The pure weight Borel--Moore homology of $\widetilde{\mathsf{D}}^*_{\boldsymbol{\delta}}$ is spanned by its fundamental class $[\widetilde{\mathsf{D}}^*_{\boldsymbol{\delta}}]\in H^{\mathrm{BM}}_{2 N_{k}}(\widetilde{\mathsf{D}}^*_{\boldsymbol{\delta}}).$

  The off-by-one weight graded pieces $\gr^{W}_{-\star+1} H^{\mathrm{BM}}_\star(\widetilde{\mathsf{D}}^*_{\boldsymbol{\delta}})$ is isomorphic to $$\mathbb{Q}(N_{k}-r)^{\oplus |I_{\boldsymbol{\delta}}^{(1)}|}\oplus \mathbb{Q}(N_{k}-1)^{\oplus k-1},$$ where the two summands are respectively given by: \begin{enumerate}
    \item the boundary map images of the fundamental classes of the closed subspaces $\{(\boldsymbol{v}, \boldsymbol{\alpha})\in \widetilde{\mathsf{D}}^*_{k}\mid \boldsymbol{v}_j = 0\}\cong \mathbb{C}^\star\times \widetilde{\mathsf{D}}^*_{{\boldsymbol{\delta}}\setminus \{j\}}$ for all $j\in I_{\boldsymbol{\delta}}^{(1)},$
    \item the torus factors coming from $\widetilde{\mathsf{D}}_k^*$.
  \end{enumerate}
\end{lemma}

\begin{lemma}\label{lem:fibDF} The parametrised mapping space $\widetilde{\mathrm{Map}}_{\boldsymbol{\delta}}^{*, \mathsf{F}}(\mathbb{P}^1, \mathbb{P}^r)$ is the fibre product:
  \[\begin{tikzcd}
	{\widetilde{\mathrm{Map}}_{\boldsymbol{\delta}}^{*, \mathsf{F}}(\mathbb{P}^1, \mathbb{P}^r)} & {\prod_{i=1}^k \mathrm{Map}_{\boldsymbol{\delta}_i}^{*}(\mathbb{P}^1, \mathbb{P}^r)} \\
	{\widetilde{\mathsf{D}}_{\boldsymbol{\delta}}^*} & {\left(T_p \mathbb{P}^r\right)^{\oplus k}}
	\arrow[from=1-1, to=1-2]
	\arrow[from=1-1, to=2-1]
	\arrow[from=1-2, to=2-2]
	\arrow[from=2-1, to=2-2]
\end{tikzcd}\]
Thus, over a point \({((v_1,\dots, v_k), [\alpha_1,\dots, \alpha_k])\in \widetilde{\mathsf{D}}_{\boldsymbol{\delta}}^*}\), the fibre of the map $\widetilde{\mathrm{Map}}_{\boldsymbol{\delta}}^{*, \mathsf{F}}(\mathbb{P}^1, \mathbb{P}^r)\to {\widetilde{\mathsf{D}}}^*_{\boldsymbol{\delta}}$ is $\prod_{i=1}^k \mathrm{Map}_{{\delta}}^{*, v_i}(\mathbb{P}^1, \mathbb{P}^r).$
\end{lemma}

\begin{corollary}\label{cor:genDF}
   Denote $\dim_{M_{\boldsymbol{\delta}}}$ as the dimension of $\prod_{i=1}^k \mathrm{Map}_{\boldsymbol{\delta}_i}^{*,w_i}$ for any choice of $(w_i)_{i=1}^k$ such that the product is non-empty. The pure weight Borel--Moore homology groups are given by $$\gr^W_{-\star} H^{\mathrm{BM}}_\star(\widetilde{\mathrm{Map}}_{\boldsymbol{\delta}}^{*, \mathsf{F}}(\mathbb{P}^1, \mathbb{P}^r)) = \mathbb{Q}\cdot [\widetilde{\mathrm{Map}}_{\boldsymbol{\delta}}^{*, \mathsf{F}}(\mathbb{P}^1, \mathbb{P}^r)],$$
  
  The off-by-one weight graded pieces are given by \begin{align*}\gr^W_{-\star+1} H^{\mathrm{BM}}_\star(\widetilde{\mathrm{Map}}_{\boldsymbol{\delta}}^{*, \mathsf{F}}(\mathbb{P}^1, \mathbb{P}^r)) & = \left(\underbrace{\mathrm{gr}^W_{-\star+1}H^{\mathrm{BM}}_\star(\widetilde{\mathsf{D}}_{\boldsymbol{\delta}}^*)}_{\text{Lemma \ref{lem:Dtilstar}}}\otimes \mathbb{Q}\cdot [\prod_{i=1}^k\mathrm{Map}_{\boldsymbol{\delta}_i}^{*,w}] \right)\\ & \oplus \left(\mathbb{Q}\cdot [\widetilde{\mathsf{D}}_{\boldsymbol{\delta}}^*]\otimes \underbrace{\mathbb{Q}(\dim_{M_{\boldsymbol{\delta}}}-r)^{\oplus k-|I_{\boldsymbol{\delta}}^{(0)}|-|I_{\boldsymbol{\delta}}^{(1)}|}}_{\gr^W_{-\star+1}H^{\mathrm{BM}}_\star(\prod_{i=1}^k\mathrm{Map}_{\boldsymbol{\delta}_i}^{*,w}), \text{ Lemma \ref{lem:mapwdelta}}}\right)\end{align*}
\end{corollary}
\begin{proof}
We compute the compactly supported cohomology of $\widetilde{\mathrm{Map}}_{\boldsymbol{\delta}}^{*, \mathsf{F}}(\mathbb{P}^1, \mathbb{P}^r)$ by the Leray spectral sequence of the map $$d_\infty: \widetilde{\mathrm{Map}}_{\boldsymbol{\delta}}^{*, \mathsf{F}}(\mathbb{P}^1, \mathbb{P}^r)\to {\widetilde{\mathsf{D}}}^*_{\boldsymbol{\delta}}.$$ As the map is the pullback of the linear projection map $$\prod_{i=1}^k \mathbb{A}^{\boldsymbol{\delta}_i(r+1)}\to \prod_{i=1}^k \mathbb{C}^{r+1}\to (T_p \mathbb{P}^r)^{\oplus k},$$ the locally constant sheaf $(d_{\infty})_! \mathbb{Q}$ has trivial monodromy, so the $E_2$-page satisfies the Künneth formula. The claim follows from weight truncating the spectral sequences, applying Lemma \ref{lem:mapwdelta}, and inspecting the weights and cohomological degrees.
\end{proof}

We set notation for the following classes corresponding to basepoints on the moduli spaces of parametrised pointed maps.

\begin{definition}\label{defn:betav}
  Let $\beta_v\in H^{2r-1}(\widetilde{\mathrm{Map}}_{\boldsymbol{\delta}}^{*, \mathsf{F}}(\mathbb{P}^1, \mathbb{P}^r))$ be:\begin{enumerate}
    \item when $\boldsymbol{\delta}_v\geq 2,$ the pullback from the factor associated to the vertex $v$ of $\beta\in H^{2r-1}(\mathrm{Map}_{\boldsymbol{\delta}_v}^{*,w})$ (Lemma \ref{lem:mapwdelta}),
    \item when $\boldsymbol{\delta}_v = 1,$ the pullback from $\mathbb{Q}(N_k-r)\subset \mathrm{gr}^W_{\star+1}H^\star(\widetilde{\mathsf{D}}_{\boldsymbol{\delta}}^*)$ corresponding to the entry $v$ (item (1) in Lemma \ref{lem:Dtilstar}). 
  \end{enumerate}
   
\end{definition}

\subsubsection{Quotienting by automorphisms}\label{subsubsec:quotaut}
From Remark \ref{rem:parunpar}, the fibre of the forgetful map
$$\mathcal{M}^{\mathbf{F}}_{(\boldsymbol{\delta}^{(r)}, \boldsymbol{m}^{(r)})}\to \prod_{i=1}^k \mathcal{M}_{0, \boldsymbol{m}_i}$$ is given by $\widetilde{\mathrm{Map}}_{\boldsymbol{\delta}}^{*, \mathsf{F}}(\mathbb{P}^1, \mathbb{P}^r).$ We calculate the weight graded pieces of $H^\star(\mathcal{M}^{\mathbf{F}}_{(\boldsymbol{\delta}^{(r)}, \boldsymbol{m}^{(r)})})$ using the Leray spectral sequence associated to the above morphism. 

We set notation for $\psi$-classes on $\mathsf{M}_{0,1}, \mathsf{M}_{0,2}.$ They pull back to pure weight classes on strata of genus zero stable maps, the strata $\mathcal{M}^{\mathbf{F}}_{(\boldsymbol{\delta}^{(r)}, \boldsymbol{m}^{(r)})},$ and \(\mathcal{M}_{\mathsf{C}_{k, (\mathbf{d}, \mathbf{m})}}.\)

\begin{definition}
  Let $\psi\in H^2(\mathsf{M}_{0,1})$ or $\psi\in H^2(\mathsf{M}_{0,2})$ be the standard generator under the isomorphism of stacks $\mathsf{M}_{0,1}\cong B \mathrm{Aut}(\mathbb{P}^1, *)$ and $\mathsf{M}_{0,2}\cong {B}\mathbb{G}_m.$ They agree with the psi classes of the marked points up to a sign, which suffices for our purposes.
  
  Let $v$ be a univalent or bivalent vertex in a $(1,n,d)$-graph, possibly with central alignment. The $\psi$-class $\psi_v$ of vertex $v$ is the cohomology class in the corresponding mapping space stratum pulled back from $\psi\in H^2(\mathsf{M}_{0,1})$ or $\psi\in H^2(\mathsf{M}_{0,2})$ along the unique inward pointing half-edge, which is denoted as $\mu_v$ in \cite{vzdualcomplex}. 
\end{definition}

\begin{lemma}\label{lem:MF}
  The pure weight singular cohomology of $\mathcal{M}^{\mathbf{F}}_{(\boldsymbol{\delta}^{(r)}, \boldsymbol{m}^{(r)})}$ is: \begin{enumerate}
    \item spanned by $\{1, \psi,\dots, \psi^{r-1}\}$ when all vertices of $\rho^{-1}(|\rho|)$ are univalent or bivalent; the class $\psi\in H^2(\mathcal{M}^{\mathbf{F}}_{(\boldsymbol{\delta}^{(r)}, \boldsymbol{m}^{(r)})})$ is the pulled back $\psi$-class of any univalent or bivalent vertex,
    \item $H^0\cong \mathbb{Q}$ when there is any vertex that is not univalent or bivalent.
  \end{enumerate}

  As a $W_\star H^\star(\mathcal{M}^{\mathbf{F}}_{(\boldsymbol{\delta}^{(r)}, \boldsymbol{m}^{(r)})})$-module, the off-by-one weight singular cohomology of $\mathcal{M}^{\mathbf{F}}_{(\boldsymbol{\delta}^{(r)}, \boldsymbol{m}^{(r)})}$ is generated by:
  \begin{enumerate}
    \item when all vertices of $\rho^{-1}(|\rho|)$ are univalent or bivalent, the pullback from the codimension one linear subspace $\langle \beta_v-\beta_w\rangle_{v,w} \subset H^{2r-1}(\widetilde{\mathrm{Map}}_{\boldsymbol{\delta}}^{*, \mathsf{F}}(\mathbb{P}^1, \mathbb{P}^r))$ (Definition \ref{defn:betav}), 
    \item when not all vertices are univalent or bivalent, \begin{enumerate}
      \item pullback of $H^1(\mathcal{M}_{0, \boldsymbol{m}_i})$ for all $\boldsymbol{m}_i\geq 4,$
      \item pullback of $[\alpha_v-\alpha_w]\in H^1((\mathbb{C}^\star)^k/\mathbb{C}^\star)$ for all vertices $v,w$ that are not univalent or bivalent,
      \item $\beta_v$ for all vertices $v.$
    \end{enumerate}
  \end{enumerate}
\end{lemma}
\begin{proof}
  By inspecting the weights and cohomological degrees of the Leray spectral sequence associated to $\mathcal{M}^{\mathbf{F}}_{(\boldsymbol{\delta}^{(r)}, \boldsymbol{m}^{(r)})}\to \prod_{i=1}^k \mathcal{M}_{0, \boldsymbol{m}_i}$, the possibly non-vanishing differentials concerning pure and off-by-one weight graded pieces are multiplicatively generated by \begin{enumerate}
  \item $d_2^{0,1}: \mathrm{gr}^W_2 E_{2}^{0,1}\to \mathrm{gr}^W_2 E_2^{2,0},$ in which each $[\alpha_v-\alpha_w]\in H^1((\mathbb{C}^\star)^k/\mathbb{C}^\star)\subset \mathrm{gr}^W_2 E_{2}^{0,1}$ is mapped to $\psi_v-\psi_w.$ This is because the $(\mathbb{C}^\star)^k/\mathbb{C}^\star$-bundle parametrises collections of compatible isomorphisms between the tangent lines at the frozen marked points \cite[§2]{bnr}, see also \cite[Lemma 3.2]{vzdualcomplex}. This implies the relation $\psi_v=\psi_w$ when pulled back to $\mathcal{M}^{\mathbf{F}}_{(\boldsymbol{\delta}^{(r)}, \boldsymbol{m}^{(r)})},$ so that all $\psi_v$ pulls back to zero when there is any vertex not univalent or bivalent, and that they all pull back to the same class $\psi$ otherwise,
  \item $d_{2r}^{0,2r-1}: \mathrm{gr}^W_{2r} E_{2}^{0,2r-1}\to \mathrm{gr}^W_{2r} E_{2r}^{2r,0},$ in which $\beta_v\mapsto \psi^r.$ This is the only possibly non-zero differential after $E_2,$ and $W_{\star}H^\star(\mathcal{M}_{(\boldsymbol{\delta}^{(r)}, \boldsymbol{m}^{(r)})}^{\boldsymbol{F}})$ would be infinite-dimensional if it were zero. This can also be seen by an explicit calculation of the $\mathbb{C}^\star$-action on $\widetilde{\mathrm{Map}}_{\boldsymbol{\delta}}^{*,\mathsf{F}}(\PP^1,\PP^r)\subset (\mathbb{C}^{\star})^k/\mathbb{C}^{\star}\times \prod_{i=1}^{k}\mathrm{Map}_{\boldsymbol{\delta}_i}^i(\PP^1,\PP^r),$ where $\mathbb{C}^\star\subset \mathrm{Aut}(\PP^r,*)$ is a maximal torus, so that $B\mathbb{C}^\star\to B\mathrm{Aut}(\PP^r,*)$ induces a cohomology isomorphism.
  \end{enumerate}
\end{proof}

\subsection{Maps from smooth elliptic curves}\label{sec:mapellip}
We describe the weight graded pieces of $\mathcal{M}_{1,n}(\PP^r,d)$ via the partial compactification $\hat{\mathcal{Q}}_{1,n}(\PP^r,d)$ defined in §\ref{sec:introgenrel}. The space $\hat{\mathcal{Q}}_{1,n}(\PP^r,d)$ is the total space of a projective bundle over the universal Picard group $\mathrm{Pic}^d_{1,n},$ such that the fibre over $((E, p_1,\dots,p_n),L)\in \mathrm{Pic}^d_{1,n}$ is given by $\mathbb{P}(H^0(E, L)^{\oplus r+1}).$ It has relative hyperplane class $H_{\hat{\mathcal{Q}}}\in H^2(\hat{\mathcal{Q}}_{1,n}(\PP^r,d)).$

\begin{remark}
  While the notation is chosen to resemble that of the moduli space of quasimaps \cite{mop}, the basepoints in $\hat{\mathcal{Q}}_{1,n}(\mathbb{P}^r,d)$ are not required to be disjoint from the marked points. In this way, the space $\hat{\mathcal{Q}}_{1,n}(\mathbb{P}^r,d)$ can be considered as a quasimaps moduli space where each marked point has zero weight.
\end{remark}

\begin{lemma}\label{lem:propsurj}
  The complement $\mathcal{B}_{1,n}(\mathbb{P}^r,d):= \hat{\mathcal{Q}}_{1,n}(\PP^r,d)\setminus \mathcal{M}_{1,n}(\PP^r,d)$ receives a proper, surjective map from $\mathcal{C}_{1,n}\times_{\mathcal{M}_{1,n}}\hat{\mathcal{Q}}_{1,n}(\mathbb{P}^r,d-1)$.
\end{lemma}
\begin{proof}
  We identify $\mathcal{C}_{1,n}\cong \mathbb{P}_{\mathrm{Pic}_{1,n}^{1}}(\mathcal{P}_1).$ The map $\mathcal{P}_1\boxtimes \mathcal{P}_{d-1}^{\oplus r+1}\to \mathcal{P}_d^{\oplus r+1}$ over the multiplication map $\otimes: \mathrm{Pic}_{1,n}^1\times \mathrm{Pic}_{1,n}^{d-1}\to \mathrm{Pic}_{1,n}^{d}$ then projectivises to a map $\mathcal{C}_{1,n}\times_{\mathcal{M}_{1,n}}\hat{\mathcal{Q}}_{1,n}(\mathbb{P}^r,d-1)\to \hat{\mathcal{Q}}_{1,n}(\mathbb{P}^r,d).$ Concretely, the map is given by multiplying sections of line bundles. The above description shows that the map is projective, hence proper.

  By construction, the image of the map has basepoints, so it lands in $\mathcal{B}_{1,n}(\mathbb{P}^r,d).$ To see surjectivity, it suffices to work over the fibre over a marked point. Given a tuple of sections $[s_0:\dots:s_r]$ with basepoint divisor $B$, pick any degree one effective divisor $p\leq B$ together with a choice of setction $s_p$ (unique up to scalar multiplication), then $[s_0:\dots:s_r]$ is the image of $(p,[s_0/s_p:\dots: s_r/s_p]).$
\end{proof}

\begin{definition}
  Let $\mathrm{Pic}^{d_1,d_2}_{1,n}$ denote the fibre product $\mathrm{Pic}^{d_1}_{1,n}\times_{\mathcal{M}_{1,n}}\times \mathrm{Pic}^{d_2}_{1,n}$.
\end{definition}

The pure weight graded pieces of the Picard groups $\mathrm{Pic}^{d}_{1,n}$ and $\mathrm{Pic}^{d_1,d_2}_{1,n}$ can be related to those of $\mathcal{M}_{1,n}$ by adapting the methods of Canning--Larson--Payne \cite[§2.2]{clp}, which uses the partial compactification by the fibre power of the universal curve $\mathcal{E}\to \mathcal{M}_{1,1}.$

\begin{lemma}\label{lem:pics}
  \begin{enumerate}
    \item  The pure weight cohomology of $\mathrm{Pic}^d_{1,n}$ is given by \[W_\star H^\star(\mathrm{Pic}^d_{1,n})\cong W_\star H^\star(\mathcal{M}_{1,n+1})\oplus W_{\star-2} H^{\star-2}(\mathcal{M}_{1,n})(-1).\] The Tate twist is given by $\Theta\in H^2(\mathrm{Pic}^d_{1,n})$ as the image of $1\in H^0(\mathcal{M}_{1,n})(-1)\to H^2(\mathcal{M}_{1,n})$ under the Gysin pushforward; this is the class of a relative polarisation of $\mathrm{Pic}^d_{1,n}\to \mathcal{M}_{1,n}.$
    \item  There is an isomorphism \[W_\star H^\star(\mathrm{Pic}^{d_1,d_2}_{1,n})\cong  W_\star H^\star(\mathcal{M}_{1,n+2})\oplus  W_{\star-2} H^{\star-2}(\mathcal{M}_{1,n+1})(-1)^{\oplus 2}\oplus W_{\star-4} H^{\star-4}(\mathcal{M}_{1,n})(-2).\] Similar to (1), the Tate twists are given by relative polarisations $\Theta_1, \Theta_2\in H^2(\mathrm{Pic}^{d_1,d_2}_{1,n})$, and $\Theta_1\cup \Theta_2\in H^4(\mathrm{Pic}^{d_1,d_2}_{1,n}).$
  \end{enumerate}
\end{lemma}

The same technique allows us to relate the off-by-one cohomology of the Picard groups $\mathrm{Pic}^d_{1,n}$ to Petersen's interpretation of Getzler's relation in $\mathrm{gr}^W_4 H^3(\mathcal{M}_{1,4})$ \cite{pet}.

\begin{lemma}\label{lem:M1noff}
  For $k$ odd, $\mathrm{gr}^W_{k+1}H^k(\mathrm{Pic}^d_{1,n})\neq 0$ only when $k = 3$ and $n\geq 3$ in which case it is isomorphic to $\mathrm{gr}^W_{4}H^3(\mathcal{M}_{1,n+1})$ and when $k=5$ and $n\geq 4$ in which case it is isomorphic to $\mathrm{gr}^W_{4}H^3(\mathcal{M}_{1,n})(-1)$ under the Gysin pushforward along any section $\mathcal{M}_{1,n}\to \mathrm{Pic}^d_{1,n}.$
\end{lemma}

The class $H_{\hat{\mathcal{Q}}}\in H^2(\mathcal{M}_{1,n}(\mathbb{P}^r,d))$ comes from the projective bundle formula and records the hyperplane class in the linear system. We relate it to hyperplane classes pulled back from $\mathbb{P}^r$. 

\begin{lemma}
  When $n\geq 1$, $d\geq 2$, $H_{\hat{\mathcal{Q}}}\in H^2(\mathcal{M}_{1,n}(\mathbb{P}^r,d))$ agrees with $\mathrm{ev}_{i}^*H$ where $H\in H^2(\mathbb{P}^r)$ is a hyperplane class and $\mathrm{ev}_{i}: \mathcal{M}_{1,n}(\mathbb{P}^r,d)\to \mathbb{P}^r$ is any evaluation map.
\end{lemma}
\begin{proof}
  Let $\hat{\mathcal{C}}_{1,n}(\mathbb{P}^r,d)$ be the universal curve over $\hat{\mathcal{Q}}_{1,n}(\mathbb{P}^r,d).$ Namely, it is the fibre product \Cartesiansquare{\hat{\mathcal{C}}_{1,n}(\mathbb{P}^r,d)}{\hat{\mathcal{Q}}_{1,n}(\mathbb{P}^r,d)}{\mathcal{C}_{1,n}\times_{\mathcal{M}_{1,n}}\mathrm{Pic}^d_{1,n}}{\mathrm{Pic}^d_{1,n}}[\pi_{\hat{\mathcal{Q}}}][\pi_{{\mathcal{C}}}][][] Let $\mathcal{C}_{1,n}(\mathbb{P}^r,d)\to \mathcal{M}_{1,n}(\mathbb{P}^r,d)$ be the universal curve over $\mathcal{M}_{1,n}(\mathbb{P}^r,d),$ then $\mathcal{C}_{1,n}(\mathbb{P}^r,d)$ is open in $\hat{\mathcal{Q}}_{1,n}(\mathbb{P}^r,d)$ pulled back from the open immersion $\mathcal{M}_{1,n}(\mathbb{P}^r,d)\subset \hat{\mathcal{Q}}_{1,n}(\mathbb{P}^r,d).$ 
  
  Let $\mathrm{ev}: \mathcal{C}_{1,n}(\mathbb{P}^r,d)\to \mathbb{P}^r$ be the evaluation map, then $\mathrm{ev}^*\mathcal{O}_{\mathbb{P}^r}(1)\cong (\pi_{{\mathcal{C}}}^* \mathcal{P}_d)|_{\mathcal{C}_{1,n}(\mathbb{P}^r,d)}$ where recall that $\mathcal{P}_d$ is the Poincaré line bundle on $\mathcal{C}_{1,n}\times_{\mathcal{M}_{1,n}}\mathrm{Pic}^d_{1,n}.$ On the other hand, the line bundle $\mathcal{O}_{\pi_*\mathcal{P}_d^{\oplus r+1}}(1)$ on $\hat{\mathcal{Q}}_{1,n}(\mathbb{P}^r,d)$ satisfies that $\pi_{\hat{\mathcal{Q}}}\mathcal{O}_{\pi_*\mathcal{P}_d^{\oplus r+1}}(1)\cong (\pi_{{\mathcal{C}}}^* \mathcal{P}_d)|_{\mathcal{C}_{1,n}(\mathbb{P}^r,d)}.$ 
  Therefore, $$\pi_{\hat{\mathcal{Q}}}^* H_{\hat{\mathcal{Q}}}|_{\mathcal{C}_{1,n}(\mathbb{P}^r,d)} = \mathrm{ev}^* H$$ on $H^2(\mathcal{C}_{1,n}(\mathbb{P}^r,d)).$ We pull back both sides along the section $s_i: \mathcal{M}_{1,n}(\mathbb{P}^r,d)\to \mathcal{C}_{1,n}(\mathbb{P}^r,d)$ corresponding to the marked point $p_i$: they are $s_i^*(\pi_{\hat{\mathcal{Q}}}^* H_{\hat{\mathcal{Q}}}|_{\mathcal{C}_{1,n}(\mathbb{P}^r,d)}) = (\pi_{\hat{\mathcal{Q}}}\circ s_i)^* H_{\hat{\mathcal{Q}}} = H_{\hat{\mathcal{Q}}}$ and $s_i^* \mathrm{ev}^* H = \mathrm{ev}_i^*H.$ Hence $H_{\hat{\mathcal{Q}}} = \mathrm{ev}_i^*H.$
\end{proof}

Using the shorthand $\mathcal{B}$ for $\mathcal{B}_{1,n}(\mathbb{P}^r,d),$ we have the excision sequence in Borel--Moore homology: \[\cdots\to H^{\mathrm{BM}}_{k}(\mathcal{B}) \xrightarrow{\phi} H^{\mathrm{BM}}_{k} (\hat{\mathcal{Q}}_{1,n}(\mathbb{P}^r,d))\to  H^{\mathrm{BM}}_{k} (\mathcal{M}_{1,n}(\mathbb{P}^r,d))\to H^{\mathrm{BM}}_{k-1}(\mathcal{B})\xrightarrow{\varphi} H_{k-1}^{\mathrm{BM}}(\mathcal{Q}_{1,n}(\mathbb{P}^r,d))\to \cdots\]

We describe the pure and off-by-one weight graded pieces of $H_\star^{\mathrm{BM}}(\mathcal{M}_{1,n}(\PP^r,d))$ using weight truncations of the long exact sequence.

\begin{lemma}\label{lem:m1nrd} \label{lem:m1nrdoff}
  \begin{enumerate}
    \item The pure weight piece $\mathrm{gr}^W_\star H^\star(\mathcal{M}_{1,n}(\mathbb{P}^r,d))$ is given by $\mathrm{gr}^W_\star H^\star(\mathrm{Pic}^d_{1,n})\otimes H_{\hat{\mathcal{Q}}}^i$ for $ 0 \leq i\leq r.$
    \item The off-by-one weight graded piece fits in a short exact sequence \begin{equation} 0\to \frac{\mathrm{gr}^W_{-k+1}H_k^{\mathrm{BM}}(\hat{\mathcal{Q}}_{1,n}(\mathbb{P}^r,d))}{\mathrm{im}(\phi)}\to \mathrm{gr}^W_{-k+1} H_k^{\mathrm{BM}}(\mathcal{M}_{1,n}(\PP^r,d))\to \ker(\gr^W_{-k+1}\varphi)\to 0\end{equation} where \begin{enumerate}
    \item $\gr^{W}_{-\star+1}H_\star^{\mathrm{BM}}(\hat{\mathcal{Q}}_{1,n}(\PP^r,d))$ is generated by pullback from $\gr^{W}_{-\star+1}H_\star^{\mathrm{BM}}(\mathrm{Pic}^d_{1,n})$ and their cap products with $H_{\hat{\mathcal{Q}}},$
    \item $\ker(\gr^W_{-\star+1}\varphi)\subset \gr^W_{-\star+1} H_{{\star-1}}^{\mathrm{BM}}(\mathcal{B})$ is generated, under the surjection $$\gr^W_{-\star+1} H_{{\star-1}}^{\mathrm{BM}}(\mathcal{C}_{1,n}\times_{\mathcal{M}_{1,n}}\hat{\mathcal{Q}}_{1,n}(\PP^r,d-1))\to \gr^W_{-\star+1} H_{{\star-1}}^{\mathrm{BM}}(\mathcal{B})$$ by $\mathrm{gr}^W_{\star} H^\star(\hat{\mathcal{Q}}_{1,n+1}(\PP^r,d-1))\cap [\mathcal{C}_{1,n}\times_{\mathcal{M}_{1,n}}\hat{\mathcal{Q}}_{1,n}(\PP^r,d-1)].$
    \end{enumerate}
  \end{enumerate}
\end{lemma}

\begin{proof}
The multiplication map $\mathcal{C}_{1,n}\times_{\mathcal{M}_{1,n}} \hat{\mathcal{Q}}_{1,n}(\PP^r,d-1)\to \hat{\mathcal{Q}}_{1,n}(\PP^r,d)$ given (over $\mathcal{M}_{1,n}$) by $$(q, (L, [s_0,\dots,s_n]))\mapsto (L\otimes \mathcal{O}(q), [s_p s_0:\dots:s_p s_n])$$ factors through $$\mathcal{C}_{1,n}\times_{\mathcal{M}_{1,n}} \hat{\mathcal{Q}}_{1,n}(\PP^r,d-1)\to \mathcal{B}\to \hat{\mathcal{Q}}_{1,n}(\PP^r,d).$$ Lemma \ref{lem:propsurj} implies that the pushforward map $$\gr^W_{-\star+1} H_{{\star-1}}^{\mathrm{BM}}(\mathcal{C}_{1,n}\times_{\mathcal{M}_{1,n}}\hat{\mathcal{Q}}_{1,n}(\PP^r,d-1))\to \gr^W_{-\star+1} H_{{\star-1}}^{\mathrm{BM}}(\mathcal{B})$$ is surjective. We then have a surjection $\ker(\gr^W_{-k+1}\tilde m_*)\twoheadrightarrow\ker(\gr^W_{-k+1}\varphi)$ along the proper pushforward. To compute this, consider the following diagram, where $m': \mathcal{C}_{1,n}\times_{\mathcal{M}_{1,n}}\mathrm{Pic}^{d-1}_{1,n}\to \mathrm{Pic}^{d}_{1,n}.$

\[\begin{tikzcd}
	{\mathcal{C}_{1,n}\times_{\mathcal{M}_{1,n}} \hat{\mathcal{Q}}_{1,n}(\PP^r,d-1)} & {\hat{\mathcal{Q}}_{1,n}(\PP^r,d)} \\
	{\mathcal{C}_{1,n}\times_{\mathcal{M}_{1,n}}\mathrm{Pic}^{d-1}_{1,n}} & {\mathrm{Pic}^{d}_{1,n}}
	\arrow["m", from=1-1, to=1-2]
	\arrow["{f_1}"', from=1-1, to=2-1]
	\arrow["{f_2}", from=1-2, to=2-2]
	\arrow["{m'}"', from=2-1, to=2-2]
\end{tikzcd}\]

Let $\Theta\in H^2(\Pic^d_{1,n})$ be the relative polarisation over $\mathcal{M}_{1,n},$ and let $\Theta_1,\Theta_2\in H^2(\mathcal{C}_{1,n}\times_{\mathcal{M}_{1,n}}\Pic^{d-1}_{1,n})$ be the relative polarisations pulled back from $\Pic^{d-1}_{1,n}\to \mathcal{M}_{1,n}$ and $\mathcal{C}_{1,n}\to \mathcal{M}_{1,n},$ respectively. We have ${m'}^*(\Theta) = \Theta_1 + \Theta_2.$ Applying projection formula, we see that the pushforward map $\alpha\in \mathrm{gr}^W_{-\star} H_{\star}^{\mathrm{BM}}(\mathrm{Pic}^{d}_{1,n}),$ it maps $\Theta_2\cap (m')^*(\alpha)\mapsto \alpha,$ and $\Theta_1\cap (m')^*(\alpha)\mapsto -\alpha$ and vanishes otherwise. In particular, the kernel is isomorphic to $\mathrm{gr}^W_{\star} H^\star(\mathrm{Pic}^{d-1}_{1,n+1})\cap [\mathcal{C}_{1,n}\times_{\mathcal{M}_{1,n}}\mathrm{Pic}_{1,n}^{d-1}].$

Extending the map ${m'}^*$ by the projective bundle formula, the map $m_*$ maps $$H_{\hat{\mathcal{Q}}}^a\cap f_1^*(\Theta_2\cap (m')^*(\alpha))\mapsto H_{\hat{\mathcal{Q}}}^{a+(r+1)}\cap f_2^*(\alpha),$$ $$H_{\hat{\mathcal{Q}}}^a\cap f_1^*(\Theta_1\cap (m')^*(\alpha))\mapsto -H_{\hat{\mathcal{Q}}}^{a+(r+1)}\cap f_2^*(\alpha),$$
and vanishes otherwise. Items (1) and (2-b) follow from taking the kernel and cokernel of $m_*,$ and (2-a) follows from applying the projective bundle formula on $\hat{\mathcal{Q}}_{1,n}(\PP^r,d)\to \Pic^d_{1,n}.$
\end{proof}

As a special case of the above Lemma, we record the following vanishing result of off-by-one weight graded piece for later use.
\begin{lemma}\label{lem:H1wt2}
  $\gr^W_{2}H^1(\mathcal{M}_{1,n}(\mathbb{P}^r,d))=0.$
\end{lemma}
\begin{proof}
  Recall that $\dim \mathcal{M}_{1,n}(\mathbb{P}^r,d) = n + d(r+1),$ which we denote as $\mathrm{d}_{n,r,d}.$ By Poincaré duality this is equivalent to $$\gr^W_{-2(\mathrm{d}_{n,r,d}-1)}H^{\mathrm{BM}}_{2\mathrm{d}_{n,r,d}-1}(\mathcal{M}_{1,n}(\mathbb{P}^r,d))=0.$$

  We inspect the short exact sequence in Lemma \ref{lem:m1nrdoff}. When $k-1>2\dim \mathcal{B}_{1,n}(\mathbb{P}^r,d),$ the term $\mathrm{gr}^W_{-k+1}H^{\mathrm{BM}}_{k-1}(\mathcal{B}_{1,n}(\mathbb{P}^r,d))$ vanishes for dimension reasons. From Lemma \ref{lem:M1noff}, $H^k(\mathrm{Pic}_{1,n}^d)$ has off-by-one weight cohomology only when $k=3.$ Applying Poincaré duality and projective bundle formula, we see that $$\gr^W_{-2(\mathrm{d}_{n,r,d}-1)}H^{\mathrm{BM}}_{2\mathrm{d}_{n,r,d}-1}(\hat{\mathcal{Q}}_{1,n}(\mathbb{P}^r,d)) = 0,$$ hence $\ker(\gr^W_{-k+1}\varphi)=0$ as well. The vanishing result follows.
\end{proof}

\subsection{Maps from nodal elliptic curves}\label{sec:mapnod}
Let $\mathsf{C}_{k, (\mathbf{d}, \mathbf{m})}$ be a degree decorated $k$-cycle with multi-degree $\mathbf{d}$ and markings $\mathbf{m},$ and let $\mathcal{M}_{\mathsf{C}_{k, (\mathbf{d}, \mathbf{m})}}$ be the mapping space specified by $\mathsf{C}_{k, (\mathbf{d}, \mathbf{m})}.$

\begin{definition}
  Define $\mathsf{C}_{k, \mathbf{m}}$ as the underlying marked dual graph. Let $\mathfrak{M}_{\mathsf{C}_{k, \mathbf{m}}}\subset \mathfrak{M}_{1, |\mathbf{m}|}$ be the moduli of nodal curves of type $\mathsf{C}_{k, \mathbf{m}}$.  Recall that $$\mathfrak{M}_{\mathsf{C}_{k, \mathbf{m}}}\cong \left(\prod_{v\in V(\mathsf{C}_{k, \mathbf{m}})}\mathcal{M}_{0, \mathbf{m}(v)+2}\right)/\mathrm{Aut}(\mathsf{C}_{k, \mathbf{m}}),$$ so it has torus stabilisers along vertices $v\in V(\mathsf{C}_{k, \mathbf{m}})$ with $|\mathbf{m}(v)| = 0.$
  
  Let $\mathrm{Pic}^{\mathbf{d}}_{\mathsf{C}_{k, \mathbf{m}}}\to \mathfrak{M}_{\mathsf{C}_{k, \mathbf{m}}}$ be the universal Picard group with multi-degree $\mathbf{d},$ which is a $\mathbb{G}_m$-torsor. Let $\mathfrak{C}_{\mathsf{C}_{k, \mathbf{m}}}\to \mathfrak{M}_{\mathsf{C}_{k, \mathbf{m}}}$ be the universal curve, with $\mathcal{P}_{\mathbf{d}}$ the Poincaré line bundle over $\mathrm{Pic}^{\mathbf{d}}_{\mathsf{C}_{k, \mathbf{m}}}\times_{\mathfrak{M}_{\mathsf{C}_{k, \mathbf{m}}}} \mathfrak{C}_{\mathsf{C}_{k, \mathbf{m}}}.$ The projection map to the universal curve is denoted as $$\pi^{(\mathbf{d})}: \mathrm{Pic}^{\mathbf{d}}_{\mathsf{C}_{k, \mathbf{m}}}\times_{\mathfrak{M}_{\mathsf{C}_{k, \mathbf{m}}}} \mathfrak{C}_{\mathsf{C}_{k, \mathbf{m}}}\to \mathrm{Pic}^{\mathbf{d}}_{\mathsf{C}_{k, \mathbf{m}}}.$$ 
  
  Let $\hat{\mathcal{Q}}_{\mathsf{C}_{k, (\mathbf{d}, \mathbf{m})}}\to \mathrm{Pic}^{\mathbf{d}}_{\mathsf{C}_{k, \mathbf{m}}}$ be the projective bundle $\mathbb{P}(\pi^{(\mathbf{d})}_* \mathcal{P}_{\mathbf{d}}^{\oplus r+1}).$ As the case of $\mathcal{M}_{1,n}(\PP^r,d)\subset \hat{\mathcal{Q}}_{1,n}(\PP^r,d)$, the mapping space $\mathcal{M}_{\mathsf{C}_{k, (\mathbf{d}, \mathbf{m})}}$ may be presented as the basepoint free locus in $\hat{\mathcal{Q}}_{\mathsf{C}_{k, (\mathbf{d}, \mathbf{m})}}.$
\end{definition}

\begin{lemma}\label{lem:Piccyc}
  When $\mathbf{m}$ satisfies that $|\mathbf{m}(v)|\geq 1$ for all vertices $v$ on the cycle, the pure weight cohomology of $\mathrm{Pic}^{\mathbf{d}}_{\mathsf{C}_{k, \mathbf{m}}}$ is pulled back from the base $W_\star H^\star(\mathfrak{M}_{\mathsf{C}_{k, \mathbf{m}}})$, which is isomorphic to $$\left[\bigotimes_{\substack{v\in V(\mathsf{C}_{k, \mathbf{m}})\\ \mathbf{m}(v)=0}}H^\star\left(\mathsf{M}_{0,2}\right)\right]^{\mathrm{Aut}(\mathsf{C}_{k, \mathbf{m}})}$$ consisting of $\mathrm{Aut}(\mathsf{C}_{k, \mathbf{m}})$-invariant polynomials in the $\psi_v$-classes for bivalent vertices $v\in V(\mathsf{C}_{k, \mathbf{m}}).$ In particular, when all $\mathbf{m}(v)\geq 1,$ the pure weight cohomology of $\mathrm{Pic}^{\mathbf{d}}_{\mathsf{C}_{k, \mathbf{m}}}$ is given by $H^0\cong \mathbb{Q}.$
  
  Its off-by-one weight graded piece is generated as a $W_\star H^\star(\mathrm{Pic}^{\mathbf{d}}_{\mathsf{C}_{k, \mathbf{m}}})$-module by pullbacks from the WDVV relation in $H^1(\mathcal{M}_{0,4})$ and $$\left[\bigotimes_{\substack{v\in V(\mathsf{C}_{k, \mathbf{m}})\\ \mathbf{m}(v)=0}}H^\star\left(\mathsf{M}_{0,2}\right)\otimes H^1(\mathbb{G}_m)\right]^{\mathrm{Aut}(\mathsf{C}_{k, \mathbf{m}})}$$ where $H^1(\mathbb{G}_m)$ receives the sign representation of $\mathrm{Aut}(\mathsf{C}_{k, \mathbf{m}})\subset D_k.$
\end{lemma}
\begin{proof}
  Let $C$ be a cycle of rational curves. Its Picard group with multi-degree $\boldsymbol{d}$ is computed by the exact sequence $$0\to H^0(C, \mathcal{O}_C^*)\to H^0(C^\nu, \mathcal{O}_{C^\nu}^*)\to \bigoplus_{e\in E(C)}H^0(\mathcal{O}_{\nu_e}^*)\to \mathrm{Pic}^{\boldsymbol{d}}(C)\to 0,$$ which is isomorphic to $0\to \mathbb{C}^\star\xrightarrow{\Delta} (\mathbb{C}^\star)^{V(C)}\to (\mathbb{C}^\star)^{E(C)}\to \mathrm{Pic}^{\boldsymbol{d}}(C)\to 0.$

  On the other hand, $\mathrm{Aut}(C)$ fits in the exact sequence $1\to D\to \mathrm{Aut}(C)\to \prod_{v\in V(C)}\mathbb{C}^\star\to 1$ where $D$ is the dihedral group of the cycle. We may compute that the $\mathrm{Aut}(C)$-action on $\mathrm{Pic}^{\boldsymbol{d}}(C)$ factors through the sign representation of $\mathrm{Aut}(C)\to \mathbb{Z}/2 \mathbb{Z},$ and $\mathbb{Z}/2 \mathbb{Z}$ acts on $\mathrm{Pic}^{\boldsymbol{d}}(C)\cong \mathbb{C}^\star$ by inversion. A similar statement holds after assigning marked points to $C$.

  Denote $\pi: \mathrm{Pic}^{\mathbf{d}}_{\mathsf{C}_{k, \mathbf{m}}}\to \mathfrak{M}_{\mathsf{C}_{k, \mathbf{m}}}$ as the natural map. The action of $\mathrm{Aut}(C)$ described above determines the monodromy of $R\pi_* \mathbb{Q}$ as a local system on $\mathfrak{M}_{\mathsf{C}_{k, \mathbf{m}}}.$ The claimed result now follows from a sheaf cohomology computation.
  \end{proof}

We define the following auxiliary moduli space of pairs of line bundles to describe the basepoint locus $\hat{\mathcal{Q}}_{\mathsf{C}_{k, (\mathbf{d}, \mathbf{m})}}\setminus \mathcal{M}_{\mathsf{C}_{k, (\mathbf{d}, \mathbf{m})}}.$

\begin{definition}
  Let $\mathrm{Pic}^{\mathbf{d}_1, \mathbf{d}_2}_{\mathsf{C}_{k, \mathbf{m}}}\to \mathfrak{M}_{\mathsf{C}_{k, \mathbf{m}}}$ be the fibre product of the universal Picard group $\mathrm{Pic}^{\mathbf{d}_1}_{\mathsf{C}_{k, \mathbf{m}}}\to \mathfrak{M}_{\mathsf{C}_{k, \mathbf{m}}}$ and $\mathrm{Pic}^{\mathbf{d}_2}_{\mathsf{C}_{k, \mathbf{m}}}\to \mathfrak{M}_{\mathsf{C}_{k, \mathbf{m}}}$.
\end{definition}

Following the considerations analogous to the previous section, we have the following lemmas that lead to a description of the off-by-one weight graded piece of $H_c^\star(\mathcal{M}_{\mathsf{C}_{k, (\mathbf{d}, \mathbf{m})}}).$

\begin{lemma}
  The pure weight cohomology of $\mathrm{Pic}^{\mathbf{d}_1, \mathbf{d}_2}_{\mathsf{C}_{k, \mathbf{m}}}$ is pulled back from the base $W_\star H^\star(\mathfrak{M}_{\mathsf{C}_{k, \mathbf{m}}}).$\end{lemma}

\begin{lemma}\label{lem:mckdm}
  Let $D(\mathsf{C}_{k, (\mathbf{d}, \mathbf{m})})$ be the set of non-negative multi-degrees $\boldsymbol{\delta}$ on $\mathsf{C}_{k, \mathbf{m}}$ such that $\boldsymbol{\delta}\leq \mathbf{d}$ and $\sum_{v\in V(\mathsf{C}_{k, \mathbf{m}})}\boldsymbol{\delta}(v)=1.$ There is a proper, surjective morphism $$\bigsqcup_{\boldsymbol{\delta}\in D(\mathsf{C}_{k, (\mathbf{d}, \mathbf{m})})} \left(\mathbb{P}\pi^{(\boldsymbol{\delta})}_* \mathcal{P}_{\boldsymbol{\delta}}\boxtimes \hat{\mathcal{Q}}_{\mathsf{C}_{k, (\mathbf{d}-\boldsymbol{\delta}, \mathbf{m})}}\right)\to \hat{\mathcal{Q}}_{\mathsf{C}_{k, (\mathbf{d}, \mathbf{m})}}\setminus \mathcal{M}_{\mathsf{C}_{k, (\mathbf{d}, \mathbf{m})}},$$ where the box product is over $\mathrm{Pic}^{\boldsymbol{\delta}, \mathbf{d}- \boldsymbol{\delta}}_{\mathsf{C}_{k, \mathbf{m}}}$ and has fibres isomorphic to $\mathbb{P}^{(r+1)(d-1)-1}.$
\end{lemma}

\begin{corollary}\label{cor:Mcycle}
  The pure weight cohomology of $\mathcal{M}_{\mathsf{C}_{k, (\mathbf{d}, \mathbf{m})}}$ is generated by $H_{\hat{\mathcal{Q}}}^i,$ $d(r+1)-r+1\leq i\leq d(r+1),$ where $H_{\hat{\mathcal{Q}}}\in H^2(\mathcal{M}_{\mathsf{C}_{k, (\mathbf{d}, \mathbf{m})}})$ is the hyperplane class, and $W_\star H^\star(\mathfrak{M}_{\mathsf{C}_{k, \mathbf{m}}}).$ The off-by-one weight graded piece $\mathrm{gr}^W_{\star+1}H^\star(\mathcal{M}_{\mathsf{C}_{k, (\mathbf{d}, \mathbf{m})}})$ is generated as a $W_\star H^\star(\mathcal{M}_{\mathsf{C}_{k, (\mathbf{d}, \mathbf{m})}})$-module by $\mathrm{gr}^W_{\star + 1}H^\star(\mathrm{Pic}^{\mathbf{d}}_{\mathsf{C}_{k, \mathbf{m}}})$ and the pure weight basepoint classes $\mathrm{gr}^W_\star H^\star(\hat{\mathcal{Q}}_{\mathsf{C}_{k, (\mathbf{d}, \mathbf{m})}}\setminus \mathcal{M}_{\mathsf{C}_{k, (\mathbf{d}, \mathbf{m})}}),$ which receives a surjection $$\bigoplus_{\boldsymbol{\delta}\in D(\mathsf{C}_{k, (\mathbf{d}, \mathbf{m})})} \mathrm{gr}^W_\star H^\star(\mathbb{P}\pi^{(\boldsymbol{\delta})}_* \mathcal{P}_{\boldsymbol{\delta}}\boxtimes \hat{\mathcal{Q}}_{\mathsf{C}_{k, (\mathbf{d}-\boldsymbol{\delta}, \mathbf{m})}})\to \mathrm{gr}^W_\star H^\star(\hat{\mathcal{Q}}_{\mathsf{C}_{k, (\mathbf{d}, \mathbf{m})}}\setminus \mathcal{M}_{\mathsf{C}_{k, (\mathbf{d}, \mathbf{m})}}).$$
\end{corollary}

\section{Generators}
In this section, we use Künneth formula to assemble the results from the previous section and describe the pure weight graded pieces $\gr^{W}_{-\star}H_\star^{\mathrm{BM}}(\Mtil_{[\mathbf{G},\rho]}).$

\begin{theorem}\label{thm:gens}
  \begin{enumerate}
    \item When $[\mathbf{G},\rho]$ has a genus one vertex with positive degree, $\gr^{W}_{-\star}H_\star^{\mathrm{BM}}(\Mtil_{[\mathbf{G},\rho]})$ are $\mathrm{Aut}([\mathbf{G}])$-orbits of
    \begin{enumerate}[(a)]
      \item a class in\footnote{The notation $(L\cup L')(\mathbf{G}^c)$ is a short hand for $L(\mathbf{G}^c)\cup L'(\mathbf{G}^c)$, namely both the legs assigned to $\mathbf{G}^c$ and the legs connecting $\mathbf{G}^c$ to the remainder of the graph.} $\gr^{W}_{-\star}H_\star^{\mathrm{BM}}(\mathcal{M}_{1,(L\cup L')(\mathbf{G}^c)}(\mathbb{P}^r,\boldsymbol{\delta}_{\mathbf{G}^c})),$ which is generated by 
      \begin{enumerate}
        \item pullback from $\gr^{W}_{-\star}H_\star^{\mathrm{BM}}(\mathcal{M}_{1,(L\cup L')(\mathbf{G}^c)})$,
        \item their cap products  with $H_{\hat{\mathcal{Q}}}^i, \Theta$ for $d(r + 1) - r + 1 \leq  i \leq d(r + 1)$ (Lemma \ref{lem:m1nrd}), 
      \end{enumerate} 
      \item a class in $\prod_{\ell\in L'(\mathbf{G}^c)}\Mbar^*_{0, \boldsymbol{m}_{\ell}}(\mathbb{P}^r, \boldsymbol{\delta}_{\ell}).$
    \end{enumerate}
    \item When the core $[\mathbf{G}]$ is a cycle $\mathsf{C}_{k, (\mathbf{d}, \mathsf{m})}$ of genus zero vertices with positive total degree, $\gr^W_{-\star} H^{\mathrm{BM}}_\star(\mathcal{M}^{\mathrm{st}}_{[\mathbf{G}]})$ are $\mathrm{Aut}([\mathbf{G}])$-orbits of
    \begin{enumerate}
      \item a class in $\gr^{W}_{-\star}H_\star^{\mathrm{BM}}(\mathcal{M}_{\mathsf{C}_{k, (\mathbf{d}, \mathsf{m})}}),$ which is generated by \begin{enumerate}
        \item pullback from $\gr^{W}_{-\star}H_\star^{\mathrm{BM}}(\mathfrak{M}_{\mathsf{C}_{k, \mathbf{m}}}),$ generated by polynomials of $\psi$-classes of bivalent vertices (Lemma \ref{lem:mckdm}) up to $\mathrm{Aut}(\mathsf{C}_{k, \mathbf{m}})$-action, 
        \item their cap product with $H_{\hat{\mathcal{Q}}},$
      \end{enumerate}
        \item a class in $\prod_{\ell\in L'(\mathbf{G}^c)}\Mbar^*_{0, \boldsymbol{m}_{\ell}}(\mathbb{P}^r, \boldsymbol{\delta}_{\ell}).$
    \end{enumerate}
    The same formula holds for the stratum $\Mtil_{1,n}(\mathbb{P}^r,d),$ with $\mathrm{Aut}(\mathbf{G}, \rho)$ replacing $\mathrm{Aut}(\mathbf{G})$ in the case of $\Mtil_{[\mathbf{G},\rho]}.$
    \item When the core has total degree zero and is a genus one vertex, $\gr^{W}_{-\star}H_\star^{\mathrm{BM}}(\mathcal{M}^{\mathrm{st}}_{[\mathbf{G}]})$ are $\mathrm{Aut}([\mathbf{G}])$-orbits of tensor products of \begin{itemize}
      \item $\gr^{W}_{-\star}H_\star^{\mathrm{BM}}(\mathcal{M}_{1,(L\cup L')(\mathbf{G}^c)})\otimes H^\star(\mathbb{P}^r),$ 
      \item a class in $\prod_{\ell\in L'(\mathbf{G}^c)}\Mbar^*_{0, \boldsymbol{m}_{\ell}}(\mathbb{P}^r, \boldsymbol{\delta}_{\ell}).$
    \end{itemize}
    The formula for $\gr^{W}_{-\star}H_\star^{\mathrm{BM}}(\Mtil_{[\mathbf{G},\rho]})$ is the orbits of tensor products of the above pulled back classes capped with $\psi_v^{i_v}, 1\leq i_v\leq r-1,$ where $v$ runs over all bivalent vertices in $L'([\mathbf{G}^{cc}])$ (Lemma \ref{lem:MF}).

    \item When the core $[\mathbf{G}]$ is a cycle $\mathsf{C}_{k, (\mathbf{d}, \mathsf{m})}$ of genus zero vertices with zero total degree, the pure weight Borel--Moore homology groups of the strata $\mathcal{M}^{\mathrm{st}}_{[\mathbf{G}]}$ and $\Mtil_{[\mathbf{G}]}$ are spanned by $\mathrm{Aut}(\mathbf{G})$-orbits of tensor products of all the generators in the previous item apart from $\gr^{W}_{-\star}H_\star^{\mathrm{BM}}(\mathcal{M}_{1,(L\cup L')(\mathbf{G}^c)}).$
  \end{enumerate}
\end{theorem}

\begin{corollary}[Theorem \ref{thm:maingen}]\label{cor:maingen}
  With the same notation as above, there is a surjection $$\bigoplus_{[\mathbf{G},\rho]}\gr^{W}_{-\star}H_\star^{\mathrm{BM}}(\Mtil_{[\mathbf{G},\rho]}) \twoheadrightarrow H_\star(\Mtil_{1,n}(\mathbb{P}^r,d)),$$ and the pure weight Borel--Moore homology groups of the strata are spanned by the list of classes from Theorem \ref{thm:gens}.
\end{corollary}

\begin{remark}
  An analogous surjection $$\bigoplus_{[\mathbf{G}]}\gr^{W}_{-\star}H_\star^{\mathrm{BM}}(\mathcal{M}_{[\mathbf{G}]}) \twoheadrightarrow \mathrm{gr}^W_{-\star} H_\star(\Mbar_{1,n}(\mathbb{P}^r,d)),$$ holds for the pure weight homology group of the stable maps space $\Mbar_{1,n}(\PP^r,d).$
\end{remark}

\subsection{Classes on strata closure}\label{subsec:strcl}

We relate the surjection to homology groups of the strata closures $\Mtil_{[\mathbf{G},\rho]}\subset \Mbar_{[\mathbf{G},\rho]}$. The spectral sequence is induced by the filtration $$\varnothing = X_{-1}\subset \cdots \subset X_p\subset X_{p+1}\subset\cdots\subset X_{\dim \mathcal{M}_{1,n}(\PP^r,d)} = \Mtil_{1,n}(\PP^r,d),$$ where $X_p$ is closed substack given by the union of strata with dimension less than or equal to $p.$ By construction, 
 \begin{align*}
 E_{p,q}^{\infty} & = \frac{\mathrm{im}(H_{p+q}(X_p)\to H_{p+q}(\Mtil_{1,n}(\PP^r,d)))}{\mathrm{im}(H_{p+q}(X_{p-1})\to H_{p+q}(\Mtil_{1,n}(\PP^r,d)))}\\ &  = \frac{\mathrm{im}\left(\bigoplus_{\dim \Mtil_{[\mathbf{G},\rho]} = p}H_{p+q}(\Mbar_{[\mathbf{G},\rho]})\to H_{p+q}(\Mtil_{1,n}(\PP^r,d))\right)}{\mathrm{im}\left(\bigoplus_{\dim \Mtil_{[\mathbf{G}',\rho']} = p-1}H_{p+q}(\Mbar_{[\mathbf{G}',\rho']})\to H_{p+q}(\Mtil_{1,n}(\PP^r,d))\right)}
 \end{align*}
The second equality uses the fact that $\bigsqcup_{\dim \Mtil_{[\mathbf{G},\rho]} = p} \Mbar_{[\mathbf{G},\rho]}\to X_p$ is proper and surjective and hence induces a surjective pushforward map of pure weight homology groups. 

Combining the above with the surjection $\bigoplus_{\dim \Mtil_{[\mathbf{G},\rho]} = p}\gr^W_{-(p+q)}H_{p+q}^{\mathrm{BM}}(\Mtil_{[\mathbf{G},\rho]})\twoheadrightarrow E^{\infty}_{p,q},$ we see that lifts of $\gr^W_{-\star}H_{\star}^{\mathrm{BM}}(\Mtil_{[\mathbf{G},\rho]})$ along the surjection from $H_\star(\Mbar_{[\mathbf{G},\rho]})$ generate $H_{\star}(\Mtil_{1,n}(\PP^r,d)),$ leading to the statement in Theorem \ref{thm:maingen}.

While there is no canonical way to lift pure weight classes on $\Mtil_{[\mathbf{G},\rho]}$ to $\Mbar_{[\mathbf{G},\rho]}$, the modular interpretation of the pure weight classes points to convenient choices of lifts:
\begin{enumerate}
  \item The pure weight cohomology classes on $\mathcal{M}_{1,n}$ are the fundamental class and pullback from $\mathsf{S}_{k+1}=\gr^W_{k} H^k(\mathcal{M}_{1,k}),$ and it is known that $H^k(\Mbar_{1,k})\to \gr^W_{k} H^k(\mathcal{M}_{1,k})$ is an isomorphism.
  \item Identifying $\mathrm{Pic}_{1,n'}^d$ with the universal curve over $\mathcal{M}_{1,n'},$ the relative polarisation can be lifted to the relative tangent bundle of the universal curve over $\Mbar_{1,n'}$ or $\Mbar_{1,n'}$ and pulled back to strata closures.
  \item When $r\geq 2,$ fixing some codimension 2 linear subspace $L\subset \mathbb{P}^r,$ let $\mathcal{H}$ be the divisor in $\Mtil_{1,n}(\PP^r,d)$ consisting of maps that meet the $L.$ The same argument from \cite[Lemma 1.1.1]{Pandharipande1999} implies that when restricted to $\mathcal{M}_{1,n}(\PP^r,d),$ the class $\mathcal{H}$ is equal to $\delta_r\cdot H_{\hat{\mathcal{Q}}}$ for some $\delta_r\in \mathbb{Z}_{>0}.$ Hence, $\mathcal{H}/\delta_r$ is a lift of $H_{\hat{\mathcal{Q}}}.$
\end{enumerate}

In short, the homology classes in $\Mtil_{1,n}(\PP^r,d)$ are either tautological classes, boundary strata, or (dual to) the cusp forms on $\mathcal{M}_{1,n'}.$ We now use this to deduce qualitative properties of the (co)homology groups of $\Mtil_{1,n}(\PP^r,d)$.

\subsection{Hodge and Tate conjectures}
The description of the generators allows us to verify the Hodge and Tate conjectures on the Vakil--Zinger space. We recall that the stack $\Mtil_{1,n}(\PP^r,d)$ is smooth and proper over $\mathrm{Spec}(\mathbb{Z}).$
\begin{corollary}\label{cor:HT}
  \begin{enumerate}
    \item The cycle class map $A^\star(\Mtil_{1,n}(\PP^r,d))_{\mathbb{Q}}\to H^{2\star}(\Mtil_{1,n}(\PP^r,d))_{\mathbb{Q}}$ is surjective.
    \item Let $K$ be a finite extension of $\mathbb{Q}$ or a finite field, and let $\Mtil_{1,n}(\PP^r,d)_{\overline{K}}$ be the base change of $\Mtil_{1,n}(\PP^r,d)$ to $\mathrm{Spec}(\overline{K})$, and similar for $\Mtil_{1,n}(\PP^r,d)_{{K}}.$ For any $\ell$ invertible in $K,$ the cycle class map $A^\star(\Mtil_{1,n}(\PP^r,d)_{{K}})_{\mathbb{Q}_{\ell}}\to H^{2\star}(\Mtil_{1,n}(\PP^r,d)_{\overline{K}}, \mathbb{Q}_\ell(\star))^{\mathrm{Gal}(\overline{K}/K)}$ is surjective.
  \end{enumerate}
\end{corollary}
\begin{proof}
  In even degrees, all the generators given by Theorem \ref{thm:gens} are cap products of tautological classes with fundamental classes of strata. Lifting these tautological classes to their strata closures in \ref{subsec:strcl} and pushing forward, we see that the even homology groups of $\Mtil_{1,n}(\PP^r,d)$ are generated by algebraic cycles, which proves item (1). We also observe that the algebraic cycles corresponding to tautological classes are defined over $K.$
   
  Suppose $K$ is a number field, and fix embeddings $\sigma: K\to \mathbb{C}$ and $\overline{\sigma}\to \mathbb{C}.$ Proper base change \cite[Exp. XII]{SGA43} gives a comparison isomorphism $$H^{2\star}(\Mtil_{1,n}(\PP^r,d))\otimes_{\mathbb{Q}}\mathbb{Q}_\ell(\star)\cong H^{2\star}(\Mtil_{1,n}(\PP^r,d)_K, \mathbb{Q}_\ell(\star))$$ between the Betti cohomology of the complex analytic space $\Mtil_{1,n}(\PP^r,d)(\mathbb{C})$ on the left hand side and the $\ell$-adic étale cohomology on the right hand side. Under the isomorphism, complex conjugation\footnote{For simplicity, we assume the embedding $\sigma$ is a real embedding. Analogous statements hold for the general case.} action on the left hand side corresponds to $\tilde{c}\in \mathrm{Gal}(\overline{K}/K)$ on the right hand side. As all the even Betti cohomology groups are Hodge--Tate, the comparison isomorphism gives $$H^{2\star}(\Mtil_{1,n}(\PP^r,d)_K, \mathbb{Q}_\ell(\star))^{\mathrm{Gal}(\overline{K}/K)}\subseteq H^{2\star}(\Mtil_{1,n}(\PP^r,d)_K, \mathbb{Q}_\ell(\star))^{\langle\tilde{c}\rangle}\cong H^{2\star}(\Mtil_{1,n}(\PP^r,d))\otimes_{\mathbb{Q}}\mathbb{Q}_\ell(\star).$$ Since the right hand side receives a surjection from tautological classes that are defined over $K,$ item (2) holds. The same argument with the Frobenius automorphism in place $\tilde{c}\in \mathrm{Gal}(\overline{K}/K)$ proves the case when $K$ is a finite field.
\end{proof}

\subsection{Degrees with odd cohomology}\label{sec:degodd}

The description of the generators controls the Hodge structures on $H^\star(\Mtil_{1,n}(\mathbb{P}^r,d))$ and in turn implies vanishing results of cohomology groups.

In the following, let $\mathsf{L} = H^2(\mathbb{P}^1)$ be the Tate Hodge structure, and let $W_k H^k(\mathcal{M}_{1,k})=: \mathsf{S}_{k+1}$ be the weight $k$ Hodge structure corresponding to $\mathrm{SL}_{2}(\mathbb{Z})$-cusp forms of weight $k+1$ under the Eichler--Shimura isomorphism.

\begin{corollary}
  The Hodge structures present in $H^\star(\Mtil_{1,n}(\mathbb{P}^r,d))$ are products of $\mathsf{L}$ and $\mathsf{S}_{k+1}.$
\end{corollary}
\begin{proof}
  The generators in Theorem \ref{thm:gens} are either Chern classes--with Tate Hodge structure--or pulled back from $W_\star H^\star(\mathcal{M}_{1,n}).$ It has been proven in \cite[§2]{clp} that each $W_k H^k(\mathcal{M}_{1,n})$ is generated by pullback of the Hodge structures $W_k H^k(\mathcal{M}_{1,k})\to W_k H^k(\mathcal{M}_{1,n})$ along the forgetful maps $\mathcal{M}_{1,n}\to \mathcal{M}_{1,k}.$ Therefore, each $W_\star H^\star(\mathcal{M}_{1,n})$ has Hodge structure $\mathsf{S}_{k+1}.$
\end{proof}

As the smallest weight for the first non-zero  $\mathrm{SL}_2(\mathbb{Z})$-cusp form is 12, $\mathsf{S}_{k+1}$ is non-zero for the first time when $k = 11.$ This recovers the following vanishing result by Fontanari \cite{Fontanari}.
\begin{corollary}\label{cor:oddvan}
  For odd $k<11$ and all $n,d$, we have $H^\star(\Mtil_{1,n}(\mathbb{P}^r,d)) = 0.$
\end{corollary}
\begin{proof}
  As $\Mtil_{1,n}(\mathbb{P}^r,d)$ is smooth and proper, $H^k(\Mtil_{1,n}(\mathbb{P}^r,d))$ carries a Hodge structure of weight $k.$ On the other hand, the above implies that the first odd weight Hodge structure that could be present on $H^\star(\Mtil_{1,n}(\mathbb{P}^r,d))$ is $\mathsf{S}_{12},$ which is off weight 11. Therefore, no odd cohomology groups in degree less than 11 can be non-vanishing.
\end{proof}

On the other hand, when $n\geq 11,$ the pullback from $\mathcal{M}_{1,11}$ gives $H^{11}(\Mtil_{1,n}(\mathbb{P}^r,d))\neq 0,$ for instance. Taking a step further, we fix $n=0$ and ask for the smallest smallest $d$ such that $\Mtil_{1,0}(\mathbb{P}^r,d)$ has odd cohomology.

\begin{corollary}
  When $r\geq 11,$ the smallest $d$ such that $\Mtil_{1,0}(\mathbb{P}^r,d)$ has odd cohomology is $d=11.$ 
\end{corollary}
\begin{proof}
  By considering pullback from $\mathcal{M}_{1,11}$ to strata in $\Mtil_{1,0}(\mathbb{P}^r,d)$, we see that $d=11$ is a lower bound. It suffices to show that the following centrally aligned type in $\Mtil_{1,0}(\mathbb{P}^r,11)$ has non-vanishing odd cohomology when $r\geq 10.$

  $$
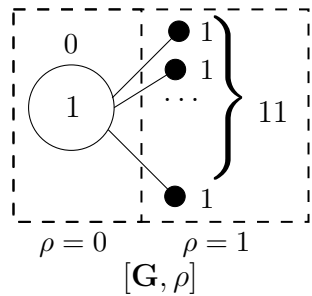
\begin{tikzpicture}[x=0.75pt,y=0.75pt,yscale=-1,xscale=1]

\draw   (25.86,65.67) .. controls (25.86,53.76) and (35.52,44.1) .. (47.43,44.1) .. controls (59.34,44.1) and (69,53.76) .. (69,65.67) .. controls (69,77.59) and (59.34,87.24) .. (47.43,87.24) .. controls (35.52,87.24) and (25.86,77.59) .. (25.86,65.67) -- cycle ;
\draw    (66,76.77) -- (99.68,110.45) ;
\draw  [fill={rgb, 255:red, 0; green, 0; blue, 0 }  ,fill opacity=1 ] (94.52,110.45) .. controls (94.52,107.6) and (96.83,105.29) .. (99.68,105.29) .. controls (102.53,105.29) and (104.84,107.6) .. (104.84,110.45) .. controls (104.84,113.3) and (102.53,115.61) .. (99.68,115.61) .. controls (96.83,115.61) and (94.52,113.3) .. (94.52,110.45) -- cycle ;
\draw    (68,60.77) -- (101.68,27.09) ;
\draw  [fill={rgb, 255:red, 0; green, 0; blue, 0 }  ,fill opacity=1 ] (96.52,27.09) .. controls (96.52,24.24) and (98.83,21.93) .. (101.68,21.93) .. controls (104.53,21.93) and (106.84,24.24) .. (106.84,27.09) .. controls (106.84,29.94) and (104.53,32.25) .. (101.68,32.25) .. controls (98.83,32.25) and (96.52,29.94) .. (96.52,27.09) -- cycle ;
\draw    (69,65.67) -- (99.92,46.43) ;
\draw  [fill={rgb, 255:red, 0; green, 0; blue, 0 }  ,fill opacity=1 ] (94.76,46.43) .. controls (94.76,43.58) and (97.07,41.27) .. (99.92,41.27) .. controls (102.76,41.27) and (105.07,43.58) .. (105.07,46.43) .. controls (105.07,49.28) and (102.76,51.59) .. (99.92,51.59) .. controls (97.07,51.59) and (94.76,49.28) .. (94.76,46.43) -- cycle ;
\draw  [dash pattern={on 4.5pt off 4.5pt}][line width=0.75]  (18.25,16.08) -- (82.79,16.08) -- (82.79,123.32) -- (18.25,123.32) -- cycle ;
\draw  [dash pattern={on 4.5pt off 4.5pt}][line width=0.75]  (18.25,16.08) -- (167.04,16.08) -- (167.04,123.32) -- (18.25,123.32) -- cycle ;

\draw (43,59) node [anchor=north west][inner sep=0.75pt]    {$1$};
\draw (42.33,27.34) node [anchor=north west][inner sep=0.75pt]  [font=\small]  {$0$};
\draw (110.93,21.14) node [anchor=north west][inner sep=0.75pt]  [font=\small]  {$1$};
\draw (72,142.38) node [anchor=north west][inner sep=0.75pt]    {$[\mathbf{G} ,\rho]$};
\draw (110.93,105.4) node [anchor=north west][inner sep=0.75pt]  [font=\small]  {$1$};
\draw (110.93,40.4) node [anchor=north west][inner sep=0.75pt]  [font=\small]  {$1$};
\draw (92.46,59.45) node [anchor=north west][inner sep=0.75pt]    {$\dotsc $};
\draw (84.77,17.82) node [anchor=north west][inner sep=0.75pt]  [font=\LARGE,rotate=-360]  {$\begin{drcases}
 & \\
 & 
\end{drcases}$};
\draw (140,61.4) node [anchor=north west][inner sep=0.75pt]    {$11$};
\draw (30,127.2) node [anchor=north west][inner sep=0.75pt]  [font=\small]  {$\rho =0$};
\draw (102.15,127.38) node [anchor=north west][inner sep=0.75pt]  [font=\small]  {$\rho =1$};

\end{tikzpicture}$$ From Theorem \ref{thm:gens}, the pure weight cohomology contribution from the stratum is $$\left[W_\star H^\star(\mathcal{M}_{1,11})\otimes \bigotimes_{j=1}^{11}\left(\bigoplus_{i=0}^{r-1} \psi_{v_j}^{i}\right)\right]^{S_{11}},$$ where $S_{11}\cong \mathrm{Aut}([\mathbf{G},\rho]).$ The odd cohomology all comes from setting $\star=11$. 

As an $S_{11}$-representation, $W_{11} H^{11}(\mathcal{M}_{1,11})\cong \mathrm{sgn}_{S_{11}}\otimes \mathsf{S}_{12},$ \cite[§5]{getzresolv}. Thus the odd cohomology simplifies to $\mathsf{S}_{12}\otimes \bigwedge^{11}\left(\bigoplus_{i = 0}^{r-1}\psi^i \right),$ which is non-zero if and only if $\dim \left(\bigoplus_{i = 0}^{r-1}\psi^i\right)\geq 11,$ namely when $r\geq 11.$ When $r = 11,$ the corresponding cohomological degree on the stratum is $121 = 11 + 2\cdot \binom{11}{2},$ which maps to $H^{123}(\Mtil_{1,0}(\mathbb{P}^r,11))$ under the Gysin pushforward. Because the odd degree generators are of pure weight carry Hodge structures $\mathsf{S}_{12},$ the differentials to and from them in the stratification spectral sequence for $H^\star(\Mtil_{1,0}(\mathbb{{P}}^r,11))$ all vanish, so the generators survive to $H^\star(\Mtil_{1,0}(\mathbb{{P}}^r,11)).$
\end{proof}

\begin{example}
  The smallest odd cohomological degree $k$ such that $H^k(\Mtil_{1,0}(\mathbb{P}^r,d))\neq 0$ is realized as $k = 13$ and $d = \binom{12}{2} = 66.$ The relevant dual graph stratum is given by $$
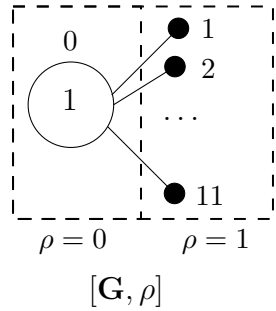
\begin{tikzpicture}[x=0.75pt,y=0.75pt,yscale=-1,xscale=1]

\draw   (25.86,65.67) .. controls (25.86,53.76) and (35.52,44.1) .. (47.43,44.1) .. controls (59.34,44.1) and (69,53.76) .. (69,65.67) .. controls (69,77.59) and (59.34,87.24) .. (47.43,87.24) .. controls (35.52,87.24) and (25.86,77.59) .. (25.86,65.67) -- cycle ;
\draw    (66,76.77) -- (99.68,110.45) ;
\draw  [fill={rgb, 255:red, 0; green, 0; blue, 0 }  ,fill opacity=1 ] (94.52,110.45) .. controls (94.52,107.6) and (96.83,105.29) .. (99.68,105.29) .. controls (102.53,105.29) and (104.84,107.6) .. (104.84,110.45) .. controls (104.84,113.3) and (102.53,115.61) .. (99.68,115.61) .. controls (96.83,115.61) and (94.52,113.3) .. (94.52,110.45) -- cycle ;
\draw    (68,60.77) -- (101.68,27.09) ;
\draw  [fill={rgb, 255:red, 0; green, 0; blue, 0 }  ,fill opacity=1 ] (96.52,27.09) .. controls (96.52,24.24) and (98.83,21.93) .. (101.68,21.93) .. controls (104.53,21.93) and (106.84,24.24) .. (106.84,27.09) .. controls (106.84,29.94) and (104.53,32.25) .. (101.68,32.25) .. controls (98.83,32.25) and (96.52,29.94) .. (96.52,27.09) -- cycle ;
\draw    (69,65.67) -- (99.92,46.43) ;
\draw  [fill={rgb, 255:red, 0; green, 0; blue, 0 }  ,fill opacity=1 ] (94.76,46.43) .. controls (94.76,43.58) and (97.07,41.27) .. (99.92,41.27) .. controls (102.76,41.27) and (105.07,43.58) .. (105.07,46.43) .. controls (105.07,49.28) and (102.76,51.59) .. (99.92,51.59) .. controls (97.07,51.59) and (94.76,49.28) .. (94.76,46.43) -- cycle ;
\draw  [dash pattern={on 4.5pt off 4.5pt}][line width=0.75]  (18.25,16.08) -- (82.79,16.08) -- (82.79,123.32) -- (18.25,123.32) -- cycle ;
\draw  [dash pattern={on 4.5pt off 4.5pt}][line width=0.75]  (18.25,16.08) -- (149.34,16.08) -- (149.34,123.32) -- (18.25,123.32) -- cycle ;

\draw (41.79,56.98) node [anchor=north west][inner sep=0.75pt]    {$1$};
\draw (42.33,27.34) node [anchor=north west][inner sep=0.75pt]  [font=\small]  {$0$};
\draw (111.93,21.14) node [anchor=north west][inner sep=0.75pt]  [font=\small]  {$1$};
\draw (55,150.38) node [anchor=north west][inner sep=0.75pt]    {$[\mathbf{G} ,\rho ]$};
\draw (108.93,105.4) node [anchor=north west][inner sep=0.75pt]  [font=\small]  {$11$};
\draw (111.93,41.4) node [anchor=north west][inner sep=0.75pt]  [font=\small]  {$2$};
\draw (92.46,70.45) node [anchor=north west][inner sep=0.75pt]    {$\dots$};
\draw (30,127.2) node [anchor=north west][inner sep=0.75pt]  [font=\small]  {$\rho =0$};
\draw (102.15,127.38) node [anchor=north west][inner sep=0.75pt]  [font=\small]  {$\rho =1$};

\end{tikzpicture}$$ which has trivial automorphism group. Therefore, the stratum contributes non-vanishing $H^{11},$ which survives in the spectral sequence and Gysin pushes forward to $H^{13}(\Mtil_{1,0}(\mathbb{P}^r, 66)).$ On the other hand, we already know that the minimum $k$ needs to be $k\geq 11$, but $\mathrm{gr}^W_{11} H^{11}(\mathcal{M}_{1,0}(\mathbb{P}^r,d))$ vanishes for all $r, d.$ Therefore the smallest such $k$ is $k=13.$
\end{example}

\subsection{Rational Picard group}
As a prelude to the next section, we determine $H^2(\Mtil_{1,n}(\PP^r,d))$ and show that it is isomorphic to the rational Picard group.

For this section, we set $\mathrm{D}  = \dim \Mtil_{1,n}(\mathbb{P}^r,d).$ After dualising to $H^{\mathrm{BM}}_{2\mathrm{D}-2}(\Mtil_{1,n}(\mathbb{P}^r,d))$ and inspecting the stratification spectral sequence, we observe that $H^2(\Mtil_{1,n}(\mathbb{P}^r,d))$ is generated by the classes $\Theta, H_{\hat{\mathcal{Q}}}$ on $\mathcal{M}_{1,n}(\mathbb{P}^r,d)$ (Lemma \ref{lem:m1nrd}) and the fundamental classes of the boundary divisors.
\begin{proposition}[Corollary \ref{cor:F}]\label{prop:Pic}
  The classes listed above form a basis of $H^2(\Mtil_{1,n}(\mathbb{P}^r,d)),$ and the cycle class map $A^1_{\mathbb{Q}}(\Mtil_{1,n}(\mathbb{P}^r,d))\to H^2(\Mtil_{1,n}(\mathbb{P}^r,d))$ is an isomorphism.
\end{proposition}
\begin{proof}
  We calculate $H^{\mathrm{BM}}_{2\mathrm{D}-2}(\Mtil_{1,n}(\mathbb{P}^r,d))$ via the stratification spectral sequence. 
  
  For dimension reasons, the only $E_1$ pages that can possibly contribute are $\gr^{W}_{-2(\mathrm{D}-1)} E^1_{\mathrm{D}-1, \mathrm{D}-1}$ and $\gr^{W}_{-2(\mathrm{D}-1)} E_1^{\mathrm{D}, \mathrm{D}-2} = H^{\mathrm{BM}}_{2(\mathrm{D}-1)}(\mathcal{M}_{1,n}(\mathbb{P}^r,d)).$ By considering the cohomological degree and weight, the latter survives to $E_\infty$-page as a subquotient of $H_c^{2(\mathrm{D}-1)}(\Mtil_{1,n}(\mathbb{P}^r,d)).$
  
  The page $E_1^{\mathrm{D}-1, \mathrm{D}-1}$ is given by the direct sum of the fundamental classes of the boundary stratum and is hence pure of weight $2(\mathrm{D}-1)$. For weight reasons, the only possibly non-zero differential will be $d_1: E^1_{\mathrm{D}-1, \mathrm{D}-1}\to E^1_{\mathrm{D}, \mathrm{D}-1},$ and $d^1$ is non-zero if and only if the composition $$E^1_{\mathrm{D}, \mathrm{D}-1}\to E^1_{\mathrm{D}-1, \mathrm{D}-1}\to \mathrm{gr}^W_{2(\mathrm{D}-1)}E_1^{\mathrm{D}, \mathrm{D}-1}=\mathrm{gr}^W_{2(\mathrm{D}-1)} H_c^{2\mathrm{D}-1}(\mathcal{M}_{1,n}(\mathbb{P}^r,d))$$ is non-zero. However, the latter is dual to $\mathrm{gr}^W_{2} H^{2}(\mathcal{M}_{1,n}(\mathbb{P}^r,d)),$ which vanishes from Lemma \ref{lem:H1wt2}. Therefore, the whole page $E_1^{\mathrm{D}-1, \mathrm{D}-1}$ survives to $E_\infty$ as a subquotient of $H_c^{2(\mathrm{D}-1)}(\Mtil_{1,n}(\mathbb{P}^r,d)).$ The basis of $H^{2}(\Mtil_{1,n}(\mathbb{P}^r,d))$ follows from combining the two subquotients of $H_c^{2(\mathrm{D}-1)}(\Mtil_{1,n}(\mathbb{P}^r,d))$ and dualising. 
  
  We turn to the Picard group of the mapping spaces. Recall that\footnote{This follows from the fact that $A^1(\mathcal{M}_{1,n}) = 0$ and $\mathrm{Pic}^d_{1,n}\cong \mathcal{C}_{1,n},$ so we use the excision sequence $\bigoplus_{n}A_\star(\mathcal{M}_{1,n})\to A_\star(\mathcal{C}_{1,n})\to A_\star(\mathcal{M}_{1,n+1})\to 0.$ Each $A^0(\mathcal{M}_{1,n})$ pushes forward to $\Theta,$ which is non-zero as it is mapped to a non-zero class in the cycle class map.} $A^1_{\mathbb{Q}}(\mathrm{Pic}^d_{1,n})=\langle \Theta\rangle,$ then by projective bundle formula on $\hat{\mathcal{Q}}_{1,n}(\mathbb{P}^r,d)\to \mathrm{Pic}^d_{1,n},$ we have that $A_{\mathrm{d}_{n,r,d}-1}(\hat{\mathcal{Q}}_{1,n}(\mathbb{P}^r,d))\cong \langle \Theta, H_{\hat{Q}}\rangle.$ Therefore, by excision sequences for the open embedding $\mathcal{M}_{1,n}(\mathbb{P}^r,d)\subset \hat{\mathcal{Q}}_{1,n}(\mathbb{P}^r,d)$ as well as $\mathcal{M}_{1,n}(\mathbb{P}^r,d)\subset \Mtil_{1,n}(\mathbb{P}^r,d),$ we have that $A^1_{\mathbb{Q}}(\Mtil_{1,n}(\mathbb{P}^r,d))$ is generated by the boundary divisors, $H$ and $\Theta.$

  Let $\mathcal{B}$ be the vector space freely generated by the set of boundary divisors, $H_{\hat{Q}}$, and $\Theta.$ The previous result on the basis of $H^2(\Mtil_{1,n}(\mathbb{P}^r,d))$ gives an isomorphism $H^2(\Mtil_{1,n}(\mathbb{P}^r,d))\cong \mathcal{B},$ then the cycle class map composes to the identity $\mathrm{id}_{\mathcal{B}},$ $$\mathcal{B}\twoheadrightarrow A^1_{\mathbb{Q}}(\Mtil_{1,n}(\mathbb{P}^r,d))\to H^2(\Mtil_{1,n}(\mathbb{P}^r,d))\cong \mathcal{B},$$ hence the cycle class map is an isomorphism.
\end{proof}

\section{Relations}
Relations among the generators correspond to images of the differentials $\gr^W_{-(p+q)}E^r_{p+r, q-r+1}\to \gr^W_{-(p+q)}E^r_{p,q}.$ They are direct sums of boundary maps $$\gr^W_{-(p+q)}H^{\mathrm{BM}}_{p+q+1}(\mathcal{M}_{[\mathbf{G},\rho]})\to \gr^W_{-(p+q)}H^{\mathrm{BM}}_{p+q}(\mathcal{M}_{[\mathbf{G}',\rho']})$$ or their subquotients, where $[\mathbf{G}',\rho']\to [\mathbf{G},\rho]$ is either a core edge contraction or level merge; higher page differentials map between the subquotients of strata related by a composition of such morphisms. The off-by-one weight pieces $\gr^W_{-(p+q)}H^{\mathrm{BM}}_{p+q+1}(\mathcal{M}_{[\mathbf{G},\rho]})$ and their modular interpretation given in Lemmas \ref{lem:m1nrdoff}, \ref{lem:MF}, and \ref{cor:Mcycle} allows us to qualitatively determine the differentials and describe the resulting relations in §\ref{subsec:baseptrel}, \ref{subsec:pres}. In short, they are:
\begin{enumerate}
  \item pullback of off-by-one weight classes on $\mathcal{M}_{0,n'}$ or $\mathcal{M}_{1,n'},$ corresponding to pullback of the WDVV or Getzler's relation, respectively,
  \item classes associated to quasimaps with basepoints, corresponding to relations among strata (possibly decorated by tautological classes) with rational tails, which we denote as basepoint relations,
  \item Künneth tensor products of the above two with pure weight classes supported on vertices, which are cup products of the above two types of relations with other cohomology classes.
\end{enumerate}

We now go over each source of off-by-one weight classes and describe their behaviours in the spectral sequence.

\subsection{Core with positive degree}\label{subsec:basptrelcomp}
This section deals with off-by-one classes on $\mathcal{M}_{1,n}(\PP^r,d)$ when $d>0.$ We note that the discussion applies to case of positive degree maps from a cycle of rational curves as well.

We recall the short exact sequence from Lemma \ref{lem:m1nrdoff}
\[ 0\to \frac{\mathrm{gr}^W_{-k+1}H_k^{\mathrm{BM}}(\hat{\mathcal{Q}}_{1,n}(\mathbb{P}^r,d))}{\mathrm{im}(\phi)}\to \mathrm{gr}^W_{-k+1} H_k^{\mathrm{BM}}(\mathcal{M}_{1,n}(\PP^r,d))\to \ker(\gr^W_{-k+1}\varphi)\to 0,\] where $\gr^W_{-k+1}\varphi$ is the pure weight truncation of the pushforward map $H^{\mathrm{BM}}_{k-1}(\mathcal{B})\to H_{k-1}^{\mathrm{BM}}(\mathcal{Q}_{1,n}(\mathbb{P}^r,d)).$

The off-by-one classes on $\hat{\mathcal{Q}}_{1,n}(\PP^r,d)$ are generated by pullback from $\gr^W_{-\star+1}H_\star^{\mathrm{BM}}(\mathrm{Pic}^d_{1,n}),$ which by Lemma \ref{lem:M1noff} are in turn generated by pullback from $\mathcal{M}_{1,n'}.$ Therefore, the off-by-one classes in $\mathrm{gr}^W_{-k+1}H_k^{\mathrm{BM}}(\hat{\mathcal{Q}}_{1,n}(\mathbb{P}^r,d))/\mathrm{im}(\phi)$ are pullback of WDVV and Getzler's relations along $H^2(\Mbar_{0,n'})\to H^2(\Mbar_{[\mathbf{G},\rho]})$ and $H^4(\Mbar_{1,n'})\to H^4(\Mbar_{[\mathbf{G},\rho]}).$ 

Comparing the stratification spectral sequences associated to (strata of) $\Mtil_{1,n}(\PP^r,d)$ and those of $\Mbar_{0,n'}$ and $\Mbar_{1,n'}$ under pullback, we see that the off-by-one classes pulled back from moduli of curves and their Künneth tensor products with pure weight classes do not survive to the $E^2$- resp. $E^3$-pages of the spectral sequence.

We now focus on the cokernel in the short exact sequence.

\begin{definition}
  We define $B_{k-1}(n,r,d):=\ker(\varphi: \gr^W_{-k+1} H^{\mathrm{BM}}_{k-1}(\mathcal{B})\to \gr^W_{-k+1} H_{k-1}^{\mathrm{BM}}(\mathcal{Q}_{1,n}(\mathbb{P}^r,d))).$ 
\end{definition}

Because $B_{k-1}(n,r,d)$ arises from the complement $\hat{\mathcal{Q}}_{1,n}(\PP^r,d)\setminus \mathcal{M}_{1,n}(\PP^r,d)$ that corresponds to tuples of sections with basepoints, we refer to elements in $B_{k-1}(n,r,d)$ as the basepoint off-by-one classes of $\mathcal{M}_{1,n}(\PP^r,d).$

\subsubsection{Rational tail contractions}\label{subsec:ratailcon}

We describe the differentials from basepoint off-by-one classes on $\mathcal{M}_{1,n}(\PP^r,d)$ associated to contracting rational tails.

Let $d>\delta$ and let $\mathbf{G}_{\delta}$ be the $(1,n,d)$-graph given below. It specifies a coarse stratum $\Mtil_{[\mathbf{G}_{\delta}]}$ in $\Mtil_{1,n}(\PP^r,d)$ that agrees with the stable maps stratum specified by the same graph. Its unique edge contraction recovers the $(1,n,d)$-graph associated to $\mathcal{M}_{1,n}(\PP^r,d).$ 

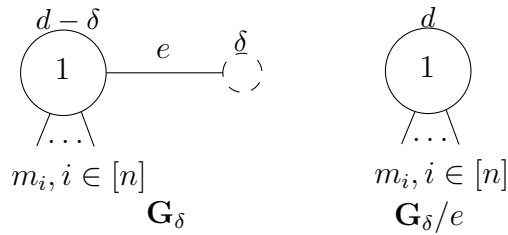
\[\begin{tikzpicture}[x=0.75pt,y=0.75pt,yscale=-1,xscale=1]

\draw   (12.86,45.89) .. controls (12.86,33.98) and (22.52,24.32) .. (34.43,24.32) .. controls (46.34,24.32) and (56,33.98) .. (56,45.89) .. controls (56,57.8) and (46.34,67.46) .. (34.43,67.46) .. controls (22.52,67.46) and (12.86,57.8) .. (12.86,45.89) -- cycle ;
\draw    (56,45.89) -- (114.88,45.89) ;
\draw    (43.43,65.46) -- (49.94,83.06) ;
\draw    (27.73,66.53) -- (21.03,82.86) ;
\draw   (196.86,44.89) .. controls (196.86,32.98) and (206.52,23.32) .. (218.43,23.32) .. controls (230.34,23.32) and (240,32.98) .. (240,44.89) .. controls (240,56.8) and (230.34,66.46) .. (218.43,66.46) .. controls (206.52,66.46) and (196.86,56.8) .. (196.86,44.89) -- cycle ;
\draw    (227.43,64.46) -- (233.94,82.06) ;
\draw    (211.73,65.53) -- (205.03,81.86) ;
\draw  [dash pattern={on 4.5pt off 4.5pt}] (114.88,45.89) .. controls (114.88,40.44) and (119.3,36.03) .. (124.74,36.03) .. controls (130.19,36.03) and (134.61,40.44) .. (134.61,45.89) .. controls (134.61,51.34) and (130.19,55.75) .. (124.74,55.75) .. controls (119.3,55.75) and (114.88,51.34) .. (114.88,45.89) -- cycle ;

\draw (28.79,37.19) node [anchor=north west][inner sep=0.75pt]    {$1$};
\draw (19.33,12.56) node [anchor=north west][inner sep=0.75pt]  [font=\small]  {$d-\delta $};
\draw (119.22,23.67) node [anchor=north west][inner sep=0.75pt]  [font=\small]  {$\delta $};
\draw (75,109.4) node [anchor=north west][inner sep=0.75pt]    {$\mathbf{G}_{\delta }$};
\draw (200,109.4) node [anchor=north west][inner sep=0.75pt]    {$\mathbf{G}_{\delta } /e$};
\draw (25,78.26) node [anchor=north west][inner sep=0.75pt]    {$\dotsc $};
\draw (7,89.38) node [anchor=north west][inner sep=0.75pt]    {$m_i, i\in [n]$};
\draw (80,30.4) node [anchor=north west][inner sep=0.75pt]    {$e$};
\draw (212.79,36.19) node [anchor=north west][inner sep=0.75pt]    {$1$};
\draw (213.33,11.56) node [anchor=north west][inner sep=0.75pt]  [font=\small]  {$d$};
\draw (209,77.26) node [anchor=north west][inner sep=0.75pt]    {$\dotsc $};
\draw (190,88.38) node [anchor=north west][inner sep=0.75pt]    {$m_i, i\in [n]$};
\end{tikzpicture}
\]

The basic geometric picture is that the stable map on the rational tail gets contracted as a basepoint on the genus one curve. As explained in §\ref{sec:introgenrel}, the rational tail contraction induces a map of pairs $(\mathcal{M}_{1,n}\cup \Mtil_{[\mathbf{G}_{\delta}]},  \Mtil_{[\mathbf{G}_{\delta}]})\to (\hat{\mathcal{Q}}_{1,n}(\PP^r,d), \hat{\mathcal{Q}}_{1,n}(\PP^r,d)\setminus \mathcal{M}_{1,n}(\PP^r,d))$ and hence a commutative diagram that compares the boundary morphisms associated to the pairs. Since the basepoint classes are pure weight classes on $\hat{\mathcal{Q}}_{1,n}(\PP^r,d)\setminus \mathcal{M}_{1,n}(\PP^r,d),$ the behaviour of the basepoint classes $B_{k-1}(n,r,d)$ under the differential is completely determined by the comparison commutative diagram.

\begin{lemma}\label{lem:cobratail}
  The boundary map $\partial: \gr^W_{-(k-1)}H_k^{\mathrm{BM}}(\mathcal{M}_{1,n}(\PP^r,d))\to \mathrm{gr}^W_{-(k-1)}H_{k-1}^{\mathrm{BM}}(\Mtil_{[\mathbf{G}_{\delta}]})$ fits into the following commutative diagram \[\begin{tikzcd}
	{\gr^W_{-(k-1)}H_k^{\mathrm{BM}}(\mathcal{M}_{1,n}(\PP^r,d))} & {\gr^W_{-(k-1)}H_{k-1}^{\mathrm{BM}}(\Mtil_{[\mathbf{G}_{\delta}]})} \\
	{B_{k-1}(n,r,d)} & {\mathrm{gr}^W_{-(k-1)}H_{k-1}^{\mathrm{BM}}({\mathcal{M}}_{1,n+1}(\PP^r,d-\delta))} 
	\arrow["\partial", from=1-1, to=1-2]
	\arrow[two heads, from=1-1, to=2-1]
	\arrow["{s}", from=2-2, to=1-2]
	\arrow[from=2-1, to=2-2]
\end{tikzcd}\]where $s$ sends $\sigma\mapsto \sigma\otimes [\mathrm{pt}]$ under the Künneth formula $H_{\star}^{\mathrm{BM}}(\Mtil_{[\mathbf{G}_{\delta}]})\cong H^{\mathrm{BM}}_{\star}(\mathcal{M}_{1,n+1}(\PP^r,d-\delta))\otimes H_\star(\Mbar_{0,0}^*(\PP^r,\delta)).$ The map $s$ is a section of the proper pushforward map along the forgetful map $\Mtil_{[\mathbf{G}_{\delta}]}\to \mathcal{M}_{1,n+1}(\PP^r,d-\delta).$
\end{lemma}
\begin{proof}
  As $\mathrm{gr}^W_{-k+1}H_k^{\mathrm{BM}}(\hat{\mathcal{Q}}_{1,n}(\mathbb{P}^r,d))$ is generated by pullback from $\mathrm{gr}^W_{-\star+1}H_\star^{\mathrm{BM}}({\mathrm{Pic}}^d_{1,n}),$ given any $$g\in \frac{\mathrm{gr}^W_{-k+1}H_k^{\mathrm{BM}}(\hat{\mathcal{Q}}_{1,n}(\mathbb{P}^r,d))}{\mathrm{im}(\phi)}\subset \gr^W_{-(k-1)}H_k^{\mathrm{BM}}(\mathcal{M}_{1,n}(\PP^r,d)),$$ we have $\pi_* \partial (g) = 0.$ Therefore, there is a well-defined map $$B_{k-1}(n,r,d)\to \gr^W_{(k-1)}H_{k-1}^{\mathrm{BM}}(\mathcal{M}_{1,n+1}(\PP^r,d-\delta)).$$

  In the following, we use $\mathrm{d}_\mathcal{B}$ as a shorthand for $\dim(\mathcal{B}).$ Recall from Lemma \ref{lem:m1nrd} that the basepoint classes $B_{\star}(n,r,d)$ receive a surjection from $\gr^W_{2\mathrm{d}_\mathcal{B}-\star}H^{2\mathrm{d}_\mathcal{B}-\star}(\hat{\mathcal{Q}}_{1,n+1}(\PP^r,d-1)).$ This allows us to describe the boundary map geometrically as follows.

  Consider the map $\hat{\mathcal{Q}}_{1,n+1}(\PP^r,d-\delta)\to \hat{\mathcal{Q}}_{1,n+1}(\PP^r,d-1)$ given as a map over $\mathcal{M}_{1,n+1}$ by $$(\mathcal{L}, [s_0:\cdots:s_r])\mapsto (\mathcal{L}\otimes \mathcal{O}((\delta-1)s_{p+1}), s_{p+1}^{\delta-1}\otimes [s_0:\cdots:s_r]).$$

  As the map lies over the isomorphism $\Pic_{1,n+1}^{d-\delta}\to \Pic_{1,n+1}^{d-1}$ over $\mathcal{M}_{1,n}$ given by $\mathcal{L}\mapsto \mathcal{L}\otimes \mathcal{O}((\delta-1)p_{n+1}),$ the cohomology pullback is an isomorphism in cohomological degrees upto $\dim \hat{\mathcal{Q}}_{1,n+1}(\PP^r,d-\delta).$ On the other hand, the proper pushforward sends $[\hat{\mathcal{Q}}_{1,n+1}(\PP^r,d-\delta)]$ to $H_{\hat{\mathcal{Q}}}^{(\delta-1)(r+1)}\cap [\hat{\mathcal{Q}}_{1,n+1}(\PP^r,d-1)].$ Projection formula hence determines the proper pushforward map on the level of Borel--Moore homology groups. The desired boundary map factors through the commutative diagram \[\begin{tikzcd}
	{B_{\star}(n,r,d)} & {\mathrm{gr}^W_{-\star}H_{\star}^{\mathrm{BM}}({\mathcal{M}}_{1,n+1}(\PP^r,d-\delta))} \\
	{\mathrm{gr}^W_{-\star}H_{\star}^{\mathrm{BM}}({\hat{\mathcal{Q}}}_{1,n+1}(\PP^r,d-1))} & {\mathrm{gr}^W_{-\star}H_{\star}^{\mathrm{BM}}({\hat{\mathcal{Q}}}_{1,n+1}(\PP^r,d-\delta))}
	\arrow[from=1-1, to=1-2]
	\arrow[two heads, from=2-1, to=1-1]
	\arrow[from=2-2, to=2-1]
	\arrow[two heads, from=2-2, to=1-2]
\end{tikzcd}\]
\end{proof}

\begin{remark}
  Since $H_{\star}^{\mathrm{BM}}(\Mtil_{[\mathbf{G}_{\delta}]})\cong H^{\mathrm{BM}}_{\star}(\mathcal{M}_{1,n+1}(\PP^r,d-\delta))\otimes H_\star(\Mbar_{0,0}^*(\PP^r,\delta)),$ the pushforward map $\Mtil_{[\mathbf{G}_{\delta}]}\to \mathcal{M}_{1,n+1}(\PP^r,d-\delta)$ is the linear projection 
  $$H^{\mathrm{BM}}_{\star}(\mathcal{M}_{1,n+1}(\PP^r,d-\delta))\otimes H_\star(\Mbar_{0,0}^*(\PP^r,\delta))\to H^{\mathrm{BM}}_{\star}(\mathcal{M}_{1,n+1}(\PP^r,d-\delta))\otimes [\mathrm{pt}].$$

  Thus, the basepoint relations concern classes coming from the genus one core but not the rational tails. Informally, the complexity of the cohomology of rational tails seen on the homology groups of each stratum `survives' to the limit $H_\star(\Mtil_{1,n}(\PP^r,d)).$
\end{remark}

When there are markings on the rational tail, we apply flat pullback along the forgetful map and reduce to the case of unmarked rational tails as treated above. Let $I\subset [n]$ be a subset, $d>\delta,$ and let $\mathbf{G}_{\delta, I}$ be the $(1,n,d)$-graph given below with specified stratum $\Mtil_{[\mathbf{G}_{\delta, I}]}.$

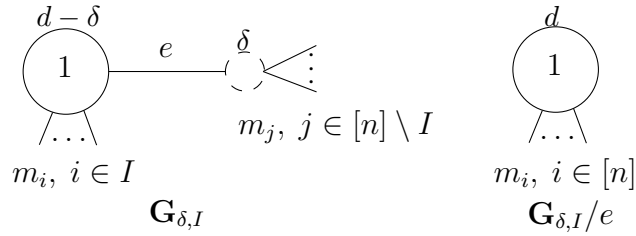
\[\begin{tikzpicture}[x=0.75pt,y=0.75pt,yscale=-1,xscale=1]

\draw   (12.86,45.89) .. controls (12.86,33.98) and (22.52,24.32) .. (34.43,24.32) .. controls (46.34,24.32) and (56,33.98) .. (56,45.89) .. controls (56,57.8) and (46.34,67.46) .. (34.43,67.46) .. controls (22.52,67.46) and (12.86,57.8) .. (12.86,45.89) -- cycle ;
\draw    (56,45.89) -- (114.88,45.89) ;
\draw    (43.43,65.46) -- (49.94,83.06) ;
\draw    (27.73,66.53) -- (21.03,82.86) ;
\draw    (133.88,45.89) -- (160.74,32.75) ;
\draw    (133.88,45.89) -- (160.85,56.18) ;
\draw   (259.86,44.89) .. controls (259.86,32.98) and (269.52,23.32) .. (281.43,23.32) .. controls (293.34,23.32) and (303,32.98) .. (303,44.89) .. controls (303,56.8) and (293.34,66.46) .. (281.43,66.46) .. controls (269.52,66.46) and (259.86,56.8) .. (259.86,44.89) -- cycle ;
\draw    (290.43,64.46) -- (296.94,82.06) ;
\draw    (274.73,65.53) -- (268.03,81.86) ;
\draw  [dash pattern={on 4.5pt off 4.5pt}] (114.88,45.89) .. controls (114.88,40.44) and (119.3,36.03) .. (124.74,36.03) .. controls (130.19,36.03) and (134.61,40.44) .. (134.61,45.89) .. controls (134.61,51.34) and (130.19,55.75) .. (124.74,55.75) .. controls (119.3,55.75) and (114.88,51.34) .. (114.88,45.89) -- cycle ;

\draw (28.79,37.19) node [anchor=north west][inner sep=0.75pt]    {$1$};
\draw (18.33,11.56) node [anchor=north west][inner sep=0.75pt]  [font=\small]  {$d-\delta $};
\draw (119.22,23.67) node [anchor=north west][inner sep=0.75pt]  [font=\small]  {$\delta $};
\draw (75,109.4) node [anchor=north west][inner sep=0.75pt]    {$\mathbf{G}_{\delta, I}$};
\draw (266,109.4) node [anchor=north west][inner sep=0.75pt]    {$\mathbf{G}_{\delta, I} /e$};
\draw (25.5,78.26) node [anchor=north west][inner sep=0.75pt]    {$\dotsc $};
\draw (6,89.38) node [anchor=north west][inner sep=0.75pt]    {$m_{i} ,\ i\in I$};
\draw (156.22,54.7) node [anchor=north west][inner sep=0.75pt]  [rotate=-270]  {$\dotsc $};
\draw (120,64.38) node [anchor=north west][inner sep=0.75pt]    {$m_{j} ,\ j\in [ n] \setminus I$};
\draw (80,30.4) node [anchor=north west][inner sep=0.75pt]    {$e$};
\draw (275.79,36.19) node [anchor=north west][inner sep=0.75pt]    {$1$};
\draw (274.33,11.56) node [anchor=north west][inner sep=0.75pt]  [font=\small]  {$d$};
\draw (272,77.26) node [anchor=north west][inner sep=0.75pt]    {$\dotsc $};
\draw (249,88.38) node [anchor=north west][inner sep=0.75pt]    {$m_{i} ,\ i\in [ n]$};

\end{tikzpicture}\]

We still have a map between pairs $$(\mathcal{M}_{1,n}(\PP^r,d)\cup \mathcal{M}_{[\mathbf{G}_{\delta, I}]},  \mathcal{M}_{[\mathbf{G}_{\delta, I}]})\to (\hat{Q}_{1,|I|+1}(\PP^r,d),\hat{\mathcal{Q}}_{1,|I|+1}(\PP^r,d)\setminus \mathcal{M}_{1,|I|+1}(\PP^r,d)),$$ and the induced commutative diagram between the boundary maps is compatible with flat pullback along the natural map that forgets the markings in $[n]\setminus I.$

\begin{lemma}\label{lem:pbss}
  Let $\mathbf{G}_{\delta, I}^{\setminus I}$ be the $(1, n-|I|, d)$-graph given by forgetting the markings in $I$ from $\mathbf{G}_{\delta, I}$ as shown in the following figure.
  
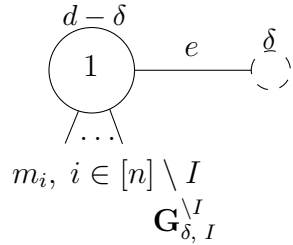
\[\begin{tikzpicture}[x=0.75pt,y=0.75pt,yscale=-1,xscale=1]

\draw   (25.86,45.89) .. controls (25.86,33.98) and (35.52,24.32) .. (47.43,24.32) .. controls (59.34,24.32) and (69,33.98) .. (69,45.89) .. controls (69,57.8) and (59.34,67.46) .. (47.43,67.46) .. controls (35.52,67.46) and (25.86,57.8) .. (25.86,45.89) -- cycle ;
\draw    (69,45.89) -- (127.88,45.89) ;
\draw    (56.43,65.46) -- (62.94,83.06) ;
\draw    (40.73,66.53) -- (34.03,82.86) ;
\draw  [dash pattern={on 4.5pt off 4.5pt}] (127.88,45.89) .. controls (127.88,40.44) and (132.3,36.03) .. (137.74,36.03) .. controls (143.19,36.03) and (147.61,40.44) .. (147.61,45.89) .. controls (147.61,51.34) and (143.19,55.75) .. (137.74,55.75) .. controls (132.3,55.75) and (127.88,51.34) .. (127.88,45.89) -- cycle ;

\draw (41.79,37.19) node [anchor=north west][inner sep=0.75pt]    {$1$};
\draw (31.33,11.56) node [anchor=north west][inner sep=0.75pt]  [font=\small]  {$d-\delta $};
\draw (132.22,23.67) node [anchor=north west][inner sep=0.75pt]  [font=\small]  {$\delta $};
\draw (77,109.4) node [anchor=north west][inner sep=0.75pt]    {$\mathbf{G}_{\delta ,\ I}^{\setminus I}$};
\draw (40.03,78.26) node [anchor=north west][inner sep=0.75pt]    {$\dotsc $};
\draw (6,89.38) node [anchor=north west][inner sep=0.75pt]    {$m_{i} ,\ i\in [ n] \setminus I$};
\draw (93,30.4) node [anchor=north west][inner sep=0.75pt]    {$e$};
\end{tikzpicture}\]

The boundary map $B_{k-1}(n, r,d)\to \gr_{-(k-1)}^{W}H_{k-1}^{\mathrm{BM}}(\Mtil_{[\mathbf{G}_{\delta, I}]})$ is determined by the following commutative diagram
\[\begin{tikzcd}
	{B_{k-1-|I|}([n]\setminus I, r,d)} & {\gr_{-(k-1-|I|)}^{W}H_{k-1-|I|}^{\mathrm{BM}}(\Mtil_{[\mathbf{G}_{\delta, I}^{\setminus I}]})} \\
	{B_{k-1}(n, r,d)} & {\gr_{-(k-1)}^{W}H_{k-1}^{\mathrm{BM}}(\Mtil_{[\mathbf{G}_{\delta, I}]})}
	\arrow[from=1-1, to=1-2]
	\arrow[from=1-1, to=2-1]
	\arrow[from=1-2, to=2-2]
	\arrow[from=2-1, to=2-2]
\end{tikzcd}\]
where the vertical maps are flat pullbacks along forgetful maps that forgets the marked points indexed by $I$ (and stabilises if necessary).

The above diagrams are compatible with the commutative diagram in the proof of Lemma \ref{lem:cobratail}.
\end{lemma}

\begin{remark}
  We recall a result of Petersen \cite[Theorems 1.1, 1.3]{pet} that the pure weight cohomology classes on $\mathcal{M}_{1,n}$ are either $H^0(\mathcal{M}_{1,n})$ or odd cohomology groups corresponding to cusp forms. The behaviour of the latter classes under pullbacks along forgetful maps $\mathcal{M}_{1,n}\to \mathcal{M}_{1,n-k}$ are described by \cite[Proposition 2.2]{clp}. These facts ensure that the above commutative diagram indeed determines the boundary map.
\end{remark}

Before moving on to rational tail contractions between strata of higher codimension, we make the following observation.

\begin{lemma}\label{lem:g1baseptE2}
  The basepoint classes $B_{k-1}(n, r,d)$ do not survive to the $E^2$-page.
\end{lemma}
\begin{proof}
  Because $B_{k-1}(n, r,d)$ arise from quasi-maps with degree one basepoints, the differentials given by contracting degree one rational tails (with zero or one markings) map all of $B_{k-1}(n,r,d)$ to non-zero classes on the next term in the $E^1$-page.
\end{proof}

Iterating the procedure described above determines arbitrary contractions of rational tails to the core. Let $\mathbf{G}$ be an $(1,n,d)$-graph with positive core degree, and let $e$ be an edge that connects one rational tail to the core.

\begin{definition}
  Let $\mathbf{G}^{(e)}$ be the subgraph of $\mathbf{G}$ with all rational tails cut off except the one connected to the core by $e.$
\end{definition}

Recall from Lemma \ref{lem:kunneth} that we have Künneth formulas for the Borel--Moore homology groups of $\Mtil_{\mathbf{G}}^{\mathrm{ord}}$ and $\Mtil_{\mathbf{G}/e}^{\mathrm{ord}}$.

\begin{lemma}\label{lem:kunnss}
  The boundary map is $\mathrm{gr}^W_{-(k-1)}H_{k}^{\mathrm{BM}}(\Mtil_{\mathbf{G}/e})\to \mathrm{gr}^W_{-(k-1)}H_{k-1}^{\mathrm{BM}}(\Mtil_{\mathbf{G}})$ is determined by the commutative diagram

  \[\begin{tikzcd}
	{\mathrm{gr}^W_{-(k-1)}H_{k}^{\mathrm{BM}}(\Mtil_{\mathbf{G}/e}^{\mathrm{ord}})} & {\mathrm{gr}^W_{-(k-1)}H_{k-1}^{\mathrm{BM}}(\Mtil_{\mathbf{G}}^{\mathrm{ord}})} \\
	{\mathrm{gr}^W_{-(k-1)}H_{k}^{\mathrm{BM}}(\Mtil_{\mathbf{G}/e})} & {\mathrm{gr}^W_{-(k-1)}H_{k-1}^{\mathrm{BM}}(\Mtil_{\mathbf{G}})}
	\arrow[from=1-1, to=1-2]
	\arrow[from=1-1, to=2-1]
	\arrow[from=1-2, to=2-2]
	\arrow[from=2-1, to=2-2]
\end{tikzcd}\]
Here the top horizontal map is the Künneth product of the boundary map $$\mathrm{gr}^W_{-(\star-1)}H_{\star}^{\mathrm{BM}}(\Mtil_{\mathbf{G}^{(e)}/e}^{\mathrm{ord}})\to \mathrm{gr}^W_{-(\star-1)}H_{\star}^{\mathrm{BM}}(\Mtil_{\mathbf{G}^{(e)}}^{\mathrm{ord}})$$ and identity on the other Künneth components, which are the cohomology groups of the other rational tails. The two vertical maps are the quotient maps given by taking $\mathrm{Aut}(\mathbf{G}/e)$ and $\mathrm{Aut}(\mathbf{G})$-actions.
\end{lemma}

Therefore, the Künneth tensor products of basepoint off-by-one classes on the genus one core and any pure weight class on the rational tail also do not survive to the $E^2$-page.

\subsubsection{Radial merge to the core}\label{subsec:radmercore}
We describe the boundary maps induced by collapsing genus zero maps on the contraction radius to the core. By the same considerations as in Lemmas \ref{lem:pbss} and \ref{lem:kunnss}, it suffices to consider the following level merge: 
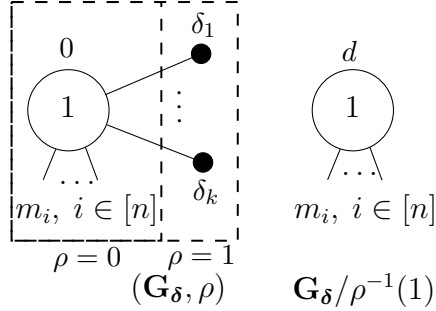
\[\begin{tikzpicture}[x=0.75pt,y=0.75pt,yscale=-0.95,xscale=0.95]

\draw   (11.86,62.89) .. controls (11.86,50.98) and (21.52,41.32) .. (33.43,41.32) .. controls (45.34,41.32) and (55,50.98) .. (55,62.89) .. controls (55,74.8) and (45.34,84.46) .. (33.43,84.46) .. controls (21.52,84.46) and (11.86,74.8) .. (11.86,62.89) -- cycle ;
\draw    (53,54.66) -- (103.88,32.89) ;
\draw  [fill={rgb, 255:red, 0; green, 0; blue, 0 }  ,fill opacity=1 ] (98.72,32.89) .. controls (98.72,30.04) and (101.03,27.73) .. (103.88,27.73) .. controls (106.73,27.73) and (109.04,30.04) .. (109.04,32.89) .. controls (109.04,35.74) and (106.73,38.05) .. (103.88,38.05) .. controls (101.03,38.05) and (98.72,35.74) .. (98.72,32.89) -- cycle ;
\draw    (42.43,82.46) -- (48.94,100.06) ;
\draw    (26.73,83.53) -- (20.03,99.86) ;
\draw   (162.86,62.89) .. controls (162.86,50.98) and (172.52,41.32) .. (184.43,41.32) .. controls (196.34,41.32) and (206,50.98) .. (206,62.89) .. controls (206,74.8) and (196.34,84.46) .. (184.43,84.46) .. controls (172.52,84.46) and (162.86,74.8) .. (162.86,62.89) -- cycle ;
\draw    (193.43,82.46) -- (199.94,100.06) ;
\draw    (177.73,83.53) -- (171.03,99.86) ;
\draw    (54,70.66) -- (104.88,90.62) ;
\draw  [fill={rgb, 255:red, 0; green, 0; blue, 0 }  ,fill opacity=1 ] (99.72,90.62) .. controls (99.72,87.78) and (102.03,85.47) .. (104.88,85.47) .. controls (107.73,85.47) and (110.04,87.78) .. (110.04,90.62) .. controls (110.04,93.47) and (107.73,95.78) .. (104.88,95.78) .. controls (102.03,95.78) and (99.72,93.47) .. (99.72,90.62) -- cycle ;
\draw  [dash pattern={on 4.5pt off 4.5pt}][line width=0.75]  (3.25,5.39) -- (82.76,5.39) -- (82.76,132.19) -- (3.25,132.19) -- cycle ;
\draw  [dash pattern={on 4.5pt off 4.5pt}][line width=0.75]  (3.25,5.39) -- (123.33,5.39) -- (123.33,132.19) -- (3.25,132.19) -- cycle ;

\draw (27.79,54.19) node [anchor=north west][inner sep=0.75pt]    {$1$};
\draw (27.33,23.56) node [anchor=north west][inner sep=0.75pt]  [font=\small]  {$0$};
\draw (97.22,9.67) node [anchor=north west][inner sep=0.75pt]  [font=\small]  {$\delta _{1}$};
\draw (65,148.4) node [anchor=north west][inner sep=0.75pt]    {$(\mathbf{G}_{\boldsymbol{\delta}},\rho)$};
\draw (151,148.4) node [anchor=north west][inner sep=0.75pt]    {$\mathbf{G}_{\boldsymbol{\delta}} /\rho ^{-1}( 1)$};
\draw (27.03,99.26) node [anchor=north west][inner sep=0.75pt]    {$\dotsc $};
\draw (4,106.38) node [anchor=north west][inner sep=0.75pt]    {$m_{i} ,\ i\in [ n]$};
\draw (178.79,54.19) node [anchor=north west][inner sep=0.75pt]    {$1$};
\draw (177.33,24.56) node [anchor=north west][inner sep=0.75pt]  [font=\small]  {$d$};
\draw (177.03,94.26) node [anchor=north west][inner sep=0.75pt]    {$\dotsc $};
\draw (152,106.38) node [anchor=north west][inner sep=0.75pt]    {$m_{i} ,\ i\in [ n]$};
\draw (90.31,71.85) node [anchor=north west][inner sep=0.75pt]  [rotate=-269.27]  {$\dotsc $};
\draw (98.22,98.4) node [anchor=north west][inner sep=0.75pt]  [font=\small]  {$\delta _{k}$};
\draw (24,134.4) node [anchor=north west][inner sep=0.75pt]  [font=\small]  {$\rho =0$};
\draw (84.76,133.59) node [anchor=north west][inner sep=0.75pt]  [font=\small]  {$\rho =1$};
\end{tikzpicture}\]

Contracting the genus zero maps to quasimaps with basepoints on the genus one core, standard functoriality properties again imply
\begin{lemma}\label{lem:cobradmer}
  The boundary map $\partial: \gr^W_{-(k-1)}H_k^{\mathrm{BM}}(\mathcal{M}_{1,n}(\PP^r,d))\to \mathrm{gr}^W_{-(k-1)}H_{k-1}^{\mathrm{BM}}(\Mtil_{[\mathbf{G}_{\boldsymbol{\delta}},\rho]})$ fits into the following commutative diagram \[\begin{tikzcd}
	{\gr^W_{-(k-1)}H_k^{\mathrm{BM}}(\mathcal{M}_{1,n}(\PP^r,d))} & {\gr^W_{-(k-1)}H_{k-1}^{\mathrm{BM}}(\Mtil_{[\mathbf{G}_{\boldsymbol{\delta}},\rho]})} \\
	{B_{k-1}(n,r,d)} & {\mathrm{gr}^W_{-(k-1)}H_{k-1}^{\mathrm{BM}}({\mathcal{M}}_{1,n+k}(\PP^r,0))} \\
	{\gr^W_{-(k-1)}H_{k-1}^{\mathrm{BM}}(\mathcal{B})} & {\mathrm{gr}^W_{-(k-1)}H_{k-1}^{\mathrm{BM}}({\mathcal{C}}_{1,n}^k\times_{\mathcal{M}_{1,n}}\hat{\mathcal{Q}}_{1,n}(\PP^r,0))}
	\arrow["\partial", from=1-1, to=1-2]
	\arrow[two heads, from=1-1, to=2-1]
	\arrow["{\pi_*}", from=1-2, to=2-2]
	\arrow[hook, from=2-1, to=3-1]
	\arrow[from=2-1, to=2-2]
	\arrow["\iota"', two heads, from=3-2, to=2-2]
	\arrow["{m_*}", from=3-2, to=3-1]
\end{tikzcd}\]where $\pi:\Mtil_{[\mathbf{G}_{\boldsymbol{\delta}},\rho]}\to \mathcal{M}_{1,n+k}(\PP^r,0)$ contracts the collection of genus zero maps and records the contraction points on the elliptic curve, the map $m: \mathcal{M}_{1,n+k}(\PP^r,0)\to \mathcal{B}\subset \hat{\mathcal{Q}}_{1,n}(\PP^r,d)$ brings in the last $k$ marked points as basepoints with multiplicities $\delta_1,\dots,\delta_k,$ and $\iota$ denotes the open embedding $$\mathcal{M}_{1,n+k}(\PP^r,0)\subset \mathcal{C}_{1,n}^k\times_{\mathcal{M}_{1,n}}\hat{\mathcal{Q}}_{1,n}(\PP^r,0).$$
\end{lemma}

\begin{remark}
  We will use $(\mathbf{G}_{\boldsymbol{\delta},\boldsymbol{m}},\rho)$ to denote the centrally aligned $(1,n,d)$-graph that consists of a degree zero vertex of genus one and a contraction radius with degree and marking distributions given by $(\boldsymbol{\delta},\boldsymbol{m}).$
\end{remark}

\subsubsection{Core edge contraction}
We consider the $(1,n,d)$-graph $\mathsf{C}_{1,(d, [n])}$ given by a 1-cycle with a single degree-$d$ vertex and all marked points, as shown below. It specifies the locally closed stratum $\mathcal{M}_{\mathsf{C}_{1,(d, [n])}}\subset \Mtil_{1,n}(\PP^r,d).$
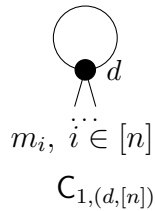
\[\begin{tikzpicture}[x=0.75pt,y=0.75pt,yscale=-1,xscale=1]

\draw  [fill={rgb, 255:red, 0; green, 0; blue, 0 }  ,fill opacity=1 ] (40.72,37.89) .. controls (40.72,35.04) and (43.03,32.73) .. (45.88,32.73) .. controls (48.73,32.73) and (51.04,35.04) .. (51.04,37.89) .. controls (51.04,40.74) and (48.73,43.05) .. (45.88,43.05) .. controls (43.03,43.05) and (40.72,40.74) .. (40.72,37.89) -- cycle ;
\draw   (29.78,21.79) .. controls (29.78,12.89) and (36.99,5.68) .. (45.88,5.68) .. controls (54.78,5.68) and (61.99,12.89) .. (61.99,21.79) .. controls (61.99,30.68) and (54.78,37.89) .. (45.88,37.89) .. controls (36.99,37.89) and (29.78,30.68) .. (29.78,21.79) -- cycle ;
\draw    (45.88,37.75) -- (51.17,56.42) ;
\draw    (45.88,37.75) -- (40.5,56.42) ;

\draw (37.03,61.38) node [anchor=north west][inner sep=0.75pt]  [font=\scriptsize]  {$\dotsc $};
\draw (30,90) node [anchor=north west][inner sep=0.75pt]    {$\mathsf{C}_{1,(d, [n])}$};
\draw (7,63.38) node [anchor=north west][inner sep=0.75pt]    {$m_{i} ,\ i\in [ n]$};
\draw (55.33,31.4) node [anchor=north west][inner sep=0.75pt]  [font=\small]  {$d$};
\end{tikzpicture}\]
The basepoint classes $B_{k-1}(n,r,d)$ do not contribute to the boundary map associated to core edge contraction in the following sense.
\begin{lemma}
  The kernel $\ker \partial$ of the boundary map $$\partial:\mathrm{gr}^W_{-(k-1)} H_{k}^{\mathrm{BM}}(\mathcal{M}_{1,n}(\PP^r,d))\to \mathrm{gr}^W_{-(k-1)} H_{k-1}^{\mathrm{BM}}(\mathcal{M}_{\mathsf{C}_{1,(d, [n])}})$$ surjects onto $B_{k-1}(n,r,d).$ 
\end{lemma}

\begin{proof}
  Let $\mathcal{M}_{\mathsf{C}_{1,[n]}}\subset \Mbar_{1,n}$ be the locally closed substack of nodal elliptic curves with a single node. We can construct the space $\hat{\mathcal{Q}}'$ of quasimaps over $\mathcal{M}_{1,n}\cup \mathcal{M}_{\mathsf{C}_{1,[n]}}$ as a projective bundle over the relative degree-$d$ Picard group of $\mathcal{M}_{1,n}\cup \mathcal{M}_{\mathsf{C}_{1,[n]}}.$ In particular, we have $\hat{\mathcal{Q}}' = \hat{\mathcal{Q}}_{1,n}(\PP^r,d)\cup  \hat{\mathcal{Q}}_{\mathsf{C}_{1,(d, [n])}}$ and an open embedding $\mathcal{M}_{\mathsf{C}_{1,(d, [n])}}\subset \hat{\mathcal{Q}}_{\mathsf{C}_{1,(d, [n])}}$ as basepoint free quasimaps.

  The excision long exact sequences associated to the two pairs $\mathcal{M}_{1,n}(\PP^r,d)\subset \mathcal{M}_{1,n}(\PP^r,d)\cup \mathcal{M}_{\mathsf{C}_{1,(d, [n])}}$ as well as $\hat{\mathcal{Q}}_{1,n}(\PP^r,d)\subset \hat{\mathcal{Q}}_{1,n}(\PP^r,d)\cup \hat{\mathcal{Q}}_{\mathsf{C}_{1,(d, [n])}}$ are compatible with the open embedding $\mathcal{M}_{1,n}(\PP^r,d)\cup \mathcal{M}_{\mathsf{C}_{1,(d, [n])}}\subset \hat{\mathcal{Q}}',$ so we have the following commutative diagram \[\begin{tikzcd}
	{\gr^W_{-(k-1)}H_k^{\mathrm{BM}}(\hat{\mathcal{Q}}_{1,n}(\PP^r,d))} & { \gr^W_{-(k-1)}H_{k-1}^{\mathrm{BM}}(\hat{\mathcal{Q}}_{\mathsf{C}_{1,(d, [n])}})} \\
	{\gr^W_{-(k-1)}H_k^{\mathrm{BM}}({\mathcal{M}}_{1,n}(\PP^r,d))} & { \gr^W_{-(k-1)}H_{k-1}^{\mathrm{BM}}({\mathcal{M}}_{\mathsf{C}_{1,(d, [n])}})} \\
	{B_{k-1}(n,r,d)}
	\arrow[from=1-1, to=1-2]
	\arrow[from=1-1, to=2-1]
	\arrow[two heads, from=1-2, to=2-2]
	\arrow["\partial", from=2-1, to=2-2]
	\arrow[two heads, from=2-1, to=3-1]
\end{tikzcd}\] where the horizontal maps are the boundary maps in the long exact sequences. Since $$B_{k-1}(n,r,d)=\frac{\gr^W_{-(k-1)}H_k^{\mathrm{BM}}({\mathcal{M}}_{1,n}(\PP^r,d))}{\mathrm{im}(\gr^W_{-(k-1)}H_k^{\mathrm{BM}}(\hat{\mathcal{Q}}_{1,n}(\PP^r,d)))},$$ and the vertical arrow on the right is surjective, a diagram chase shows that $\ker(\partial)$ surjects onto $B_{k-1}(n,r,d).$
\end{proof}

Intuitively, this statement says that boundary map associated to core edge contraction reflects degeneration of the underlying curve and is hence `orthogonal' to the basepoint classes. The same construction applies to core edge contractions of nodal elliptic curves in general.

\subsubsection{Core with loops}\label{subsec:baseptloops}
We briefly introduce the notation of basepoint classes associated to maps from a cycles of rational curves. Recall that $\mathsf{C}_{\ell,(\boldsymbol{\delta},\boldsymbol{m})}$ denotes the $(1,n,d)$-graph given by an $\ell$-cycle with degree and marking distributions given by $\boldsymbol{\delta}$ and $\boldsymbol{m}.$

\begin{definition}
  Define $$B_{k-1}(\mathsf{C}_{\ell,(\boldsymbol{\delta},\boldsymbol{m})}):= \ker(\gr^W_{-(k-1)}H_{k-1}^{\mathrm{BM}}(\hat{\mathcal{Q}}_{\mathsf{C}_{\ell,(\boldsymbol{\delta},\boldsymbol{m})}}\setminus \mathcal{M}_{\mathsf{C}_{\ell,(\boldsymbol{\delta},\boldsymbol{m})}})\to \gr^W_{-(k-1)}H_{k-1}^{\mathrm{BM}}(\hat{\mathcal{Q}}_{\mathsf{C}_{\ell,(\boldsymbol{\delta},\boldsymbol{m})}})).$$ It receives a surjection from $\gr^W_{-(k-1)}H_k^{\mathrm{BM}}(\mathcal{M}_{\mathsf{C}_{\ell,(\boldsymbol{\delta},\boldsymbol{m})}}).$
\end{definition}

The behaviours of the basepoint classes $B_{k-1}(\mathsf{C}_{\ell,(\boldsymbol{\delta},\boldsymbol{m})})$ in the spectral seqeunce differentials are analogous to the classes $B_{k-1}(n,r,d)$ associated to $\mathcal{M}_{1,n}(\PP^r,d)$:\begin{enumerate}
  \item The boundary maps corresponding to rational tail contractions or level merge from the contraction radius send $B_{k-1}(\mathsf{C}_{\ell,(\boldsymbol{\delta},\boldsymbol{m})})$ to pure weight classes on the core with lower degrees, which are extended to pure weight classes on the strata by taking $\otimes [\mathrm{pt}]$ in the Künneth formula.
  \item Contractions of degree one rational tails lead to non-zero boundary maps from the whole of $B_{k-1}(\mathsf{C}_{\ell,(\boldsymbol{\delta},\boldsymbol{m})}),$ so they do not survive to the $E^2$-page.
\end{enumerate}

\subsection{Contracted core}
The off-by-one classes potentially contributing to differentials that correspond to radial merges are the off-by-one classes on the core: they are in turn pulled back from the off-by-one classes on $\mathcal{M}_{0,n'},$ $\mathcal{M}_{1,n'},$ and torus bundle fibers. As earlier, the relations they contribute to are the WDVV and Getzler's relations, and the relations in the strata closures of $\mathcal{M}_{1,n}^{\mathrm{cen}},$ which follows from the blow-up formula of $\Mbar_{1,n}^{\mathrm{cen}}\to \Mbar_{1,n}$ \cite{rspw}. They contribute to relations on $\Mtil_{1,n}(\PP^r,d)$ through pullback from $\Mbar_{1,n'}^{\mathrm{cen}}$ to the relevant graph strata and pushing forward.

\subsection{Genus zero maps with factorisation property}
Apart from the core, the other source of off-by-one weight classes is genus zero maps that satisfy the factorisation property, which are denoted as $\mathcal{M}^{\mathbf{F}}_{(\boldsymbol{\delta}^{(r)}, \boldsymbol{m}^{(r)})}:$ the pair $(\boldsymbol{\delta}^{(r)}, \boldsymbol{m}^{(r)})$ indicates the degree and marking distributions on the contraction radius. By Lemma \ref{lem:MF}, up to pure weight classes, the off-by-one classes on its singular cohomology are given by \begin{enumerate}
  \item pullback from $H^1(\mathcal{M}_{0,m'})$,
  \item pullback from $H^{2r-1}(\widetilde{\mathrm{Map}}_{\boldsymbol{\delta}}^{\mathsf{F}}(\mathbb{P}^1, \mathbb{P}^r)),$ which we denote as the genus zero basepoint classes,
  \item the off-by-one classes coming from torus fibres.
\end{enumerate}

Among them, item (1) again corresponds to pullbacks of the WDVV relation. As with the discussion on the torus bundles in the proof of Lemma \ref{lem:MF}, item (3) corresponds to the relation $\psi_v = \psi_w$ for any univalent or bivalent vertices $v,w.$ This exhausts all the off-by-one classes from torus bundles, which hence do not survive to the $E^2$-page of the spectral sequence.

\begin{remark}
  Since the torus bundles parametrising non-vanishing linear dependencies are pulled back from strata of $\Mbar_{1,n}^{\mathrm{cen}}$ \cite{rspw}, the relation $\psi_v = \psi_w$ is pulled back from $\Mbar_{1,n}^{\mathrm{cen}}.$
\end{remark}

The behaviour of the basepoint classes is similar to the discussion in §\ref{subsec:ratailcon}: the rational tails are contracted as genus zero quasimaps with basepoints. To be more precise, let $[\mathbf{G}]$ and $[\mathbf{G}/e]$ be the following coarse equivalence classes of centrally aligned $(1,n,d)$-graphs.

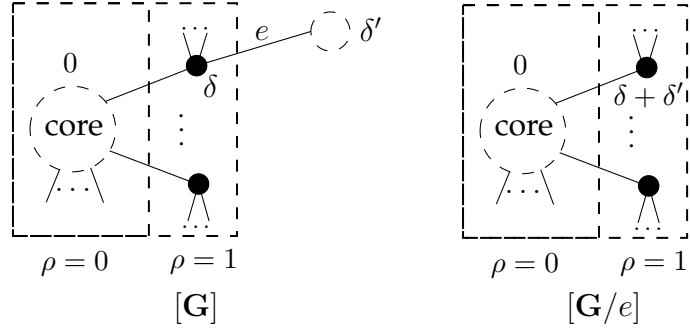
\[\begin{tikzpicture}[x=0.75pt,y=0.75pt,yscale=-1,xscale=1]

\draw  [dash pattern={on 4.5pt off 4.5pt}] (25.86,88.89) .. controls (25.86,76.98) and (35.52,67.32) .. (47.43,67.32) .. controls (59.34,67.32) and (69,76.98) .. (69,88.89) .. controls (69,100.8) and (59.34,110.46) .. (47.43,110.46) .. controls (35.52,110.46) and (25.86,100.8) .. (25.86,88.89) -- cycle ;
\draw    (64,74.89) -- (109.88,56.89) ;
\draw  [fill={rgb, 255:red, 0; green, 0; blue, 0 }  ,fill opacity=1 ] (104.72,56.89) .. controls (104.72,54.04) and (107.03,51.73) .. (109.88,51.73) .. controls (112.73,51.73) and (115.04,54.04) .. (115.04,56.89) .. controls (115.04,59.74) and (112.73,62.05) .. (109.88,62.05) .. controls (107.03,62.05) and (104.72,59.74) .. (104.72,56.89) -- cycle ;
\draw    (56.43,108.46) -- (62.94,126.06) ;
\draw    (40.73,109.53) -- (34.03,125.86) ;
\draw    (66,99.89) -- (109.75,116.52) ;
\draw  [fill={rgb, 255:red, 0; green, 0; blue, 0 }  ,fill opacity=1 ] (105.72,116.62) .. controls (105.72,113.78) and (108.03,111.47) .. (110.88,111.47) .. controls (113.73,111.47) and (116.04,113.78) .. (116.04,116.62) .. controls (116.04,119.47) and (113.73,121.78) .. (110.88,121.78) .. controls (108.03,121.78) and (105.72,119.47) .. (105.72,116.62) -- cycle ;
\draw  [dash pattern={on 4.5pt off 4.5pt}][line width=0.75]  (17.25,25.38) -- (85.81,25.38) -- (85.81,142.86) -- (17.25,142.86) -- cycle ;
\draw  [dash pattern={on 4.5pt off 4.5pt}][line width=0.75]  (17.25,25.38) -- (130.26,25.38) -- (130.26,142.86) -- (17.25,142.86) -- cycle ;
\draw    (110.88,116.62) -- (116.17,135.3) ;
\draw    (110.88,116.62) -- (105.5,135.3) ;
\draw    (109.18,59.86) -- (103.17,40.3) ;
\draw    (109.18,59.86) -- (115.17,40.3) ;
\draw    (109.88,56.89) -- (167.5,39.63) ;
\draw  [dash pattern={on 4.5pt off 4.5pt}] (167.5,39.63) .. controls (167.5,34.18) and (171.92,29.77) .. (177.36,29.77) .. controls (182.81,29.77) and (187.23,34.18) .. (187.23,39.63) .. controls (187.23,45.08) and (182.81,49.49) .. (177.36,49.49) .. controls (171.92,49.49) and (167.5,45.08) .. (167.5,39.63) -- cycle ;
\draw  [dash pattern={on 4.5pt off 4.5pt}] (252.86,89.84) .. controls (252.86,77.93) and (262.52,68.27) .. (274.43,68.27) .. controls (286.34,68.27) and (296,77.93) .. (296,89.84) .. controls (296,101.76) and (286.34,111.41) .. (274.43,111.41) .. controls (262.52,111.41) and (252.86,101.76) .. (252.86,89.84) -- cycle ;
\draw    (291,75.84) -- (336.88,57.84) ;
\draw  [fill={rgb, 255:red, 0; green, 0; blue, 0 }  ,fill opacity=1 ] (331.72,57.84) .. controls (331.72,54.99) and (334.03,52.68) .. (336.88,52.68) .. controls (339.73,52.68) and (342.04,54.99) .. (342.04,57.84) .. controls (342.04,60.69) and (339.73,63) .. (336.88,63) .. controls (334.03,63) and (331.72,60.69) .. (331.72,57.84) -- cycle ;
\draw    (283.43,109.41) -- (289.94,127.01) ;
\draw    (267.73,110.49) -- (261.03,126.82) ;
\draw    (293,100.84) -- (336.75,117.47) ;
\draw  [fill={rgb, 255:red, 0; green, 0; blue, 0 }  ,fill opacity=1 ] (332.72,117.58) .. controls (332.72,114.73) and (335.03,112.42) .. (337.88,112.42) .. controls (340.73,112.42) and (343.04,114.73) .. (343.04,117.58) .. controls (343.04,120.43) and (340.73,122.74) .. (337.88,122.74) .. controls (335.03,122.74) and (332.72,120.43) .. (332.72,117.58) -- cycle ;
\draw  [dash pattern={on 4.5pt off 4.5pt}][line width=0.75]  (244.25,26.34) -- (312.81,26.34) -- (312.81,143.81) -- (244.25,143.81) -- cycle ;
\draw  [dash pattern={on 4.5pt off 4.5pt}][line width=0.75]  (244.25,26.34) -- (357.26,26.34) -- (357.26,143.81) -- (244.25,143.81) -- cycle ;
\draw    (337.88,117.58) -- (343.17,136.25) ;
\draw    (337.88,117.58) -- (332.5,136.25) ;
\draw    (336.18,60.82) -- (330.17,41.25) ;
\draw    (336.18,60.82) -- (342.17,41.25) ;

\draw (41.33,49.56) node [anchor=north west][inner sep=0.75pt]  [font=\small]  {$0$};
\draw (97,170) node [anchor=north west][inner sep=0.75pt]    {$[\mathbf{G}]$};
\draw (37.03,117.26) node [anchor=north west][inner sep=0.75pt]    {$\dotsc $};
\draw (100.31,99.85) node [anchor=north west][inner sep=0.75pt]  [rotate=-269.27]  {$\dotsc $};
\draw (30,149.4) node [anchor=north west][inner sep=0.75pt]  [font=\small]  {$\rho =0$};
\draw (95.15,148.59) node [anchor=north west][inner sep=0.75pt]  [font=\small]  {$\rho =1$};
\draw (101.03,136.26) node [anchor=north west][inner sep=0.75pt]  [font=\scriptsize]  {$\dotsc $};
\draw (117.55,39.79) node [anchor=north west][inner sep=0.75pt]  [font=\scriptsize,rotate=-179.95]  {$\dotsc $};
\draw (191,31.4) node [anchor=north west][inner sep=0.75pt]    {$\delta '$};
\draw (111.88,60.29) node [anchor=north west][inner sep=0.75pt]    {$\delta $};

\draw (138,36) node [anchor=north west][inner sep=0.75pt]    {$e$};
\draw (268.33,50.51) node [anchor=north west][inner sep=0.75pt]  [font=\small]  {$0$};
\draw (295,170) node [anchor=north west][inner sep=0.75pt]    {$[\mathbf{G} /e]$};
\draw (264.03,118.22) node [anchor=north west][inner sep=0.75pt]    {$\dotsc $};
\draw (327.31,100.8) node [anchor=north west][inner sep=0.75pt]  [rotate=-269.27]  {$\dotsc $};
\draw (257,150.35) node [anchor=north west][inner sep=0.75pt]  [font=\small]  {$\rho =0$};
\draw (322.15,149.54) node [anchor=north west][inner sep=0.75pt]  [font=\small]  {$\rho =1$};
\draw (329.03,137.22) node [anchor=north west][inner sep=0.75pt]  [font=\scriptsize]  {$\dotsc $};
\draw (344.55,40.74) node [anchor=north west][inner sep=0.75pt]  [font=\scriptsize,rotate=-179.95]  {$\dotsc $};
\draw (318.88,64.24) node [anchor=north west][inner sep=0.75pt]  [font=\small]  {$\delta +\delta '$};
\draw (32,83) node [anchor=north west][inner sep=0.75pt]   [align=left] {core};
\draw (259,83) node [anchor=north west][inner sep=0.75pt]   [align=left] {core};

\end{tikzpicture}
\]
By Künneth formula and compatibility with flat pullback - analogous to Lemmas \ref{lem:pbss} and \ref{lem:kunnss} - the boundary map of this edge contraction determines the general case of contracting rational tails to the contraction radius. Let $(\boldsymbol{\delta}^{(r)}, \boldsymbol{m}^{(r)})$ be the degree and marking distributions on the contraction radius of $[\mathbf{G}],$ and let $(\boldsymbol{\delta'}^{(r)}, \boldsymbol{m}^{(r)})$ be the ones on $[\mathbf{G}/e],$ so that $\boldsymbol{\delta'}^{(r)} = \boldsymbol{\delta}^{(r)} + \delta'\cdot \mathbf{e}_{v}$ for the vertex $v$ on $[\mathbf{G}]$ and $[\mathbf{G}/e].$

Let $\mathbf{G}^{{cc}}$ be the contraction core of $[\mathbf{G}],$ then $H_{\star}^{\mathrm{BM}}(\Mtil_{[\mathbf{G}]}) = H_{\star}^{\mathrm{BM}}(\Mtil_{[\mathbf{G}^{{cc}}]})\otimes H_{\star}^{\mathrm{BM}}(\Mbar_{0,0}^*(\PP^r,\delta')).$

\begin{lemma}\label{lem:g0baspt}
  When $\delta' = 1,$ the boundary map $$\gr^W_{-(k-1)}H_{k}^{\mathrm{BM}}(\Mtil_{[\mathbf{G}/e]})\to \gr^W_{-(k-1)}H_{k-1}^{\mathrm{BM}}(\Mtil_{[\mathbf{G}]})$$ sends $\beta_v\cap [\mathcal{M}^{\mathbf{F}}_{(\boldsymbol{\delta}^{(r)}, \boldsymbol{m}^{(r)})}]$ to $[\Mtil_{[\mathbf{G}^{cc}]}]\otimes [\mathrm{pt}].$ The boundary map vanishes on $\beta_v\cap H_{k}^{\mathrm{BM}}(\Mtil_{[\mathbf{G}/e]})\subset \gr^W_{-(k-1)}H_{k}^{\mathrm{BM}}(\Mtil_{[\mathbf{G}/e]})$ otherwise.
\end{lemma}
\begin{proof}
  The boundary image in the case $\delta'=1$ follows from collapsing the genus zero stable map as a basepoint of the genus one map on vertex $v,$ similar to the genus one situation (§\ref{subsec:ratailcon}). (Co)homological degree contraints force the map to vanish in the other cases. 
\end{proof}

\begin{corollary}\label{cor:g0basptE2}
  The (Poincaré duals of) off-by-one classes $\beta_v$ do not survive to the $E^2$-page.
\end{corollary}

\subsection{Basepoint relations}\label{subsec:baseptrel}
It suffices to describe the images of the classes in the case of a single rational tail, since taking their Künneth tensor products with classes supported on the rational tails linearly span all the basepoint relations.
\begin{definition}
  We define basepoint relations as the boundary map images of: 
  \begin{enumerate}
    \item $$B_{k-1}(n,r,d)\to \bigoplus_{\substack{[\mathbf{G},\rho]: \mathbf{G}\text{ loop-free} \\\text{codim}\Mtil_{[\mathbf{G},\rho]}=1}} \mathrm{gr}^W_{-(k-1)} H_{k-1}^{\mathrm{BM}}(\Mtil_{[\mathbf{G},\rho]}),$$
    \item $$B_{k-1}(\mathsf{C}_{\ell,(\boldsymbol{\delta},\boldsymbol{m})})\to \bigoplus_{\substack{\mathsf{C}_{\ell,(\boldsymbol{\delta},\boldsymbol{m})}\rightsquigarrow[\mathbf{G}',\rho']\\ \text{rational tail contraction}}} \mathrm{gr}^W_{-(k-1)} H_{k-1}^{\mathrm{BM}}(\Mtil_{[\mathbf{G}',\rho']}),$$
    \item for all $[\mathbf{G},\rho]$ with $|\rho| = 1$ and no rational tails,$$\gr^W_{-(\star-1)} H_{\star}(\Mtil_{[\mathbf{G},\rho]})\supset \langle \beta_{v} \cap [\Mtil_{[\mathbf{G},\rho]}] \rangle\to \bigoplus_{\substack{[\mathbf{G},\rho]\rightsquigarrow[\mathbf{G}',\rho']\\ \text{rational tail contraction}}} \gr^W_{\star}H^{\mathrm{BM}}_{\star}(\Mtil_{[\mathbf{G}',\rho']}).$$ The homological degree on the target is $2(\dim \Mtil_{[\mathbf{G},\rho]}-(r-1)).$
  \end{enumerate}

  Let $\boldsymbol{\mathsf{B}}\subset \bigoplus_{[\mathbf{G},\rho]}\gr^W_{-\star}H_{\star}^{\mathrm{BM}}(\Mtil_{[\mathbf{G},\rho]})$ be the vector subspace generated by taking the cap products of the basepoint relations and all pure weight cohomology classes on the strata.
\end{definition}

Combining Lemmas \ref{lem:cobratail}, \ref{lem:pbss}, and \ref{lem:cobradmer} we give explicit formula of the basepoint relations coming from $B_{\star}(n,r,d)$ via the surjection $$\gr^W_{-\star}H_{\star}^{\mathrm{BM}}\hat{\mathcal{Q}}_{1,n+1}(\PP^r,d-1)\twoheadrightarrow B_{k-1}(r,n,d).$$ From Lemma \ref{lem:pics} and projective bundle formula, a basis of pure weight classes on $\hat{\mathcal{Q}}_{1,n+1}(\PP^r,d-1)$ is given by $$\left\{\left((\alpha_1 + \Theta\cdot \alpha_2)\cdot H_{\hat{\mathcal{Q}}}^m\right)\cap [\hat{\mathcal{Q}}_{1,n+1}(\PP^r,d-1)]\right\},$$ where \begin{enumerate}
  \item $\alpha_1,\alpha_2$ runs through some bases of $H^\star(\mathcal{M}_{1,n+2})$ and $H^\star(\mathcal{M}_{1,n+1})$ respectively,
  \item $0\leq m\leq (d-1)(r+1)-1.$
\end{enumerate}
Given such a basis element, consider all subsets $I\subset [n]$ such that the pair $(\alpha_1,\alpha_2)$ is pulled back from some $(\alpha_1^{(I)}, \alpha_2^{(I)})\in H^\star(\mathcal{M}_{I+2})\times H^\star(\mathcal{M}_{I+1}):$ for each $I,$ such $(\alpha_1^{(I)}, \alpha_2^{(I)})$ is unique if it exists. The basepoint relation is the collection of pure weight classes given by:
\begin{enumerate}
  \item $[(\alpha_1^{(I)}+\Theta \cdot \alpha_2^{(I)})\cdot H_{\hat{\mathcal{Q}}}^{m-\delta(r+1)}]\otimes [\mathrm{pt}]$ on $\Mtil_{[\mathbf{G}_{\delta, I}]}$ (defined before Lemma \ref{lem:pbss}), and the term is understood to vanish if $m-\delta(r+1)<0,$
  \item $[(\alpha_1^{(I)}+\Theta \cdot \alpha_2^{(I)})\cdot H_{\hat{\mathcal{Q}}}^{m-\delta(r+1)}]\otimes [\mathrm{pt}]$ on $\Mtil_{[\mathbf{G}_{\boldsymbol{\delta},\boldsymbol{m}},\rho]}$ when the contraction radius has a single vertex, and the markings on the core are the subset $I,$
  \item zero on all the other codimension one strata.
\end{enumerate}

\begin{remark}
  The basepoint relations are among classes of the form $\varphi\otimes [\mathrm{pt}]$ where $\varphi$ is a pure weight class on the mapping space from an elliptic curve and the rational tail or the contraction radius is decorated by the class $[\mathrm{pt}].$ In other words, the basepoint relations do not concern algebraic cycles on genus zero stable maps.
\end{remark}

Analogous formulas can be obtained for relations from $B_{k-1}(\mathsf{C}_{\ell,(\boldsymbol{\delta},\boldsymbol{m})}),$ and Lemma \ref{lem:g0baspt} already determines the relation induced by the classes $\beta_v.$

\subsection{Presentation of the homology group}\label{subsec:pres}

\begin{definition}
  Let $\boldsymbol{\mathsf{P}}\subset \bigoplus_{[\mathbf{G},\rho]}\gr^W_{-\star}H_{\star}^{\mathrm{BM}}(\Mtil_{[\mathbf{G},\rho]})$ be the boundary images of the cap product between any pure weight class and one of the following off-by-one cohomology classes:
  \begin{enumerate}
    \item pullback from\footnote{Strictly speaking, these classes will contribute to non-trivial differentials only at the $E^2$-page, so it only make sense to talk images of these classes under $E^2$-differential.} $H^3(\mathcal{M}_{1,4}),$
    \item pullback from $H^1(\mathcal{M}_{0,4}),$
    \item the off-by-one classes coming from torus bundles.
  \end{enumerate}
\end{definition}
Again, the tuples of pure weight classes in $\mathbf{P}$ correspond to pullback of relations from $\Mbar_{1,n}^{\mathrm{cen}}$ corresponding to Getzler's relation, WDVV relation, and $\psi_v = \psi_w$ for two vertices $v,w$ on the same level.

\begin{theorem}\label{thm:relations}
  We have $H_{\star}(\Mtil_{1,n}(\PP^r,d)) = \bigoplus_{[\mathbf{G},\rho]}\gr^W_{-\star}H^{\mathrm{BM}}_{\star}(\Mtil_{[\mathbf{G},\rho]})/\langle \boldsymbol{\mathsf{B}}, \boldsymbol{\mathsf{P}} \rangle.$
\end{theorem}
\begin{proof}
  The generators of the off-by-one classes $\bigoplus_{[\mathbf{G},\rho]}\mathrm{gr}^W_{-(\star-1)}H_{\star}^{\mathrm{BM}}(\Mtil_{[\mathbf{G},\rho]})$ are known to be pullback from the moduli of curves, torus fibres, and basepoint classes in genus zero and one (including cycles of rational curves): see Lemmas \ref{lem:m1nrdoff}, \ref{lem:MF}, and \ref{cor:Mcycle}. The first two types of off-by-one classes are covered in $\boldsymbol{\mathsf{P}}$ and discussion at the beginning of \ref{subsec:basptrelcomp} implies that they will not survive to the $E^3$-page. Lemmas \ref{lem:g1baseptE2}, \ref{cor:g0basptE2} and discussion in §\ref{subsec:baseptloops} state that the basepoint classes all have non-zero images in the $E^1$-page differential--recorded by $\boldsymbol{\mathsf{B}}$--and do not survive to the $E^2$-page. 
  
  Therefore, after $E^3$-page, there is no off-by-one class that produce non-zero images in the spectral sequence differentials, which implies that we have exhausted all the possible relations among the pure weight classes in the $E^1$-page.
\end{proof}

\begin{remark}
It is more convenient to work with classes on $H^\star(\Mtil_{1,n}(\PP^r,d))$ and their relations. The relations in $\boldsymbol{\mathsf{P}}$ are known to be pulled back from relations on $\Mbar_{1,n}^{\mathrm{cen}}.$ 

On the other hand, let $\beta\in \boldsymbol{\mathsf{B}}\cap \bigoplus_{\dim \Mtil_{[\mathbf{G},\rho]=p}}\gr^W_{-\star}H_{\star}^{\mathrm{BM}}(\Mtil_{[\mathbf{G},\rho]})$ be a basepoint relation among $p$-dimensional strata, and let $\tilde{\beta}\in \bigoplus_{[\mathbf{G},\rho]}H_{\star}(\Mbar_{[\mathbf{G},\rho]})$ be a lift, such as the ones introduced in §\ref{subsec:strcl}. Then the pushforward of $\tilde{\beta}$ in $\Mtil_{1,n}(\PP^r,d)$ must lie in the pushforward image of $(p-1)$-dimensional strata. Setting up an induction on the dimension of $\Mtil_{1,n}(\PP^r,d)$ (varying $(n,r,d)$), the generators and relations on the $(p-1)$-dimensional strata are known by inductive hypothesis and Künneth formula. In short, we may iteratively produce a set of relations among the lifted generators in $H_\star(\Mtil_{1,n}(\PP^r,d))$ that is in bijection with the span of $\boldsymbol{\mathsf{B}}$ and $\boldsymbol{\mathsf{P}}.$ This leads to the statement of Theorem \ref{thm:rels}.
\end{remark}

\bibliographystyle{alpha}
\bibliography{g1maps.bib}

\begin{thebibliography}{CLPW24b}

\bibitem[AC98]{ac}
Enrico Arbarello and Maurizio Cornalba.
\newblock Calculating cohomology groups of moduli spaces of curves via algebraic geometry.
\newblock {\em Publications Math\'ematiques de l'IH\'ES}, 88:97--127, 1998.

\bibitem[Ara05]{ara}
Donu Arapura.
\newblock The {Leray} spectral sequence is motivic.
\newblock {\em Inventiones mathematicae}, 160(3):567--589, June 2005.
\newblock arXiv:math/0301140.

\bibitem[BC23]{battistellacarocci}
Luca Battistella and Francesca Carocci.
\newblock A smooth compactification of the space of genus two curves in projective space: via logarithmic geometry and {Gorenstein} curves.
\newblock {\em Geometry \& Topology}, 27(3):1203--1272, June 2023.

\bibitem[BFP24]{bfp}
J.~Bergström, C.~Faber, and S.~Payne.
\newblock Polynomial point counts and odd cohomology vanishing on moduli spaces of stable curves.
\newblock {\em Annals of Mathematics}, 199(3), May 2024.

\bibitem[BM96]{bm}
K.~Behrend and Yu. Manin.
\newblock {Stacks of stable maps and Gromov-Witten invariants}.
\newblock {\em Duke Mathematical Journal}, 85(1):1 -- 60, 1996.

\bibitem[BNR21]{bnr}
Luca Battistella, Navid Nabijou, and Dhruv Ranganathan.
\newblock Curve counting in genus one: {E}lliptic singularities and relative geometry.
\newblock {\em Algebraic Geometry}, pages 637--679, 11 2021.

\bibitem[CL24]{clchow}
S.~Canning and H.~Larson.
\newblock On the {Chow} and cohomology rings of moduli spaces of stable curves.
\newblock {\em Journal of the European Mathematical Society}, November 2024.

\bibitem[CLP23]{clp11}
Samir Canning, Hannah Larson, and Sam Payne.
\newblock The eleventh cohomology group of {$\overline{\mathcal {M}}_{g,n}$}.
\newblock {\em Forum of Mathematics, Sigma}, 11:e62, 2023.

\bibitem[CLP24]{clp}
Samir Canning, Hannah Larson, and Sam Payne.
\newblock Extensions of tautological rings and motivic structures in the cohomology of {${\overline {\mathcal {M}}}_{g,n}$}.
\newblock {\em Forum of Mathematics, Pi}, 12:e23, 2024.

\bibitem[CLPW24a]{clpw}
Samir {Canning}, Hannah {Larson}, Sam {Payne}, and Thomas {Willwacher}.
\newblock {Moduli spaces of curves with polynomial point counts}.
\newblock {\em arXiv e-prints}, page arXiv:2410.19913, October 2024.

\bibitem[CLPW24b]{clpw2}
Samir {Canning}, Hannah {Larson}, Sam {Payne}, and Thomas {Willwacher}.
\newblock {The motivic structures $\mathsf{LS}_{12}$ and $\mathsf{S}_{16}$ in the cohomology of moduli spaces of curves}.
\newblock {\em arXiv e-prints}, page arXiv:2411.12652, November 2024.

\bibitem[CLPW25]{clpw25}
Samir {Canning}, Hannah {Larson}, Sam {Payne}, and Thomas {Willwacher}.
\newblock {FA-modules of holomorphic forms on $\overline{\mathcal{M}}_{g,n}$}.
\newblock {\em arXiv e-prints}, page arXiv:2509.08774, September 2025.

\bibitem[Coo14]{Cooper2014}
Yaim Cooper.
\newblock The geometry of stable quotients in genus one.
\newblock {\em Mathematische Annalen}, 361(3-4):943--979, September 2014.

\bibitem[DA73]{SGA43}
P.~Deligne and M.~Artin.
\newblock {\em Théorie des {T}opos et {C}ohomologie {É}tale des {S}chémas}.
\newblock Springer Berlin Heidelberg, 1973.

\bibitem[Fon07]{Fontanari}
Claudio Fontanari.
\newblock Towards the cohomology of moduli spaces of higher genus stable maps.
\newblock {\em Archiv der Mathematik}, 89(6):530--535, Oct 2007.

\bibitem[FW16]{farbwolf}
Benson Farb and Jesse Wolfson.
\newblock Topology and arithmetic of resultants. {I}.
\newblock {\em New York J. Math.}, 22:801--821, 2016.

\bibitem[Get99]{getzresolv}
E.~Getzler.
\newblock {Resolving mixed Hodge modules on configuration spaces}.
\newblock {\em Duke Mathematical Journal}, 96(1):175 -- 203, 1999.

\bibitem[GP06]{getzpand}
Ezra Getzler and Rahul Pandharipande.
\newblock The {Betti} numbers of {$\overline{\mathcal{M}}_{0,n}(r,d)$}.
\newblock {\em Journal of Algebraic Geometry}, 15(4):709--732, May 2006.

\bibitem[KS24a]{vzdualcomplex}
Siddarth {Kannan} and Terry~Dekun {Song}.
\newblock {The dual complex of $\mathcal{M}_{1,n}(\mathbb{P}^r,d)$ via the geometry of the {V}akil--{Z}inger moduli space}.
\newblock {\em arXiv e-prints}, page arXiv:2411.03518, 2024.

\bibitem[KS24b]{genusonechar}
Siddarth {Kannan} and Terry~Dekun {Song}.
\newblock {The {$S_n$}-equivariant {E}uler characteristic of {$\overline{\mathcal{M}}_{1, n}(\mathbb{P}^r, d)$}}.
\newblock {\em arXiv e-prints}, page arXiv:2412.12317, 2024.

\bibitem[KS26]{virtualhodge}
Siddarth {Kannan} and Terry~Dekun {Song}.
\newblock {Virtual Hodge numbers of $\mathcal{M}_{g, n}(\mathbb{P}^r, d)$: stability and calculations}.
\newblock {\em arXiv e-prints}, page arXiv:2601.07981, January 2026.

\bibitem[MM07]{mminter}
Andrei Musta{\c{t}}ǎ and Magdalena~Anca Musta{\c{t}}ǎ.
\newblock Intermediate moduli spaces of stable maps.
\newblock {\em Inventiones mathematicae}, 167(1):47--90, 2007.

\bibitem[MM08a]{mmchow}
Anca~M Mustaţă and Andrei Mustaţă.
\newblock The {C}how ring of {$\overline{M}_{0,m}(n, d)$}.
\newblock {\em Journal f\"{u}r die reine und angewandte Mathematik (Crelles Journal)}, 2008(615), January 2008.

\bibitem[MM08b]{mmtaut}
Anca~M. Mustaţǎ and Andrei Mustaţǎ.
\newblock Tautological rings of stable map spaces.
\newblock {\em Advances in Mathematics}, 217(4):1728--1755, 2008.

\bibitem[MOP11]{mop}
Alina Marian, Dragos Oprea, and Rahul Pandharipande.
\newblock The moduli space of stable quotients.
\newblock {\em Geometry \& Topology}, 15(3):1651--1706, October 2011.

\bibitem[Opr06a]{opretorus}
Dragos Oprea.
\newblock Tautological classes on the moduli spaces of stable maps to {$\mathbb{P}^r$} via torus actions.
\newblock {\em Advances in Mathematics}, 207(2):661--690, December 2006.

\bibitem[Opr06b]{opreaflag}
Dragos Oprea.
\newblock The tautological rings of the moduli spaces of stable maps to flag varieties.
\newblock {\em Journal of Algebraic Geometry}, 15(4):623--655, June 2006.

\bibitem[Pan99]{Pandharipande1999}
Rahul Pandharipande.
\newblock Intersections of {$\mathbb{Q}$}-divisors on {K}ontsevich’s moduli space {$\overline{M}_{0,n}(\mathbb{P}^r,d)$} and enumerative geometry.
\newblock {\em Transactions of the American Mathematical Society}, 351(4):1481–1505, 1999.

\bibitem[Pet14]{pet}
Dan Petersen.
\newblock The structure of the tautological ring in genus one.
\newblock {\em Duke Mathematical Journal}, 163(4), March 2014.
\newblock arXiv:1205.1586 [math].

\bibitem[Pet15]{PetA2}
Dan Petersen.
\newblock Cohomology of local systems on the moduli of principally polarized abelian surfaces.
\newblock {\em Pacific Journal of Mathematics}, 275(1):39--61, April 2015.

\bibitem[Pet16]{PetM2ct}
Dan Petersen.
\newblock Tautological rings of spaces of pointed genus two curves of compact type.
\newblock {\em Compositio Mathematica}, 152(7):1398--1420, 2016.

\bibitem[RSPW19]{rspw}
Dhruv Ranganathan, Keli Santos-Parker, and Jonathan Wise.
\newblock Moduli of stable maps in genus one and logarithmic geometry, {I}.
\newblock {\em Geometry \& Topology}, 23(7):3315--3366, December 2019.

\bibitem[Son26]{relations}
Terry~Dekun Song.
\newblock Revisiting relations on genus zero stable maps to projective space, 2026+.
\newblock In preparation.

\bibitem[VZ08]{vakilzinger}
Ravi Vakil and Aleksey Zinger.
\newblock A desingularization of the main component of the moduli space of genus-one stable maps into {$\mathbb{P}^n$}.
\newblock {\em Geometry \& Topology}, 12(1):1--95, February 2008.

\end{thebibliography}
\,

\end{document}